\documentclass[a4paper,notitlepage,twoside,reqno,11pt]{amsart}

\usepackage{anysize}
\marginsize{3.4cm}{3.4cm}{3cm}{3cm}

\usepackage{bbm,pifont,latexsym}

\usepackage{dcolumn,indentfirst}
\usepackage[hypertexnames=false,bookmarksopen=true,linktocpage=true,pdfstartview={XYZ null null 1.25}]{hyperref}
\usepackage{amsmath,amssymb,amscd,amsthm,amsfonts,mathrsfs,stmaryrd}
\usepackage{color,graphicx,xcolor,graphics,subfigure,extarrows,caption2}
\usepackage{pdfpages,titletoc}

\usepackage{autonum}

\newtheorem{thm}{Theorem}[section]
\newtheorem{lem}[thm]{Lemma}
\newtheorem{cor}[thm]{Corollary}
\newtheorem{prop}[thm]{Proposition}

\newtheorem{thmx}{Theorem}
\theoremstyle{definition}
\newtheorem*{defi}{Definition}
\newtheorem*{rmk}{Remark}

\newtheorem*{ques}{Question}

\newcommand{\EC}{\widehat{\mathbb{C}}}

\newcommand{\C}{\mathbb{C}}
\newcommand{\D}{\mathbb{D}}

\newcommand{\N}{\mathbb{N}}
\newcommand{\Q}{\mathbb{Q}}
\newcommand{\R}{\mathbb{R}}
\newcommand{\T}{\mathbb{T}}

\newcommand{\Z}{\mathbb{Z}}

\newcommand{\MC}{\mathcal{C}}

\newcommand{\MO}{\mathcal{O}}
\newcommand{\MP}{\mathcal{P}}

\newcommand{\MU}{\mathcal{U}}

\newcommand{\ii}{\textup{i}}

\newcommand{\re}{\textup{Re}\,}
\newcommand{\im}{\textup{Im}\,}

\newcommand{\Crit}{\textup{Crit}}

\newcommand{\diam}{\textup{diam}}
\newcommand{\dist}{\textup{dist}}
\newcommand{\Mod}{\textup{mod}\,}
\newcommand{\Int}{\textup{int}}

\newcommand{\Par}{\textup{Par}}
\newcommand{\Sie}{\textup{Sie}}

\newcommand{\mv}{\mathbf{v}}
\newcommand{\mm}{\mathbf{m}}

\makeatletter\@addtoreset{equation}{section}\makeatother

\begin{document}

\author[Yuming Fu]{Yuming Fu}
\address{College of Mathematics and Statistics, Shenzhen University, Shenzhen 518061, P. R. China}
\email{yumingfuxy@szu.edu.cn}

\author[Fei Yang]{Fei Yang}
\address{Department of Mathematics, Nanjing University, Nanjing 210093, P. R. China}
\email{yangfei@nju.edu.cn}

%---------------------------------------------------------------------------------------------------------------
\title[Mating Siegel and parabolic polynomials]{Mating Siegel and parabolic quadratic polynomials}

\begin{abstract}
Let $f_\theta(z)=e^{2\pi\ii\theta}z+z^2$ be the quadratic polynomial having an indifferent fixed point at the origin. For any bounded type irrational number $\theta\in\R\setminus\Q$ and any rational number $\nu\in\Q$, we prove that $f_\theta$ and $f_\nu$ are conformally mateable, and that the mating is unique up to conjugacy by a M\"{o}bius map. This gives an affirmative (partial) answer to a question raised by Milnor in 2004.

A crucial ingredient in the proof relies on an expansive property when iterating certain rational maps near Siegel disk boundaries. Combining this with the expanding property in repelling petals of parabolic points, we also prove that the Julia sets of a class of Siegel rational maps with parabolic points are locally connected.
\end{abstract}

% AMS subject classifications (used in AMS journals)

% AMS keywords (used in AMS journals)

% today's date, or fill in whatever date you prefer
\date{\today}

% acknowledge support, etc
% \thanks{This research was partially supported by NSF grant DOA-123456789.}
% \thanks{We would like to thank our colleagues for their helpful criticism.}

% dedication
% \dedicatory{Dedicated to Professor Donald Knuth on the occasion of his $100$th birthday}

\maketitle

%----------------------------------------------------------------------------------------------------------------
%\vskip1.0cm
%{\setcounter{tocdepth}{1}
%\tableofcontents
%}
%\tableofcontents
%----------------------------------------------------------------------------------------------------------------

\section{Introduction}\label{introduction}

\subsection{Mating and definitions}\label{subsec:mating-defi}

In the 1980s, Douady and Hubbard observed that the dynamics of some rational maps can be well decomposed into that of a pair of polynomials.
In order to partially parametrize quadratic rational maps by a pair of quadratic polynomials, they introduced a topological construction which is called \emph{mating} \cite{Dou83}.
There are several types of mating. We first recall the definitions of some of them. One may refer to \cite{PM12} for a nice illustration on this topic.

Let $f_1$ and $f_2$ be two monic quadratic polynomials whose \textit{filled Julia sets} $K(f_i)$ are locally connected, where $i=1,2$. The B\"{o}ttcher map $\Phi_i$ is a conformal isomorphism with $\Phi_i(\infty)=\infty$ and $\Phi_i'(\infty)=1$ which maps the \textit{basin of infinity} $\EC\setminus K(f_i)$ onto the outside of the closed unit disk $\EC\setminus \overline{\D}$ such that
\begin{equation}
\Phi_i(f_i(z))=(\Phi_i(z))^2 \text{\quad for all }z\in\EC\setminus K(f_i).
\end{equation}
By Carath\'{e}odory's theorem, the inverse map $\Phi_i^{-1}$ has a continuous extension to their corresponding boundaries:
$\Phi_i^{-1}: \partial \D\to J(f_i)$, where $J(f_i)=\partial K(f_i)$ is the \textit{Julia set} of $f_i$. The continuous parametrization
\begin{equation}
\eta_i(t)=\Phi_i^{-1}(e^{2\pi \ii t}):\T=\R/\Z \to J(f_i)
\end{equation}
is known as the \textit{Carath\'{e}odory loop} of $J(f_i)$.
We obtain a new topological space
\begin{equation}
X=(K(f_1)\sqcup K(f_2))/(\eta_1(t)\sim \eta_2(-t))
\end{equation}
by gluing the filled Julia sets $K(f_1)$ and $K(f_2)$ in reverse directions,
where the equivalence relation $\sim$ is referred as \textit{ray equivalence}, and denoted by $\sim_r$. If $X$ is homeomorphic to the $2$-sphere $S^2$, then the pair of polynomials $(f_1,f_2)$ is callled \emph{topologically mateable}. The induced map
\begin{equation}
f_1 \sqcup_{\mathcal{T}} f_2= (f_1|_{K(f_1)})\sqcup (f_2|_{K(f_2)}) /(\eta_1(t)\sim \eta_2(-t))
\end{equation}
is the \emph{topological mating} of $f_1$ and $f_2$.

If, in addition, there exist a quadratic rational map $F:\EC\to\EC$ and a homeomorphism $\Phi:S^2\to \EC$ such that $\Phi$ conjugates $f_1 \sqcup_{\mathcal{T}} f_2$ to $F$ and that $\Phi$ is conformal in the interiors of $K(f_1)$ and $K(f_2)$ in case there is an interior, then $f_1$ and $f_2$ are said to be \textit{conformally mateable}.
The quadratic rational map $F$ is called \textit{a conformal mating} of $f_1$ and $f_2$.
If such $F$ is unique up to conjugacy by a M\"{o}bius map, we refer to it as \textit{the conformal mating} of $f_1$ and $f_2$.
When proving theorems in practice, the following equivalent definition is more convenient for application (see \cite{YZ01} and \cite{PM12}).

\begin{defi}[{conformal mating}]
Let $f_1$ and $f_2$ be quadratic polynomials with locally connected Julia sets. A quadratic rational map $F:\EC\to\EC$ is called a \textit{conformal mating} $f_1\sqcup f_2$ of $f_1$ and $f_2$ if there exist continuous semiconjugacies
\begin{equation}
\phi_i:K(f_i)\to \EC \text{\quad with \quad} \phi_i\circ f_i=F\circ \phi_i,
\end{equation}
conformal in the interiors of the filled Julia sets in case there is an interior, such that $\phi_1(K(f_1))\cup \phi_2(K(f_2))=\EC$ and for $i,j=1,2$, $\phi_i(z)=\phi_j(w)$ if and only if $z\sim_r w $.
\end{defi}

\subsection{History and results}

If two quadratic polynomials $f_1$ and $f_2$ are in conjugate limbs of the Mandelbrot set with locally connected Julia sets, then even the topological mating of $f_1$ and $f_2$ is impossible since $(K(f_1)\sqcup K(f_2))/\sim_r$ is not homeomorphic to the $2$-sphere.
By using Thurston's topological characterization of critically finite rational maps \cite{DH93}, Tan Lei, Rees and Shishikura proved that any two pair of subhyperbolic quadratic polynomials are conformally mateable if they are not in conjugate limbs of the Mandelbrot set (\cite{Tan92}, \cite{Ree86c}, \cite{Ree92}, \cite{Shi00b}).
In particular, when both of the quadratic polynomials have a strictly preperiodic critical point, then their filled Julia sets have empty interiors and the Julia set after mating is the whole Riemann sphere. See \cite{Mil04} for an interesting illustration. The mateable results on subhyperbolic quadratics has been extended to geometrically finite case by parabolic surgery \cite{HT04}.

The first example of mating quadratic polynomials without using Thurston's theorem was given in \cite{Luo95}. The approach involved consists of locating a candidate rational map for the mating, and then using Yoccoz puzzle partitions and complex bounds of Yoccoz to prove that this candidate rational map is a mating.
Recently, one of the representatives in this direction is the mating of $z\mapsto z^2-1$ (the corresponding Julia set is known as the \textit{basilica}) with some other quadratic polynomials (See \cite{Tim08}, \cite{AY09}, \cite{Dud11}). In the proofs, the parameter space of the family of quadratic rational maps consisting of a super-attracting $2$-cycle was studied, and the candidate of mating was found in the parameter space by using parapuzzles.

\medskip
For the quadratic polynomial having an indifferent fixed point at the origin:
\begin{equation}
f_\theta(z):=e^{2\pi\ii\theta}z+z^2, \text{\quad where } \theta\in\R,
\end{equation}
Milnor raised the following question in \cite[p.\,61]{Mil04}:

\begin{ques}[{Milnor}]
Suppose the Julia sets of $f_\theta$ and $f_\nu$ are locally connected with $e^{2\pi\ii(\theta+\nu)}\neq 1$.
Does the conformal mating $f_{\theta}\sqcup f_{\nu}$ exist?
\end{ques}

This question has been answered in the following cases\footnote{Zhang has answered Milnor's question affirmatively when $\theta$ and $\nu$ are typical irrational numbers.
See \href{https://www.math.univ-toulouse.fr/~cheritat/Exposes/Banff.pdf}{https://www.math.univ-toulouse.fr/$\sim$cheritat/Exposes/Banff.pdf}, 2011.}:
\begin{itemize}
\item Both $\theta$ and $\nu$ are bounded type irrational numbers \cite{YZ01};
\item Both $\theta$ and $\nu$ are rational numbers \cite{HT04}.
\end{itemize}
Note that if $\theta$ is of bounded type, then $f_\theta$ has a Siegel disk centered at the origin \cite{Sie42}, and the Julia set of $f_\theta$ is locally connected \cite{Pet96} (see also \cite{Yam99}). If $\nu$ is a rational number, then $f_\nu$ has a parabolic fixed point at the origin, and the Julia set of $f_\nu$ is also locally connected \cite[Chap.\,10]{DH85a}. See Figure \ref{Fig_Julias}.

\begin{figure}[!htpb]
 \setlength{\unitlength}{1mm}
  \centering
  \includegraphics[width=0.85\textwidth]{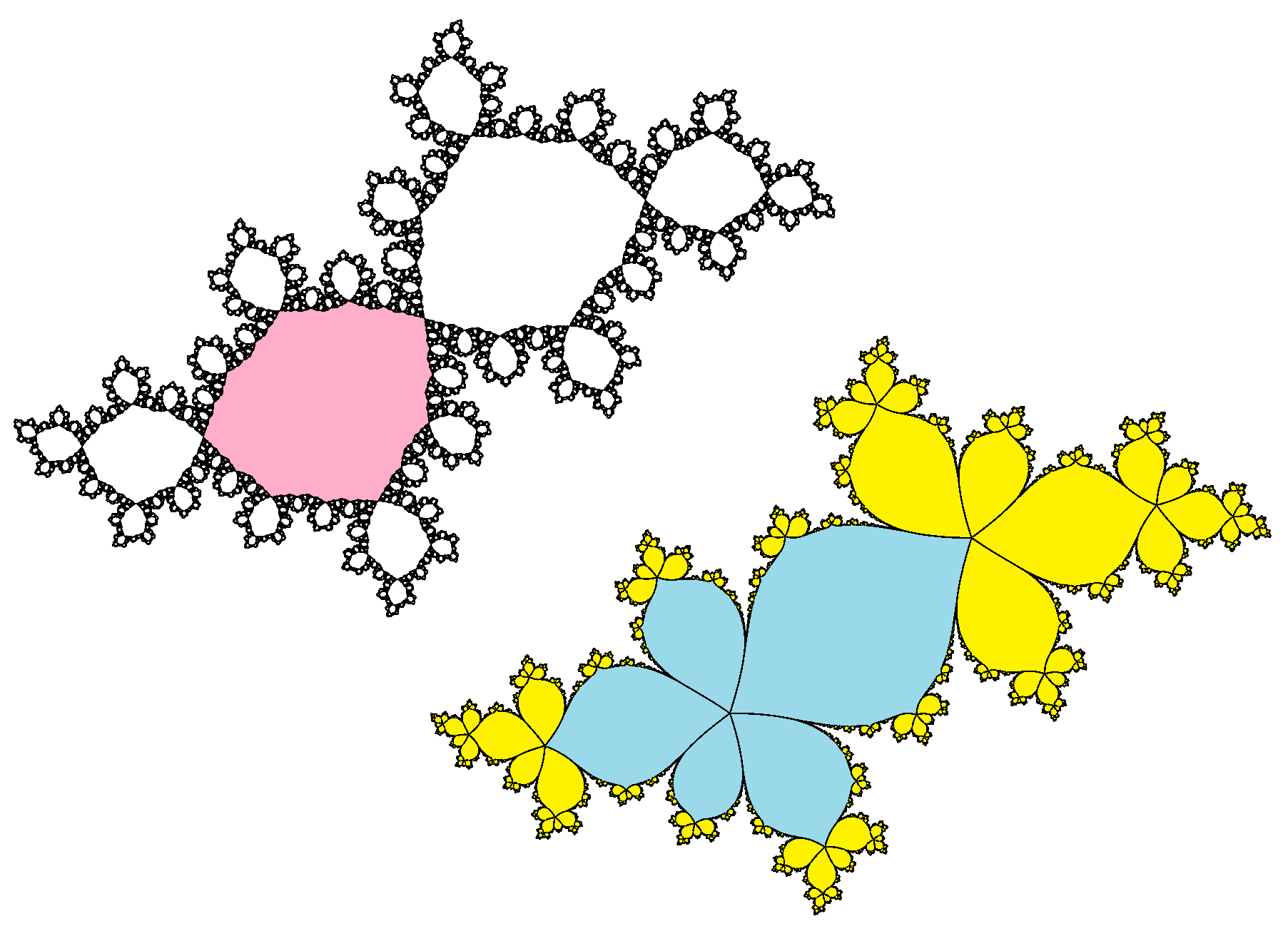}
  \caption{The Julia sets of $f_\theta$ with a Siegel disk (colored pink) and $f_\nu$ with a parabolic point (whose immediate parabolic basins are colored cyan), where $\theta=(\sqrt{5}-1)/2$ and $\nu=3/5$. They are both locally connected.}
  \label{Fig_Julias}
\end{figure}

In this paper, we consider the \textit{mixture} of the above known two cases: mating a bounded type Siegel quadratic polynomial $f_\theta$ with a parabolic quadratic polynomial $f_\nu$. Our first result is:

\begin{thmx}\label{thm:mateable}
If $\theta\in\R\setminus\Q$ is of bounded type and $\nu\in\Q$, then $f_\theta$ and $f_\nu$ are conformally mateable, and the mating is unique up to conjugacy by a M\"{o}bius map.
\end{thmx}

This gives an affirmative (partial) answer to the above question of Milnor.
In other words, one can paste the types of filled Julia sets in Figure \ref{Fig_Julias} along their boundaries in reverse directions to obtain a $2$-sphere, and the actions of the polynomials on their filled Julia sets match up to give an action on the sphere which is conjugate to a quadratic rational map with one fixed Siegel disk and one fixed parabolic point. See Figure \ref{Fig_Mating}.

\begin{figure}[!htpb]
 \setlength{\unitlength}{1mm}
  \centering
  \includegraphics[width=0.6\textwidth]{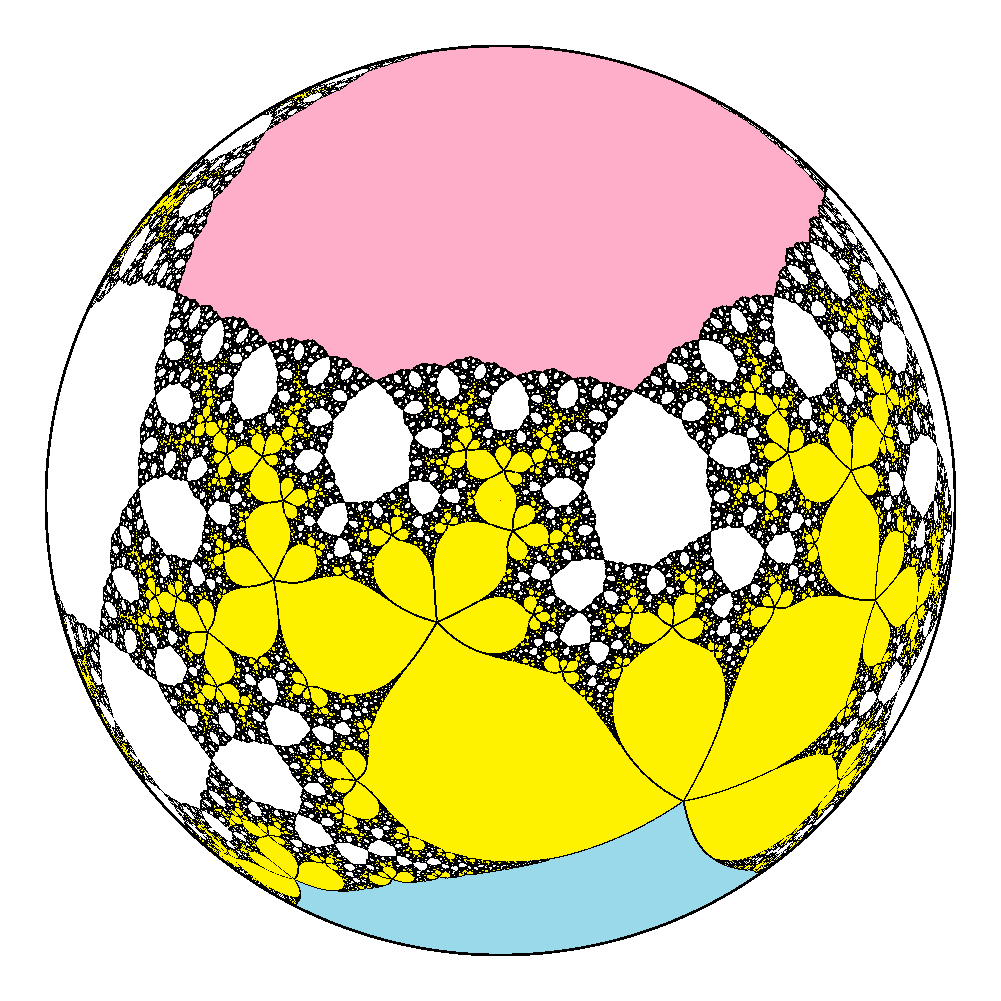}
  \caption{The Julia set of the mating $f_\theta\sqcup f_\nu$ on the Riemann sphere, where $\theta=(\sqrt{5}-1)/2$ and $\nu=3/5$. See also Figure \ref{Fig_Mating-plane}.}
  \label{Fig_Mating}
\end{figure}

\medskip
As an effective tool to construct interesting dynamics of rational maps, mating is also considered for pairs of higher degree polynomials.
Because of the emergence of more critical orbits and more combinatorics, there are some essential differences between the mating of higher degree polynomials and the quadratic case \cite{ST00b}. See \cite{Tan97}, \cite{Che12}, \cite{Sha13}, \cite{AR16}, \cite{Sha19} and \cite{Sha23} for the mating of special cubic polynomials.

In certain cases, it is difficult to obtain the explicit formulas of rational maps after mating. There are some efficient algorithms which can be used to compute their numerical approximations, see \cite{BH12}, \cite{Jun17} and \cite{Wil19}.
For some other results on the mating of polynomials, see \cite{Eps98}, \cite{Pet99}, \cite{BV06}, \cite{CPT12}, \cite{Mey14}, \cite{Yan15}, \cite{Wil16}, \cite{Ma19}, \cite{Luo22} and the references therein.
For more questions on polynomial matings, see \cite{BEKMPRT12}.

\subsection{Contraction property and local connectivity}

Before stating the strategy of the proof of Theorem \ref{thm:mateable}, we first recall the main idea in \cite{YZ01}, where the conformally mateable of two non-conjugate bounded type Siegel quadratic polynomials $f_\theta$ and $f_\nu$ was proved. The following quadratic rational map
\begin{equation}
F_{\theta,\nu}(z):=\frac{e^{2\pi\ii \nu}z+z^2}{1+e^{2\pi\ii \theta}z}
\end{equation}
serves naturally as the candidate of the mating of $f_\theta$ and $f_\nu$.
In view of the definition of conformal mating, the most important ingredient in the proof is to construct two continuous semiconjugacies $\phi_1$ and $\phi_2$ from the filled Julia sets $K(f_\theta)$ and $K(f_\nu)$ to the dynamical space of $F_{\theta,\nu}$.
Based on the local connectivity result of Petersen \cite{Pet96}, the crucial ingredient reduces to showing that the spherical diameter of the ``limbs" stretched out from the Siegel disks of $F_{\theta,\nu}$ tends to zero. In \cite{YZ01} this was carried out by providing a cubic Blaschke product model for $F_{\theta,\nu}$, and then using complex a prior bounds which was established in \cite{Yam99} to control the size of the ``limbs".
Such a similar idea has been used to the proof of mating Siegel quadratic polynomial with some other polynomials (see \cite{BV06}, \cite{Yan15}).

We have mentioned that when $\theta$ and $\nu$ are rational numbers, the mateable result has been established in \cite{HT04} by parabolic surgery. Hence for the proof of Theorem \ref{thm:mateable}, it seems that all necessary tools are already available. In fact, in \cite[p.\,31]{YZ01}, Yampolsky and Zakeri suggested that one may use the techniques developed in \cite{YZ01}, combined with the parabolic surgery developed in \cite{Hai98}, to handle the matings of parabolic quadratics with the bounded type Siegel quadratics.

In this paper, to prove Theorem \ref{thm:mateable} we shall use a \textit{different} argument from \cite{YZ01}. In particular, we don't rely on analytic Blaschke models and complex a prior bounds, nor do we rely on parabolic surgery. Instead, we use an expansive property near the Siegel disks which was established recently in the Main Lemma of \cite{WYZZ22} (see Lemma \ref{main-lem-WYZZ}). Our argument is based on the contraction property when pulling back certain rational maps near the bounded type Siegel disk boundaries. Such a contraction was used to prove the local connectivity of the Julia sets of certain rational maps with bounded type Siegel disks. The proof of Theorem \ref{thm:mateable} is based on combining this contraction property and the shrinking property in repelling sectors at parabolic points when pulling back the maps. As a by-product, we also prove that the Julia sets of some Siegel rational maps with parabolic points are locally connected.

\begin{thmx}\label{thm:local-connectivity}
Suppose $f$ is a Siegel rational map having a connected Julia set $J(f)$, and that the forward orbit of every critical point of $f$ satisfies one of the following:
\begin{enumerate}
\item It is finite; or
\item It lies in an attracting or a parabolic basin; or
\item It intersects the closure of a bounded type Siegel disk.
\end{enumerate}
Then $J(f)$ is locally connected.
\end{thmx}

Note that Theorem \ref{thm:local-connectivity} is a slight improvement of the Main Theorem of \cite{WYZZ22}, where the parabolic basin is not allowed.
By the assumption, the rational maps in Theorem \ref{thm:local-connectivity} can be regarded as a mixture of bounded type Siegels and geometrically finite ones.
In particular, Theorem \ref{thm:local-connectivity} gives an alternative proof of a partial result of \cite[Theorem 4(3)]{Roe10}, where the local connectivity of the Julia sets of cubic polynomials having a parabolic fixed point and a bounded type Siegel disk was proved.

\medskip
Recently, a Thurston theorem concerning the topological characterization of the rational maps with bounded type Siegel disk was established in \cite{Zha22}. Combining this and the contraction property near the Siegel disks mentioned above, one may conclude that the bounded type Siegel and subhyperbolic quadratic polynomials are conformally mateable.

\subsection{Organization of the paper}

In \S\ref{sec:lc}, we survey the dynamics near parabolic points, introduce suitable orbifold metric and prove that the boundary of every immediate attracting and parabolic basin of $f$ in Theorem \ref{thm:local-connectivity} is locally connected based on the contraction property near Siegel disk boundaries and the shrinking property in repelling sectors. Then Theorem \ref{thm:local-connectivity} is proved by using the contraction properties again based on Whyburn's criterion on the local connectivity.

In \S\ref{sec:defi}, we discuss the combinatorics of the maps $f_{\theta}$ and $f_{\nu}$ in Theorem \ref{thm:mateable} and give an address for each preimage of Siegel disk under $f_{\theta}$ and each preimage of immediate parabolic basins under $f_{\nu}$. In particular, the ``limbs'' and ``itinearies'' of both types of polynomials are defined.

In \S\ref{proof}, we emphasize that the rational map $F_{\theta,\nu}$ having a Siegel disk with rotation number $\theta$ and a parabolic fixed point with multiplier $e^{2\pi\ii \nu}$ serves the model map of the mating of $f_{\theta}$ and $f_{\nu}$. We prove that the diameters of analogous definition of ``limbs'' of $F_{\theta,\nu}$ goes to zero as the depth goes to infinity. Based on this, we prove Theorem \ref{thm:mateable} by verifying all the conditions in the definition of the conformal mating.

\medskip
\noindent\textbf{Acknowledgements.}
The authors would like to thank Guizhen Cui and Gaofei Zhang for helpful conversations and comments. This work was supported by NSFC (Grant Nos. 12222107 and 12071210).

\section{Contraction near Siegel and parabolic boundaries} \label{sec:lc}

In this section, we collect several types of contraction properties: the contraction near parabolic points, near Siegel disk boundaries and the global contraction with respect to the orbifold metric, to prove Theorem \ref{thm:local-connectivity} without using puzzles. The contraction established in this section will be used later to prove the mateability of the polynomials in Theorem \ref{thm:mateable}.

\subsection{Dynamics with parabolic points}

Let $f$ be a rational map with degree at least two having a multiple fixed point at $\zeta_0$, i.e., $f'(\zeta_0)=\zeta_0$ and $f'(\zeta_0)=1$. Up to a conformal coordinate transformation, we assume $\zeta_0=0$. Then there exist $a\in\C\setminus\{0\}$ and an integer $p\geqslant 1$ such that in a neighborhood of $0$,
\begin{equation}\label{equ:f}
f(z)=z+az^{p+1}+O(z^{p+2}).
\end{equation}
A complex number $\mv$ is called a \textit{repelling vector} (resp. \textit{attracting vector}) for $f$ at $0$ if $pa\textbf{v} ^p=1$ (resp. $pa\textbf{v} ^p=-1$).
If an orbit $f:z_0\mapsto z_1\mapsto \cdots$ converges to zero nontrivially (i.e., the orbit of $z_0$ is not iterated onto $0$), then according to \cite[\S 10]{Mil06}, $z_n\sim\textbf{v} /n^{1/p}$ as $n\to\infty$ for an attracting vector $\mv$.
In this case we say that the orbit of $z_0$ converges to $0$ in the direction $\mv$ and call $P_{\mv}$ the \textit{immediate parabolic basin} (attaching at the parabolic fixed point $0$) in the direction $\mv$, where $P_{\mv}$ is the Fatou component containing $z_n$ for all large enough $n$.

Note that $f$ has exactly $p$ repelling (resp. attracting) vectors at $0$. Let $\mv$ be an attracting or repelling vector.
For any $\zeta_0\in\C$, $\alpha\in(0,\frac{2\pi}{p})$ and $\delta>0$, we denote
\begin{equation}\label{equ:A-mv-1}
A_{\mv}^{\zeta_0}(\alpha,\delta):=\big\{z\in\C:\big|\arg\big(\tfrac{z-\zeta_0}{\mv}\big)\big|<\tfrac{\alpha}{2} \text{ and } 0<|z-\zeta_0|<\delta\big\}.
\end{equation}
For simplicity, we denote $A_{\mv}(\alpha,\delta):=A_{\mv}^{\zeta_0}(\alpha,\delta)$ if $\zeta_0=0$.
For $z_0\in\C$ and $\delta>0$, we denote $\D(z_0,\delta):=\{z\in\C:|z-z_0|<\delta\}$.
For a subset $Z$ in $\EC$, the forward orbit of $Z$ under $f$ is $O^+(Z):=\bigcup_{n\geqslant 0} f^{\circ n}(Z)$.
Let $\Crit(f)$ be the set of all critical points of $f$.

\begin{lem}\label{lem:att-p}
There exists $\delta>0$  such that
\begin{enumerate}
    \item If $\mv$ is an attracting vector of $f$ at $0$, then for any $\alpha\in(0,\frac{2\pi}{p})$, there exists a Jordan domain $Q_\mv=Q_\mv(\alpha)$ in $P_{\mv}$, such that $\overline{Q}_\mv\cap\partial P_{\mv}=\{0\}$, $\overline{f(Q_\mv)}\subset Q_\mv\cup \{0\}$ and $A_{\mv}(\alpha,\delta)\subset Q_\mv\cap\D(0,\delta)\subset A_\mv(\frac{1}{2}(\alpha+\frac{2\pi}{p}),\delta)$;
    \item If $\alpha\in(\tfrac{15\pi}{8p},\frac{2\pi}{p})$, $|z|<\delta$ and $z\not\in Q_{\mv}\cup\{0\}$ for any attracting vector $\mv$ of $f$ at $0$, then $z\in A_{\mv'}(\frac{\pi}{4p},\delta)$ for a repelling sector $\mv'$ of $f$ at $0$.
\end{enumerate}

Moreover, if $P_{\mv}$ is simply connected for the attracting vector $\mv$, then $Q_\mv$ can be chosen further to be a quasi-disk satisfying $P_{\mv}\cap O^+(\Crit(f))\subset Q_\mv$ and having a piecewise smooth boundary.
\end{lem}

\begin{proof}
The idea of the proof can be found\footnote{We refer to Lemma 2.2 in \href{https://math.univ-angers.fr/~tanlei/papers/tymsri.pdf}{https://math.univ-angers.fr/$\sim$tanlei/papers/tymsri.pdf}, where an expanded version of \cite{TY96} was attached.} in \cite{TY96}. For the completeness, we sketch a proof here. Let $\mv$ be an attracting vector at $0$. By Leau-Fatou's theorem \cite[\S 10]{Mil06}, there exists a holomorphic function $\Phi:P_\mv\to\C$ such that $\Phi(f(z))=T_1\circ\Phi(z)$ for all $z\in P_\mv$, where $T_1(w)=w+1$. Moreover,
\begin{equation}\label{equ:Phi-expand}
\Phi(z)\sim -1/(apz^p) \text{\quad as } z\to 0 \text{ in } P_\mv.
\end{equation}
The critical points of $\Phi$ are points $z\in P_\mv$ such that $f^{\circ k}(z)$ is a critical point of $f$ for an integer $k\geqslant 0$.

For $x\in\R$ and $\eta\in(0,2\pi)$, we denote
\begin{equation}
S(x,\eta):=\big\{w\in\C: |\arg(w-x)|<\tfrac{\eta}{2}\big\}.
\end{equation}
For given $\alpha\in (0,\frac{2\pi}{p})$, let $\eta_0:=\frac{\pi}{2}+\frac{3p}{4}\alpha$. Then
\begin{equation}\label{equ:angle-1}
\tfrac{\eta_0}{p}\in\big(\alpha,\tfrac{1}{2}(\alpha+\tfrac{2\pi}{p})\big).
\end{equation} Taking $x_0=x_0(\alpha)>0$ large enough such that $S(x_0,\eta_0)$ contains no critical values of $\Phi$. Note that $\Phi:P_\mv\to\C$ is a surjection and $S(x_0,\eta_0)$ is simply connected.
Let $Q_\mv$ be the unique connected component of $\Phi^{-1}(S(x_0,\eta_0))$ whose boundary contains the parabolic fixed point $0$. By \eqref{equ:Phi-expand}, $Q_\mv$ is a Jordan domain, and in particular, by \eqref{equ:angle-1} it is a quasi-disk having a piecewise smooth boundary satisfying Part (a) if $\delta>0$ is small enough. Moreover, Part (b) holds for $\alpha\in(\tfrac{15\pi}{8p},\frac{2\pi}{p})$ and small enough $\delta>0$.

\medskip
Suppose $P_{\mv}$ is simply connected and denote $D_0:=Q_\mv$ for a fixed $\alpha\in(\tfrac{15\pi}{8p},\frac{2\pi}{p})$. In the following we construct a new $Q_\mv$ by taking preimages of $D_0$ in $P_\mv$.
Without loss of generality we assume that the above $x_0=x_0(\alpha)>0$ is chosen such that $\partial Q_\mv\setminus\{0\}$ is disjoint from $O^+(\Crit(f))$.
For each $n\geqslant 1$, denote by $D_n$ the unique connected component of $f^{-n}\left(D_0\right)$ containing $D_0$.
Since $f: D_n \rightarrow D_{n-1}$ is a branched covering and $P_{\mv}$ is simply connected, by Riemann-Hurwitz's formula, it follows that each $D_n$ is a quasi-disk having a piecewise smooth boundary.
Note that $\overline{D}_n \cap \partial P_\mv$ consists of finitely many points $a_n$ satisfying $f^{\circ n}(a_n)=0$ and $\overline{D}_n \cap \partial P_\mv\supset \overline{D}_{n-1} \cap \partial P_\mv$.
Let $n_0\geqslant 1$ be an integer such that $P_{\mv}\cap O^+(\Crit(f))\subset D_{n_0}$. Now we modify $D_{n_0}$ to obtain a new $Q_\mv$ as follows.

Let $x_0=x_0(\alpha)>0$ be chosen as above for $\alpha\in(\tfrac{15\pi}{8p},\frac{2\pi}{p})$. We take two sequence of numbers $(x_j)_{1\leqslant j\leqslant n_0}$ and $(y_j)_{1\leqslant j\leqslant n_0}$ such that
\begin{itemize}
\item $x_0<x_{n_0}<x_{n_0-1}<\cdots<x_2<x_1$ and $0<y_{n_0}<y_{n_0-1}<\cdots<y_2<y_1$;
\item $x_j+1<x_{j-1}$  for all $2\leqslant j\leqslant n_0$; and
\item There is no critical value of $\Phi$ in $S_{n_0}$, where for $1\leqslant j\leqslant n_0$,
\begin{equation}
S_j:=S(x_0-2j,\eta_0)\setminus\big\{w\in\C: \re w\leqslant x_j \text{ and } |\im w|\leqslant y_j\big\}.
\end{equation}
\end{itemize}
Take $a \in (\overline{D}_{n_0} \cap \partial P_\mv)\setminus\{0\}$. Assume $1\leqslant j\leqslant n_0$ is the minimal integer such that $f^{\circ j}(a)=0$. Let $W_a$ be the unique connected component of $\Phi^{-1}\left(S_j\right)$ with $a\in \overline{W}_a$. Note that for $2\leqslant j\leqslant n_0$,
\begin{equation}
\Phi \big(f(W_a)\big)=T_1\big(\Phi(W_a)\big)=T_1(S_j) \supset \overline{S}_{j-1}=\overline{\Phi\big(W_{f(a)}\big)}.
\end{equation}
Then for sufficiently small $\delta>0$, the new $Q_\mv:=D_{n_0}\setminus\bigcup_{a \in (\overline{D}_{n_0} \cap \partial P_\mv)\setminus\{0\}}\overline{W}_a$ satisfies the required properties.
\end{proof}

The Jordan domain $Q_\mv$ in Lemma \ref{lem:att-p} satisfying (a) is called an \textit{attracting petal}.
For a repelling vector $\mv$, based on Lemma \ref{lem:att-p}(b) we call $A_{\mv}(\delta):=A_{\mv}(\frac{\pi}{4p},\delta)$ a \textit{thin repelling sector}.

\medskip

For a rational map $f: \EC\to \EC$ with $\deg(f)\geqslant 2$, the \textit{postcritical set} of $f$ is
$$
\MP(f):=\overline{\bigcup_{n\geqslant 1 } f^{\circ n}\big(\Crit(f)\big)}.
$$
Let $\dist_{\EC}(\cdot,\cdot)$ and $\diam_{\EC}(\cdot)$ denote the distance and diameter with respect to the spherical metric respectively.
A sequence $\{V_n\}_{n\geqslant 0}$ is called \textit{a pull back sequence} of a Jordan domain $V_0$ under $f$ if $V_{n+1}$ is a connected component of $f^{-1}(V_n)$ for all $n\geqslant 0$. The following classical shrinking lemma was proved in \cite[p.\,86]{LM97} (see also \cite{Man93}, \cite{TY96}).

\begin{lem}[Classical shrinking lemma]\label{lem:semi-hyperbolic}
Let $U_0, V_0$ be two Jordan domains in $\EC$ and $D\geqslant 1$. Suppose $U_0$ is not contained in any rotation domain of $f$ and that $V_0$ is compactly contained in $U_0$.
Then for any $\varepsilon>0$, there exists an $N\geqslant 1$ such that for any pull back sequence $\{U_n\}_{n\geqslant 0}$ satisfying $\deg(f^{\circ n}:U_n\to U_0)\leqslant D$, $\diam_{\EC}(V_n)<\varepsilon$ for all $n\geqslant N$, where $V_n$ is any component of $f^{-n}(V_0)$ contained in $U_n$.
\end{lem}

In particular, Lemma \ref{lem:semi-hyperbolic} holds when $\overline{V}_0\cap \MP(f)=\emptyset$ and $V_0$ is not contained in any rotation domain.

\begin{lem}\label{lem:parabolic-rat}
Let $f$ be a rational map having a multiple fixed point at $0$ with $\deg(f)\geqslant 2$.
Suppose $J(f)\cap\MP(f)\cap\D(0,\delta')=\{0\}$ for some $\delta'>0$. Then there exists $\delta\in(0,\delta']$ such that for any $\varepsilon>0$, there exists an $N\geqslant 1$ such that for any pull back sequence $\{V_n\}_{n\geqslant 0}$ of any thin repelling sector $A_\mv(\delta)$ at $0$, $\diam_{\EC}(V_n)<\varepsilon$ for all $n\geqslant N$.
\end{lem}

\begin{proof}
Since $f'(0)=1$, there exists an open neighborhood $U$ of $0$ such that the univalent branch $g$ of $f^{-1}$ with $g(0)=0$ is well defined in $U$. According to local dynamics of $g$ near $0$ (see \cite[\S 10]{Mil06}), the attracting vectors, repelling vectors and immediate parabolic basins etc for $g$ are also defined analogously as $f$. In particular, every attracting (resp. repelling) vector of $g$ at $0$ is exactly the repelling (resp. attracting) vector of $f$ and vice versa. Hence if $\mv$ is a repelling vector of $f$, then it is an attracting vector of $g$. By the definition of thin repelling sector, we conclude that $V_0:=A_\mv(\delta)$ is contained in the immediate parabolic basin attaching at $0$ in the direction $\mv$ for $g$ if $\delta>0$ is small enough. In particular, we have $\diam_{\EC}(g^{\circ n}(V_0))\to 0$ as $n\to\infty$.

Suppose $J(f)\cap\MP(f)\cap\D(0,\delta')=\{0\}$ for some $\delta'>0$. Decreasing the above $\delta>0$ if necessary, we assume that $\delta\in(0,\delta']$, $\overline{A_\mv(\delta)}\cap \MP(f)=\{0\}$ and $A_\mv(\delta)\subset Q$, where $Q$ is a small attracting petal for $g$ satisfying $\overline{g(Q)}\subset Q\cup\{0\}$ and $\overline{Q}\cap \MP(f)=\{0\}$. Therefore, it suffices to prove this lemma for $V_0:=Q$.
Let $\varepsilon>0$ be given. Then there exists an integer $N_1\geqslant 1$ such that
\begin{equation}\label{equ:eps-1}
\diam_{\EC}(g^{\circ n}(V_0))<\varepsilon \text{\quad for all\quad} n\geqslant N_1.
\end{equation}

Since $V_0$ is a Jordan domain and $V_0\cap \MP(f)=\emptyset$, it follows that every component $V_n$ of $f^{-n}(V_0)$ is a Jordan domain and the restriction of $f^{\circ n}$ on $V_n$ is conformal.  Note that $\overline{V}_0\cap \MP(f)=\{0\}$ and $f^{-1}(V_0)$ has exactly $d=\deg(f)\geqslant 2$ connected components $g(V_0)$, $V_1^{(1)}$, $\cdots$, $V_1^{(d-1)}$.
Then there exists $n_0\geqslant 1$ such that for any $1\leqslant i\leqslant d-1$, the closure of any connected component of $f^{-n_0}(V_1^{(i)})$ is disjoint with $\MP(f)$.
By Lemma \ref{lem:semi-hyperbolic}, there exists $N_2\geqslant 1$ such that
\begin{equation}\label{equ:eps-2}
\diam_{\EC}(V_n)<\varepsilon \text{\quad for all \quad} n\geqslant N_2 \text{ and } 1\leqslant i\leqslant d-1,
\end{equation}
where $V_n$ is any connected component of $f^{-(n-1)}(V_1^{(i)})$.
Note that for each $k\geqslant 2$, $f^{-1}\big(g^{\circ (k-1)}(V_0)\big)$ has exactly $d$ connected components $g^{\circ k}(V_0)$, $V_{k}^{(1)}$, $\cdots$, $V_{k}^{(d-1)}$, where $V_{k}^{(i)}\subset V_1^{(i)}$ for all $1\leqslant i\leqslant d-1$ since $g^{\circ (k-1)}(V_0)\subset V_0$. By \eqref{equ:eps-1} and \eqref{equ:eps-2}, we have
\begin{equation}\label{equ:eps-3}
\diam_{\EC}(V_n)<\varepsilon \text{\quad for all \quad} n\geqslant N_1+N_2,
\end{equation}
where $V_n$ is any connected component of $f^{-n}(V_0)$.
The proof is finished by setting $N:=N_1+N_2$.
\end{proof}

\subsection{Orbifold metric and contraction properties}

For the local connectivity of compact subsets on the Riemann sphere, the following criterion is very useful (see \cite[Theorem 4.4, p.113]{Why42}).

\begin{lem}[LC criterion]\label{lem:LC-criterion}
A compact subset $X$ in $\EC$ is locally connected if and only if the following two conditions hold:
\begin{enumerate}
\item  The boundary of every component of $\EC\setminus X$ is locally connected; and
\item  For any given $\varepsilon>0$, there are only finitely many components of $\EC\setminus X$ whose spherical diameters are greater than $\varepsilon$.
\end{enumerate}
\end{lem}

To apply the above criterion to the Julia sets of rational maps, one needs to study the local connectivity of the boundaries of Fatou components and control their size.
In \cite{Pet96}, Petersen proved that the Julia set of $f_\theta(z)=e^{2\pi\ii\theta}z+z^2$ is locally connected when $\theta$ is of bounded type. This result has been extended to a number of rational maps with bounded type Siegel disks in \cite{WYZZ22} based on the above local connectivity criterion, and the following lemma plays a crucial role (see \cite[Main Lemma]{WYZZ22}):

\begin{lem}\label{main-lem-WYZZ}
Let $f$ be a rational map having a fixed bounded type Siegel disk $\Delta$ with $\deg(f)\geqslant 2$.  Suppose $\dist_{\EC}(\MP(f)\setminus\partial{\Delta},\partial{\Delta})>0$. Then for any $\varepsilon>0$ and any Jordan domain $V_0\subset \EC\setminus\overline{\Delta}$ with $\emptyset\neq\overline{V}_0 \cap \MP(f) \subset \partial\Delta$, there exists an $N=N(\varepsilon,V_0,f) \geqslant 1$, such that for any pull back sequence $\{V_n\}_{n\geqslant 0}$ of $V_0$, $\diam_{\EC}(V_n)<\varepsilon$ for all $n\geqslant N$.
\end{lem}

The main purpose of this section is to extend the Main Theorem in \cite{WYZZ22} slightly to rational maps with extra parabolic points.
Let $f$ be a rational map in Theorem \ref{thm:local-connectivity}. By the assumption, the periodic Fatou components of $f$ can only be attracting basins (including super-attracting), parabolic basins or bounded type Siegel disks. Note that the periodic cycles in $J(f)$ can only be parabolic or repelling.
Since $J(f)$ is connected by the assumption, all Fatou components of $f$ are simply connected. Without loss of generality, we assume that
\begin{enumerate}
    \item All periodic Fatou components of $f$ have period one (considering an appropriate iteration of $f$ if necessary) and they consist of\footnote{We assume that $f$ has at least one parabolic basin since otherwise, $f$ will become the same map as the one studied in \cite{WYZZ22}.}
\begin{itemize}
\item $r_0\geqslant 1$ fixed Siegel disks $\Delta_r$, where $1\leqslant r\leqslant r_0$ and $\infty\in\Delta_1$,
\item $s_0\geqslant 0$ fixed immediate attracting basins $A_s$, where $1\leqslant s\leqslant s_0$, and
\item $t_0\geqslant 1$ fixed immediate parabolic basins $P_t$ whose boundary contains the parabolic fixed point\footnote{Note that different immediate parabolic basins may correspond to the same parabolic fixed point.} $\zeta_t$, where $1\leqslant t\leqslant t_0$; and
\end{itemize}
    \item $f$ is postcritically finite in the attracting basins (if any, by a standard quasiconformal surgery, see \cite[Chapter 4]{BF14a}).
\end{enumerate}

In the rest of this section we fix the rational map $f$ according to the above normalization.
Note that all attracting and parabolic basins of $f$ are bounded in $\C$. This is convenient for us to use the Euclidean metric at some places.
For every $1\leqslant s \leqslant s_0$, there exists a small quasi-disk $B_s$ containing the unique critical point (without counting multiplicity) in $A_s$ such that
\begin{equation}\label{equ:f-B-s}
f(\overline{B}_s)\subset B_s.
\end{equation}
Suppose $f$ has exactly $t_0'\in[1,t_0]$ parabolic fixed points $\{\zeta_{t'}:1\leqslant t'\leqslant t_0'\}$ and that $f$ has exactly $p_{t'}\geqslant 1$ attracting directions at each $\zeta_{t'}$. Then $\sum_{t'=1}^{t_0'} p_{t'}=t_0$.
The following result is an immediate corollary of Lemma \ref{lem:att-p}.

\begin{cor}\label{cor:petal}
There exists $\delta>0$  such that for every immediate parabolic basin $P_t$ attaching at $\zeta_{t'}$ with attracting direction $\mv_t$, where $1\leqslant t'\leqslant t_0'$, there exists a quasi-disk $Q_t$ in $P_t$ having a piecewise smooth boundary such that
\begin{enumerate}
\item $P_t\cap O^+(\Crit(f))\subset Q_t$, $\overline{Q}_t\cap\partial P_t=\{\zeta_{t'}\}$ and $\overline{f(Q_t)}\subset Q_t\cup \{\zeta_{t'}\}$;
\item $A_{\mv_t}^{\zeta_{t'}}(\alpha,\delta)\subset Q_t\cap\D(\zeta_{t'},\delta)\subset A_{\mv_t}^{\zeta_{t'}}(\frac{1}{2}(\alpha+\frac{2\pi}{p_{t'}}),\delta)$, where $\alpha=\tfrac{31\pi}{16p_{t'}}$;
\item If $z\in \D(\zeta_{t'},\delta)$ and $z\not\in Q_{\mv}\cup\{\zeta_{t'}\}$ for any attracting vector $\mv$ of $f$ at $\zeta_{t'}$, then $z\in A_{\mv'}^{\zeta_{t'}}(\frac{\pi}{4p_{t'}},\delta)$ for a repelling sector $\mv'$ of $f$ at $\zeta_{t'}$.
\end{enumerate}
\end{cor}

By remarking the subscripts, we use $\{Q_t:1\leqslant t\leqslant t_0\}$ to denote the collection of all attracting petals obtained in Corollary \ref{cor:petal}, where $Q_t\subset P_t$ for all $1\leqslant t\leqslant t_0$. Decreasing $\delta>0$ if necessary, in the following we assume that $\overline{\D}(\zeta_{t'},\delta)\cap \overline{\D}(\zeta_{t''},\delta)=\emptyset$ for any two different parabolic fixed points $\zeta_{t'}$ and $\zeta_{t''}$.

\medskip
By Corollary \ref{cor:petal}(a) and the assumption in Theorem \ref{thm:local-connectivity}, the following set is finite or empty:
\begin{equation}\label{equ:MP-1}
\MP_1(f) := \MP(f)\setminus \Big(\bigcup_{r=1}^{r_0}\overline{\Delta}_r \cup \bigcup_{s=1}^{s_0}\overline{B}_s \cup \bigcup_{t=1}^{t_0}\overline{Q}_t\Big).
\end{equation}
Denote
\begin{equation}\label{equ:W-W1}
W:=\EC\setminus\Big(\MP(f)\cup \bigcup_{r=1}^{r_0}\overline{\Delta}_r \cup \bigcup_{s=1}^{s_0}\overline{B}_s \cup \bigcup_{t=1}^{t_0}\overline{Q}_t\Big)
\text{\quad and\quad}
W_1:=W\cup \MP_1(f).
\end{equation}
Note that $W$ and $W_1$ are domains in $\EC$ (actually are bounded domains in $\C$ by the assumption $\infty\in\Delta_1$).
By \eqref{equ:f-B-s} and Corollary \ref{cor:petal}(a), we have\footnote{Note that $f^{-1}(W)$ and  $f^{-1}(W_1)$ may have several connected components.} $f^{-1}(W)\subset W$ and  $f^{-1}(W_1)\subset W_1$.

\begin{lem}\label{lem:contrac}
There exists $\delta_0>0$ such that for any $\varepsilon>0$, there exists $N\geqslant 1$, such that for any Jordan domain $V_0$ in $W_1$ with $\diam_{\EC}(V_0)< \delta_0$, then $\diam_{\EC}(V_n)<\varepsilon$ for all $n\geqslant N$, where $V_n$ is any component of $f^{-n}(V_0)$.
\end{lem}

\begin{proof}
According to \cite{Zha11}, every $\partial{\Delta_r}$ is a quasi-circle and $\overline{\Delta}_{r'}\cup\overline{\Delta}_{r''}=\emptyset$ for any different integers $r',r''\in [1,r_0]$.
Note that $\partial W_1$ consists of $r_0+s_0+t_0'$ connected components:
\begin{itemize}
\item $r_0+s_0$ of them are Jordan curves;
\item $t_0'$ of them are written as $\bigcup_{i=1}^{p_{t'}}\partial Q_{t'(i)}$, where $1\leqslant t'\leqslant t_0'$ and $\{Q_{t'(i)}:1\leqslant i\leqslant p_{t'}\}$ are attracting petals attaching at the parabolic fixed point $\zeta_{t'}$.
\end{itemize}
This implies that there exist finitely many Jordan domains $\{U_k:1\leqslant k\leqslant k_0\}$ in $W_1$ and a small number $\delta_0>0$ such that
\begin{itemize}
    \item $W_1\subset \bigcup_{k=1}^{k_0} U_k$;
    \item Every $U_k$ satisfies one of the following cases:
    \begin{itemize}
    \item $\sharp(U_k\cap \MP(f))=\sharp(U_k\cap \MP_1(f))\leqslant 1$, or
    \item $\overline{U}_k\cap \MP(f)\subset \partial\Delta_r$ for some $1\leqslant r\leqslant r_0$, or
    \item $\overline{U}_k\cap \MP(f)=\{\zeta_{t'}\}$ is a parabolic fixed point and $U_k$ is contained in a thin repelling sector $A_{\mv'}^{\zeta_{t'}}(\frac{\pi}{4p_{t'}},\delta)$ of $f$ at $\zeta_{t'}$ (by Corollary \ref{cor:petal}(c));
    \end{itemize}
    \item Any Jordan domain $V_0$ in $W_1$ with $\diam_{\EC}(V_0)<\delta_0$ is contained in some $U_k$.
\end{itemize}
Given $\varepsilon>0$, by Lemmas \ref{lem:semi-hyperbolic}, \ref{lem:parabolic-rat} and \ref{main-lem-WYZZ}, there exists $N\geqslant 1$ such that for all $n\geqslant N$, any component of $f^{-n}(U_k)$ has spherical diameter less than $\varepsilon$.
\end{proof}

Let $\N_+:=\{1,2,3,\cdots\}$ be the set of all positive integers. The following definition can be found in \cite{Thu84} (see also \cite[Appendix A.2]{McM94b} and \cite[Appendix E]{Mil06}).

\begin{defi}[Orbifolds]
A pair $\MO=(S,\nu)$ consisting of a Riemann surface $S$ and a ramification function $\nu:S\to\N_+$ which takes the value $\nu(z)=1$ except on a discrete closed subset is called a \textit{Riemann surface orbifold} (\textit{orbifold} in short). A point $z\in S$ is called a \textit{ramified point} if $\nu(z)\geqslant 2$.
\end{defi}

Suppose that $h:\widetilde{S}\to S$ is a holomorphic map of Riemann surfaces, and that $\widetilde{\MO}=(\widetilde{S},\widetilde{\nu})$ and $\MO=(S,\nu)$ are orbifolds. Then
\begin{itemize}
\item The map $h:\widetilde{\MO}\to \MO$ is \textit{holomorphic} if $\nu(h(w))$ divides $\deg_{w}(h)\,\widetilde{\nu}(w)$ for all $w\in \widetilde{S}$, where $\deg_{w}(h)$ is the local degree of $h$ at $w$;
\item The map $h:\widetilde{\MO}\to \MO$ is an \textit{orbifold covering} if $h:\widetilde{S}\to S$ is a branched covering and $\nu(h(w))=\deg_{w}(h)\,\widetilde{\nu}(w)$ for all $w\in \widetilde{S}$;
\item If $h:\widetilde{\MO}\to \MO$ is an orbifold covering, $\widetilde{S}$ is simply connected and $\widetilde{\nu}\equiv 1$ on $\widetilde{S}$, then $h$ is called a \textit{universal covering}.
\end{itemize}

With only two exceptions, each orbifold $\MO=(S,\nu)$ has a universal covering surface which is conformally equivalent to either $\EC$, $\C$ or $\D$. In the last case the orbifold is called \textit{hyperbolic}, and we only consider such type of orbifolds in this paper. In particular, if $S$ is a hyperbolic Riemann surface, then $\MO$ is a hyperbolic orbifold (see \cite[Theorem A.2]{McM94b}) and the universal covering $h:\D\to \MO$ descends the hyperbolic metric $\rho_{\D}(w)|dw|$ in $\D$ to the \textit{orbifold metric} $\sigma_{\MO}(z)|dz|$ in $S$ which satisfies $\sigma_{\MO}(h(w))|h'(w)|=\rho_{\D}(w)$.

The well-known Schwarz-Pick theorem for hyperbolic Riemann surfaces generalizes to the setting of hyperbolic orbifolds \cite[Proposition 17.4]{Thu84} (see also \cite[Theorem A.3]{McM94b}). In particular, it has the following consequence.

\begin{lem}\label{lem:Thurston}
Let $\widetilde{\MO}=(\widetilde{S},\widetilde{\nu})$ and $\MO=(S,\nu)$ be two hyperbolic orbifolds such that $\widetilde{S}\subset S$. Suppose $h: \widetilde{\MO}\to \MO$ is an orbifold covering map, and that the inclusion $\widetilde{\MO}\hookrightarrow \MO$ is holomorphic but not an orbifold covering. Then
\begin{equation}
\|D h(z)\|_{\mathcal{O}}:=|h'(z)| \frac{\sigma_{\mathcal{O}}(h(z))}{\sigma_{\mathcal{O}}(z)}>1, \quad \text { for } z \in \widetilde{\MO}.
\end{equation}
\end{lem}

For the hyperbolic Riemann surface $W_1$ introduced in \eqref{equ:W-W1}, we define
\begin{equation}\label{equ:nu}
\nu(z):=
\left\{
\begin{array}{ll}
1  &~~~~~~~\text{if}~ z\in W_1\setminus \MP_1(f), \\
\text{lcm} \big\{\deg_{w}(f)\nu(w):w\in f^{-1}(z)\big\} &~~~~~~\text{if}~ z\in \MP_1(f).
\end{array}
\right.
\end{equation}
Since $\MP_1(f)$ is finite or empty, we conclude that $\nu$ is well-defined. Then
\begin{equation}\label{equ:O-f}
\MO_f:=(W_1,\nu)
\end{equation}
is a hyperbolic orbifold and we use $\sigma_{\MO_f}(z)|dz|$ to denote the orbifold metric therein.
Let $\widetilde{W}_1$ be any connected component of $f^{-1}(W_1)$. Then $\widetilde{W}_1$ is a proper subset of $W_1$. The orbifold $(\widetilde{W}_1,\widetilde{\nu})$ is defined by the ramification index
\begin{equation}\label{equ:nu-tilde}
\widetilde{\nu}(z):=\frac{\nu(f(z))}{\deg_z(f)}, \text{\quad for }z\in \widetilde{W}_1.
\end{equation}
Since $f : \widetilde{W}_1\to W_1$ is a branched covering, based on \eqref{equ:nu} and \eqref{equ:nu-tilde}, it is straightforward to verify that $f: (\widetilde{W}_1,\widetilde{\nu})\to(W_1,\nu)$ is an orbifold covering map.

\medskip
Recall that $f$ has exactly $r_0$ Siegel disks $\{\Delta_r:1\leqslant r\leqslant r_0\}$ and $t_0'$ parabolic fixed points $\{\zeta_{t'}:1\leqslant t'\leqslant t_0'\}$.
For $z\in\EC$, we define
\begin{equation}
d_{\Sie+\Par}(z):=\dist_{\EC}\Bigg(z,\,\,\bigcup^{r_0}_{r=1}\overline{\Delta}_r\cup\bigcup_{t'=1}^{t_0'}\{\zeta_{t'}\}\Bigg).
\end{equation}
Note that $\partial\widetilde{W}_1\cap\partial W_1$ may be non-empty. However, for any $\delta>0$, the set $\{z\in \widetilde{W}_1: d_{\Sie+\Par}(z)>\delta\}$ is compactly contained $W_1$.
Hence as an immediate corollary of Lemma \ref{lem:Thurston}, we have the following result.

\begin{lem}\label{lem:expand-orbifold}
For every $\delta>0$, there exists $\mu=\mu(\delta)>1$ such that
\begin{equation}
\| Df(z) \|_{\MO_f}=|f'(z)| \frac{\sigma_{\MO_f}(f(z))}{\sigma_{\MO_f}(z)}\geqslant \mu>1
\end{equation}
for $z\in \widetilde{W}_1$ such that $d_{\Sie+\Par}(z)>\delta$.
\end{lem}

The above result shows that $f$ expands the orbifold metric at the points in $\widetilde{W}_1$ which are away from Siegel disks and parabolic fixed points.

\subsection{Boundaries of attracting and parabolic basins}

To prove the local connectivity of the Julia sets of the maps $f$ in Theorem \ref{thm:local-connectivity}, we shall use the criterion in Lemma \ref{lem:LC-criterion}.
The main goal in this subsection is to prove the following result.

\begin{prop}\label{prop:attracting}
The boundary of every immediate attracting and parabolic basin of $f$ is locally connected.
\end{prop}

The idea to prove Proposition \ref{prop:attracting} is constructing a sequence of Jordan curves and proving that it is convergent uniformly. The convergence is based on the several expanding properties obtained in the last subsection.

\medskip
For the normalized rational map $f$ in Theorem \ref{thm:local-connectivity} (i.e., $\infty$ is contained in a Siegel disk of $f$), all attracting and parabolic basins are bounded in $\C$.
Note that every immediate attracting basin $A_s$ of $f$ is super-attracting and contains exactly one critical point (without counting multiplicity), where $1\leqslant s\leqslant s_0$.
There exists a conformal map $\psi: \D\to A_s$ which conjugates $w\mapsto w^u: \D\to \D$ to $f: A_s\to A_s$ for some integer $u\geqslant 2$.
For $r\in \R$ and $\theta\in\T:=\R/\Z$, we denote
\begin{equation}
R_\theta(r):=\psi\big((\tfrac{1}{2})^{\frac{1}{u^r}}e^{2\pi\ii\theta}\big).
\end{equation}
The image $R_\theta :=R_\theta\big(\R\big)$ is the \textit{internal ray} of angle $\theta$ in $A_s$.

\begin{defi}[Ray segments in $A_s$]
For every $\theta\in\T$ and $n\geqslant 1$, the following curve is called a \textit{ray segment} in the immediate super-attracting basin $A_s$:
\begin{equation}\label{defi-ray-seg-attr}
R_{\theta,n}:=R_\theta\big([1,n+1]\big).
\end{equation}
\end{defi}

For a given immediate parabolic basin $P_t$ which attaches at the parabolic fixed point $\zeta_{t'}$, where $1\leqslant t'\leqslant t_0'$, let $Q^{(0)}:=Q_t$ be the quasi-disk in $P_t$ determined by Corollary \ref{cor:petal}.
Define
\begin{equation}\label{equ:Q-n}
Q^{(n)}:=P_t\cap f^{-n}(Q^{(0)}).
\end{equation}
Then there exists an integer $u\geqslant 2$ such that for all $n\geqslant 0$,
\begin{itemize}
\item $Q^{(n)}$ is a quasi-disk in $P_t$ having a piecewise smooth boundary;
\item $\overline{Q^{(n)}}\subset Q^{(n+1)}\cup(f^{-n}(\zeta_{t'})\cap\partial P_t)$; and
\item $f:Q^{(n+1)}\to Q^{(n)}$ is a branched covering of degree $u\geqslant 2$.
\end{itemize}

In order to prove the local connectivity of $\partial P_t$, we give a typical parameterization of $(\partial Q^{(n)})_{n\geqslant 0}$ (see \cite[p.\,520]{ARS22} for example) and prove that the sequence of these curves converges uniformly.
Based on the above properties of $(Q^{(n)})_{n\geqslant 0}$, there exists a continuous function
\begin{equation}
\Lambda:\T\times [0,1]\to \overline{P}_t
\end{equation}
such that
\begin{itemize}
    \item $\Lambda(\cdot,0):\T\to \partial{Q^{(0)}}$ is a homeomorphism;
    \item $\Lambda(\cdot,1):\T\to \partial{Q^{(1)}}$ is a homeomorphism;
    \item $\Lambda(0,r)=\zeta_{t'}$ for all $r\in[0,1]$; and
    \item $\Lambda(\theta,r)\in P_t\setminus \overline{Q^{(0)}}$ for $\theta\neq 0$ and all $r\in(0,1)$.
\end{itemize}
This define a homotopy between $\partial{Q^{(0)}}$ and $\partial{Q^{(1)}}$ in $\overline{P}_t$.

\medskip
Let $\eta_0(\theta):=\Lambda(\theta,[0,1])$ be the curve connecting $z_0(\theta):=\Lambda(\theta, 0)$ with $z_1(\theta):=\Lambda(\theta, 1)$. We may assume that $\Lambda$ is smooth such that for $\theta$ sufficiently close to $0$,  $z_1(\theta)$ is the image of $z_0(\theta)$ under the branch of $f^{-1}$ that fixes $\zeta_{t'}$.
Now we extend $\Lambda$ to a map on $\T \times[0, +\infty)$ as follows. For $\theta \in \T$, let $\theta_1 \in \T$ be the unique number such that $f(z_1(\theta))=z_0(\theta_1)$. Since $f: P_t\setminus \overline{Q^{(1)}} \rightarrow P_t\setminus \overline{Q^{(0)}}$ is a covering map, there exists a unique continuous extension $\Lambda:\T \times[0,2] \to \overline{P}_t$ such that
\begin{equation}
f\big(\Lambda(\theta, r+1)\big)=\Lambda(\theta_1, r), \text{\quad for all } \theta \in \T \text{ and } r\in[0,1].
\end{equation}

Continuing inductively, we obtain the desired continuous extension:
\begin{equation}\label{equ:Lambda}
\Lambda: \T \times[0, +\infty) \rightarrow \overline{P}_t.
\end{equation}
For $\theta\in \T$ and $n \geqslant 0$, the curve $\eta_n(\theta):=\Lambda(\theta,[n, n+1])$ connects $\partial Q^{(n)}$ and $\partial Q^{(n+1)}$, and it is a pullback of the curve $\eta_0(\theta_n)$ for some $\theta_n \in \T$ under $f^{\circ n}$.

\begin{defi}[Ray segments in $\overline{P}_t$]
For every $\theta\in\R/\Z$ and integer $n\geqslant 1$, the following curve is called a \textit{ray segment} in the closure of the immediate parabolic basin $P_t$:
\begin{equation}\label{defi-ray-seg-para}
R_{\theta,n}:= \Lambda(\theta,[1,n+1]).
\end{equation}
\end{defi}

A ray segment $R_{\theta,n}$ is \textit{degenerate} if it is a singleton for some $n\geqslant 1$.
A non-degenerate ray segment $R_{\theta,n}=\eta_1(\theta)\cup\eta_2(\theta)\cup\cdots\cup\eta_n(\theta)$ is a piecewise smooth curve connecting $\partial{Q^{(1)}}$ with $\partial{Q^{(n+1)}}$.
Recall that $\sigma_{\MO_f}(z)|dz|$ is the orbifold metric defined in $W_1$ (see \eqref{equ:O-f}). We may assume further that $\Lambda$ is smooth in $\T\times[0,1]$ such that:
The length of $R_{\theta,1}=\eta_1(\theta)=\Lambda(\theta,[1,2])$ with respect to the Euclidean metric and the orbifold metric $\sigma_{\MO_f}(z)|dz|$ are uniformly bounded\footnote{If $z_0\in W_1$ is a ramified point (i.e., $\nu(z_0)\geqslant 2$), then $\sigma_{\MO_f}(z)$ is comparable to $1/|z-z_0|^\kappa$ in a small neighborhood of $z_0$ for some $\kappa\in(0,1)$ and a simple smooth curve has finite orbifold length in a small neighborhood of $z_0$ (see \cite[\S 19]{Mil06}).}.

In view of that all attracting and parabolic basins of $f$ are bounded in $\C$, we consider the length $l_\C(\cdot)$ with respect to the Euclidean metric at some places.
The proof of the following result is inspired by \cite[Lemma 5.5]{WYZZ22}. The difference here is that besides the ray segments \eqref{defi-ray-seg-attr} in attracting basins, we also need to consider the ray segments in parabolic basins \eqref{defi-ray-seg-para}.

\begin{lem}\label{lem:essential-bd}
For any $\varepsilon>0$, there exists $C=C(\varepsilon)>0$ such that for any non-degenerate ray segment $R_{\theta,n}$ with $\theta\in \T$ and $n\geqslant 1$, there exists a continuous curve $\widetilde{R}_n$ which is homotopic to $R_{\theta,n}$ in $W_1$ relative to their end points,  such that $\widetilde{R}_n$ is the union of at most three continuous curves $\widetilde{R}_n^{ini}$, $\widetilde{R}_n^{ess}$ and $\widetilde{R}_n^{end}$ satisfying
\begin{equation}
\diam_{\C}(\widetilde{R}_n^{ini})<\varepsilon, \quad
l_{\C}(\widetilde{R}_n^{ess})< C \text{\quad and\quad} \diam_{\C}(\widetilde{R}_n^{end})<\varepsilon.
\end{equation}
\end{lem}

\begin{proof}
We divide the arguments into several steps.

\medskip
\textbf{Step 1} (Quasiconformal coordinate)
Note that $f$ has exactly $t_0$ immediate parabolic basins and the attracting petals $\{Q_t:1\leqslant t\leqslant t_0\}$ can be classified into $t_0'$ parts: $\{Q_{t'(i)}:1\leqslant i\leqslant p_{t'}\}$, where $1\leqslant t'\leqslant t_0'$ and each $Q_{t'(i)}$ attaches at the parabolic fixed point $\zeta_{t'}$.
It is known that the closures of the Siegel disks $\{\overline{\Delta}_r:1\leqslant r\leqslant r_0\}$ of $f$ are pairwise disjoint closed quasi-disks.
By the choice of the attracting petals $Q_t$ (see Corollary \ref{cor:petal}(b)), there exists a quasiconformal mapping $\Psi:\EC\to\EC$ such that the restriction
\begin{equation}\label{equ:Psi-restrict}
\Psi: \EC\setminus \Big(\bigcup_{r=1}^{r_0}\overline{\Delta}_r \cup \bigcup_{t=1}^{t_0}\overline{Q}_t \Big)
\to \EC\setminus \Big(\bigcup_{r=1}^{r_0}\overline{D}_r \cup \bigcup_{t'=1}^{t_0'}\overline{S}_{t'} \Big)
\end{equation}
is conformal, where
\begin{itemize}
\item $D_1=\EC\setminus\overline{\D}(0,R')$ for $R'>0$, $D_r=\D(a_r,\delta_r)$ for $a_r\in\C$ and $\delta_r>0$ with $2\leqslant r\leqslant r_0$;
\item $S_{t'}=\{z\in\C: 0<|z-b_{t'}|^{p_{t'}}<\delta' \text{ and } 0<\arg (z-b_{t'})^{p_{t'}}<\pi\}$ for $b_{t'}\in\C$ and $0<\delta'<1$ with $1\leqslant t'\leqslant t_0'$;
\item $\overline{D}_r$ with $1\leqslant r\leqslant r_0$ and $\overline{S}_{t'}$ with $1\leqslant t'\leqslant t_0'$ are pairwise disjoint.
\end{itemize}
Let $F:=\Psi\circ f\circ \Psi^{-1}:\EC\to\EC$. Denote
\begin{equation}
\MP(F):=\Psi(\MP(f)) \text{\quad and\quad} \MP_1(F):=\Psi(\MP_1(f)),
\end{equation}
where $\MP_1(f)$ is defined in \eqref{equ:MP-1}. The map $F$ is regarded as the quasiconformal model of $f$ and it has the advantage that serves a nice geometry near the boundaries of some domains.
In the following, for convenience, sometimes we don't distinguish the notations in the dynamical plane of $f$ and that of $F$. For example, we use $R_{\theta,n}$ (not $\Psi(R_{\theta,n})$) to denote the ray segment of $F$. It suffices to prove this lemma in the dynamical plane of $F$ when $\widetilde{R}_n^{ess}$ is disjoint with an open neighborhood (depending only on the given small $\varepsilon$) of $\bigcup_{r=1}^{r_0}\overline{D}_r \cup \bigcup_{t'=1}^{t_0'}\overline{S}_{t'}$.

\medskip
Let $W$ and $W_1$ be domains defined in \eqref{equ:W-W1}.
We assume that the corresponding domains $W$ and $W_1$ in the dynamical plane of $F$ have the following form:
\begin{equation}\label{equ:W-new}
W=\D(0,R')\setminus\Big(\MP(F)\cup\bigcup_{r=2}^{r_0}\overline{D}_r \cup \bigcup_{s=1}^{s_0}\overline{B}_s \cup \bigcup_{t'=1}^{t_0'}\overline{S}_{t'}\Big) \text{\quad and\quad} W_1=W\cup\MP_1(F),
\end{equation}
where each $B_s$ is a Jordan disk with smooth boundary satisfying $F(\overline{B}_s)\subset B_s$.
By \eqref{equ:Psi-restrict}, the restriction of $\Psi$ in the original $W_1$ is conformal. Hence the original orbifold $\MO_f=(W_1,\nu)$ induces a hyperbolic orbifold by conformal isomorphism $\Psi$ and we use a similar notation $\MO_F=(W_1,\nu)$, where $\sigma_{\MO_F}(z) |dz|$ is the corresponding orbifold metric.

\medskip
Let $\varepsilon\in(0,\delta'/2)$ be small. Then $\big\{\D(b_{t'},\varepsilon/2)\setminus \overline{S}_{t'}:1\leqslant t'\leqslant t_0'\big\}$ consists of $t_0$ components which are sectors. We denote them by
\begin{equation}
\{Z_t:1\leqslant t\leqslant t_0\}.
\end{equation}
For $1\leqslant r\leqslant r_0$, we define
\begin{equation}
Y_r:=\big\{z\in W: \dist_{\C}(z,\partial D_r)<\varepsilon/2\big\}.
\end{equation}
Decreasing $\varepsilon>0$ if necessary, we assume that
\begin{itemize}
\item $\Psi^{-1}(Z_t)$ is contained in a thin repelling sector of $f$ which is disjoint with $\MP(f)$, where $1\leqslant t\leqslant t_0$; and
\item $\overline{Y}_r$ with $1\leqslant r\leqslant r_0$ and $\overline{S}_{t'}$ with $1\leqslant t'\leqslant t_0'$ are pairwise disjoint.
\end{itemize}
For two continuous curves $\gamma_1$ and $\gamma_2$ in $W$ with the same end points\footnote{For convenience, we say that a continuous curve $\gamma:[0,1]\to\EC$ is contained in $W$ if $\gamma\big((0,1)\big)$ is.} (the end points are allowed in $W_1$), we say that $\gamma_1$ is \textit{homotopic} to $\gamma_2$ in $W$ relative to their end points and denote
\begin{equation}
\gamma_1\simeq \gamma_2 \text{\quad in~~} W.
\end{equation}

\medskip
\textbf{Step 2} (Inductive assumptions in parabolic basins)
We prove this lemma by induction. Without loss of generality, we only consider the case when the ray segments are contained in the closure of an \textit{immediate parabolic basin} $P_t$. The proof for the attracting basins is simpler and completely the same.
As mentioned just before this lemma, we may assume that $\varepsilon>0$ is small enough such that
\begin{itemize}
\item The length $l_\sigma \big(R_{\theta,1}\big)$ of the ray segment $R_{\theta,1}=\Lambda(\theta,[1,2])$ (may be degenerate, i.e., a singleton) with respect to the orbifold metric $\sigma_{\MO_F}(z) |dz|$ is uniformly bounded above (with respect to $\theta$) by a constant $C_0>0$; and
\item $R_{\theta,1}\cap Y_r=\emptyset$ for all $\theta\in\T$ and $1\leqslant r\leqslant r_0$.
\end{itemize}
Then there exists $C_0'\geqslant C_0>0$ such that for any $z_0:=\Lambda(\theta,1)$ and $z_1:=\Lambda(\theta,2)$ with $\theta\in\T$, there exists a continuous curve $\widetilde{R}_1=\widetilde{R}_1(\theta)\subset W$ satisfying
\begin{itemize}
\item $\widetilde{R}_1\simeq R_{\theta,1}$ in $W$;
\item $\widetilde{R}_1$ is the union of $3$ continuous curves (may be degenerate) $\widetilde{R}_1^{ini}$, $\widetilde{R}_1^{ess}$ and $\widetilde{R}_1^{end}$;
\item $\widetilde{R}_1^{ess}\cap Y_r=\emptyset$ for all $1\leqslant r\leqslant r_0$, $\widetilde{R}_1^{ess}\cap Z_t=\emptyset$ for all $1\leqslant t\leqslant t_0$ and $l_{\sigma}(\widetilde{R}_1^{ess})<C_0'$.
\end{itemize}
In particular,
\begin{itemize}
\item If $z_0, z_1\notin Z_t$ for any $t$, then $\widetilde{R}_1$ can be decomposed as $\widetilde{R}_1^{ini}=\emptyset$, $\widetilde{R}_1^{ess}=\widetilde{R}_1$ and $\widetilde{R}_1^{end}=\emptyset$;

\item If $z_0\in Z_t$ for some $t$ and $z_1\not\in Z_t$ for any $t$, then $\widetilde{R}_1$ can be decomposed as $\diam_{\C}(\widetilde{R}_1^{ini})<\varepsilon$ and $\widetilde{R}_1^{end}=\emptyset$;

\item If $z_0\notin Z_t$ for any $t$ and $z_1\in Z_t$ for some $t$, then $\widetilde{R}_1$ can be decomposed as $\widetilde{R}_1^{ini}=\emptyset$ and $\diam_{\C}(\widetilde{R}_1^{end})<\varepsilon$;

\item If $z_0\in Z_t$, $z_1\in Z_{t'}$ for some $t,t'$, then $\widetilde{R}_1$ can be decomposed as $\diam_{\C}(\widetilde{R}_1^{ini})<\varepsilon$ and $\diam_{\C}(\widetilde{R}_1^{end})<\varepsilon$.
\end{itemize}
Moreover, in all the above cases we have
\begin{equation}\label{c0}
\diam_{\C}(\widetilde{R}_1^{ini})<\varepsilon,\quad l_{\sigma}(\widetilde{R}_1^{ess})< C_0' \text{\quad and\quad}\diam_{\C}(\widetilde{R}_1^{end})<\varepsilon.
\end{equation}

Suppose there exists $n\geqslant 1$ such that
\begin{itemize}
  \item For any non-degenerate ray segment $R_{\theta,n}$ with $\theta\in \T$, there exists a continuous curve $\widetilde{R}_n\subset W$ such that $\widetilde{R}_n \simeq R_{\theta,n}$ in $W$, and that $\widetilde{R}_n$ can be written as the union of three continuous curves (may be degenerate) $\widetilde{R}_n^{ini}$, $\widetilde{R}_n^{ess}$ and $\widetilde{R}_n^{end}$;
  \item $\widetilde{R}_n^{ess}\cap Y_r=\emptyset$ for all $1\leqslant r\leqslant r_0$, $\widetilde{R}_n^{ess}\cap Z_t= \emptyset$ for all $1\leqslant t\leqslant t_0$, and $l_{\sigma}(\widetilde{R}_n^{ess})< C'$, where $C'=C'(\varepsilon)>0$ will be specified later; and
  \item $\diam_{\C}(\widetilde{R}_n^{ini})<\varepsilon$, $\diam_{\C}(\widetilde{R}_n^{end})<\varepsilon$.
\end{itemize}

\medskip
\textbf{Step 3} (Deformation of curves) Let $g_n$ be the inverse branch of $F$ which maps $R_{\theta_n,n}$ to $\Lambda(\theta_{n+1},[2,n+2])$, where $\theta_n,\theta_{n+1}\in\T$. Then we have
\begin{equation}
R_{\theta_{n+1},n+1}=R_{\theta_{n+1},1}\cup g_n(R_{\theta_n,n}).
\end{equation}
Note that $g_n$ can be extended analytically to any continuous curve in $W$ which is homotopic to $R_{\theta_n,n}$.
Therefore, for the continuous curve $\widetilde{R}_n$ obtained in the inductive assumption which satisfies $\widetilde{R}_n=\widetilde{R}_n^{ini}\cup\widetilde{R}_n^{ess}\cup\widetilde{R}_n^{end} \simeq R_{\theta_n,n}$ in $W$, $g_n(\widetilde{R}_n)$ is well-defined.
Let $\widetilde{R}_1^{ini}$, $\widetilde{R}_1^{ess}$ and $\widetilde{R}_1^{end}$ be the curves (may be degenerate) corresponding to $R_{\theta_{n+1},1}$ in Step 2.
Then
\begin{equation}
\gamma_{n+1}:=\widetilde{R}_1^{ini}\cup \widetilde{R}_1^{ess} \cup \widetilde{R}_1^{end} \cup g_n(\widetilde{R}_n^{ini}\cup\widetilde{R}_n^{ess}\cup\widetilde{R}_n^{end}) \simeq R_{\theta_{n+1},n+1} \text{\quad in } W.
\end{equation}
In the following we deform $\gamma_{n+1}$ in $W$ such that the inductive assumptions hold for the step $n+1$. See Figure \ref{Fig-homot}.

\begin{figure}[!htpb]
  \setlength{\unitlength}{1mm}
  \centering
  \includegraphics[width=0.8\textwidth]{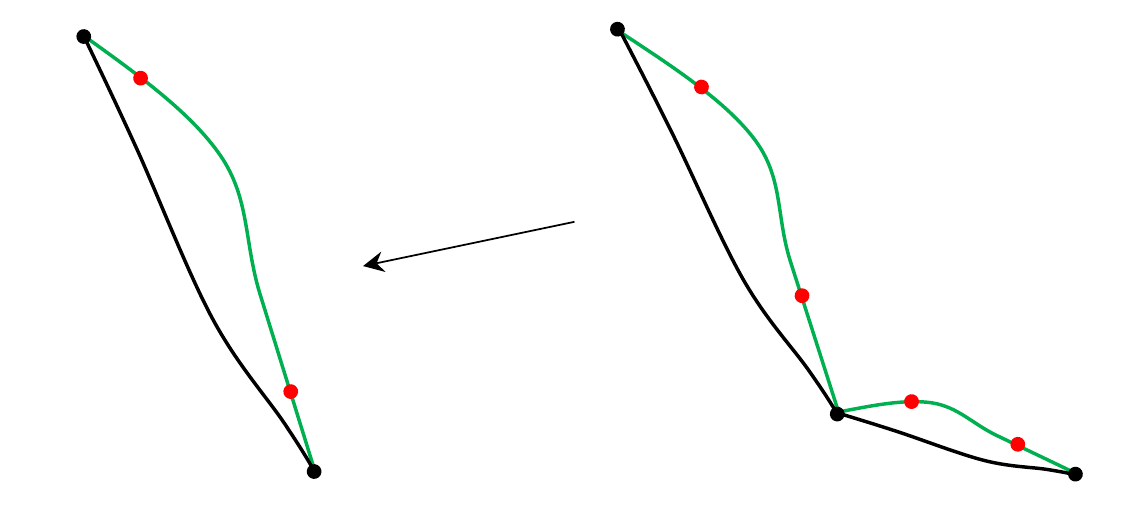}
  \put(-104,23){$R_{\theta_n,n}$}
  \put(-81.3,7.5){$\widetilde{R}_n^{ini}$}
  \put(-87.5,28.5){$\widetilde{R}_n^{ess}$}
  \put(-102.5,47.5){\small{$\widetilde{R}_n^{end}$}}
  \put(-68,29){$F$}
  \put(-55,23){$g_n(R_{\theta_n,n})$}
  \put(-30.3,18.5){\small{$g_n(\widetilde{R}_n^{ini})$}}
  \put(-34,33){$g_n(\widetilde{R}_n^{ess})$}
  \put(-46,47){\small{$g_n(\widetilde{R}_n^{end})$}}
  \put(-56,48){$z_n$}
  \put(-33.5,9){$z_1$}
  \put(-3.5,3.5){$z_0$}
  \put(-26,4){$R_{\theta_{n+1},1}$}
  \put(-8,7.5){$\widetilde{R}_1^{ini}$}
  \put(-17.5,11.5){$\widetilde{R}_1^{ess}$}
  \put(-28.5,12.5){\small{$\widetilde{R}_1^{end}$}}
  \caption{The mapping relation of some marked curves in the dynamical plane of $F=\Psi\circ f\circ \Psi^{-1}$. }
  \label{Fig-homot}
\end{figure}
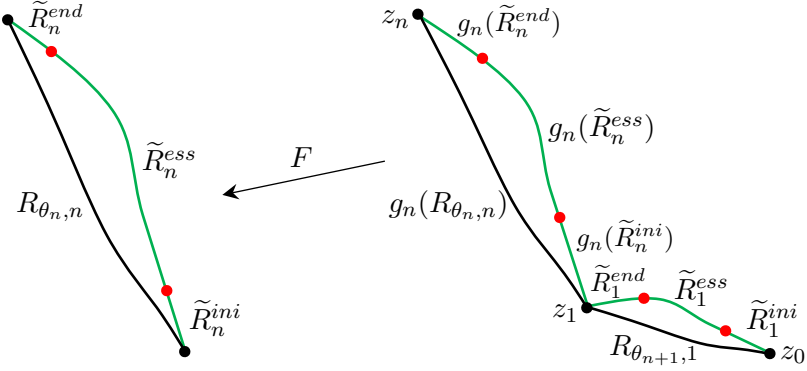

Let $z_0=\Lambda(\theta_{n+1},1)$ and $z_n=\Lambda(\theta_{n+1},n+2)$ be the two end points of $\gamma_{n+1}$. Denote $z_1=\Lambda(\theta_{n+1},2)$. Without loss of generality, we assume that $z_0$ and $z_1$ are not on the boundary of the immediate parabolic basin $P_t$. Otherwise, the inductive assumptions hold immediately.
By the inductive assumptions in the step $n$ and Lemma \ref{lem:expand-orbifold}, there exists $0<\mu=\mu(\varepsilon)<1$ which is independent of $n$ and a small number $\varepsilon'>0$ depending only on $\varepsilon$ (where $\varepsilon'\to 0$ as $\varepsilon\to 0$) such that
\begin{equation}\label{equ:l-W-mu}
l_{\sigma}\big(g_n(\widetilde{R}_n^{ess})\big)< \mu C',
\end{equation}
and
\begin{equation}\label{equ:end-bd}
\diam_{\C}\big(g_n(\widetilde{R}_n^{ini})\big)<\varepsilon' \text{\quad and\quad} \diam_{\C}\big(g_n(\widetilde{R}_n^{end})\big)<\varepsilon'.
\end{equation}

Denote $\Theta:=\bigcup_{r=1}^{r_0}Y_r \cup \bigcup_{t=1}^{t_0} Z_t$. We have the following $6$ cases:
\begin{itemize}
\item[(1)] $z_0\in Z_t$ for some $1\leqslant t\leqslant t_0$, and $z_n\in Y_r$ for some $1\leqslant r\leqslant r_0$;
\item[(2)] $z_0\in Z_t$ for some $1\leqslant t\leqslant t_0$, and $z_n\in Z_{\hat{t}}$ for some $1\leqslant \hat{t}\leqslant t_0$;
\item[(3)] $z_0\in Z_t$ for some $1\leqslant t\leqslant t_0$, and $z_n\not\in \Theta$;
\item[(4)] $z_0\not\in \Theta$, and $z_n\in Y_r$ for some $1\leqslant r\leqslant r_0$;
\item[(5)] $z_0\not\in \Theta$, and $z_n\in Z_{t}$ for some $1\leqslant t\leqslant t_0$; or
\item[(6)] $z_0\not\in \Theta$, and $z_n\not\in \Theta$.
\end{itemize}

Suppose Case (1) holds.
Let $z_*$ be the middle point of the arc $\partial Z_t\setminus \partial S_{t'}$, and $z^*$ be the point on $\partial Y_r\setminus \partial D_r$ such that the segment $[z^*, z_n]$ is perpendicular to $\partial Y_r$ (see Figure \ref{Fig-homo-seg}).
By \eqref{c0}, \eqref{equ:l-W-mu} and \eqref{equ:end-bd}, there exist a number $C_1=C_1(\varepsilon)>0$ and a continuous curve $\widetilde{\gamma}_{n+1}$ such that $\widetilde{\gamma}_{n+1}\simeq \gamma_{n+1}$ in $W$, and $\widetilde{\gamma}_{n+1}=[z_0,z_*]\cup\gamma_{n+1}'\cup[z^*,z_n]$, where $\gamma_{n+1}'$ is a continuous curve in $W\setminus\Theta$ with end points $z_*$ and $z^*$ satisfying
\begin{equation}
l_{\sigma}\big(\gamma_{n+1}'\big)\leqslant l_{\sigma}\big(g_n(\widetilde{R}_n^{ess})\big)+l_{\sigma}(\widetilde{R}_1^{ess})+C_1<\mu C'+C_0'+C_1.
\end{equation}
In this case we define $\widetilde{R}_{n+1}:=\widetilde{R}_{n+1}^{ini}\cup\widetilde{R}_{n+1}^{ess}\cup \widetilde{R}_{n+1}^{end}$, where
\begin{equation}
\widetilde{R}_{n+1}^{ini}:=[z_0,z_*],  \quad
\widetilde{R}_{n+1}^{ess}:=\gamma_{n+1}' \text{\quad and\quad} \widetilde{R}_{n+1}^{end}:=[z^*,z_n].
\end{equation}
Then the induction at the step $n+1$ is finished in this case by setting $C':=(C_0'+C_1)/(1-\mu)$.

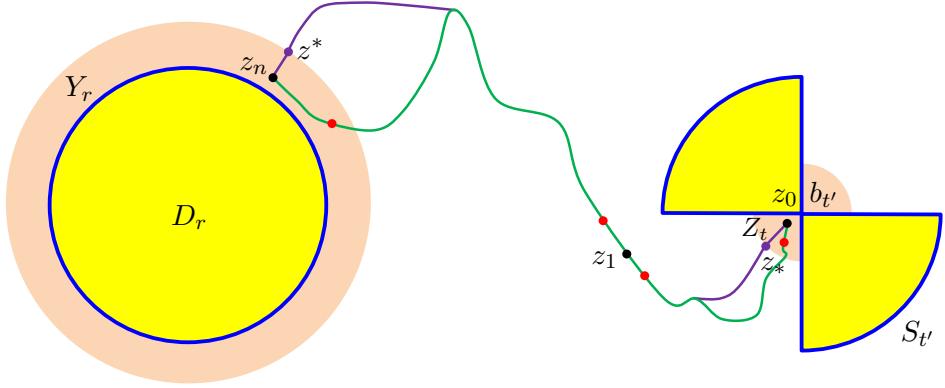
\begin{figure}[!htpb]
  \setlength{\unitlength}{1mm}
  \centering
  \includegraphics[width=0.9\textwidth]{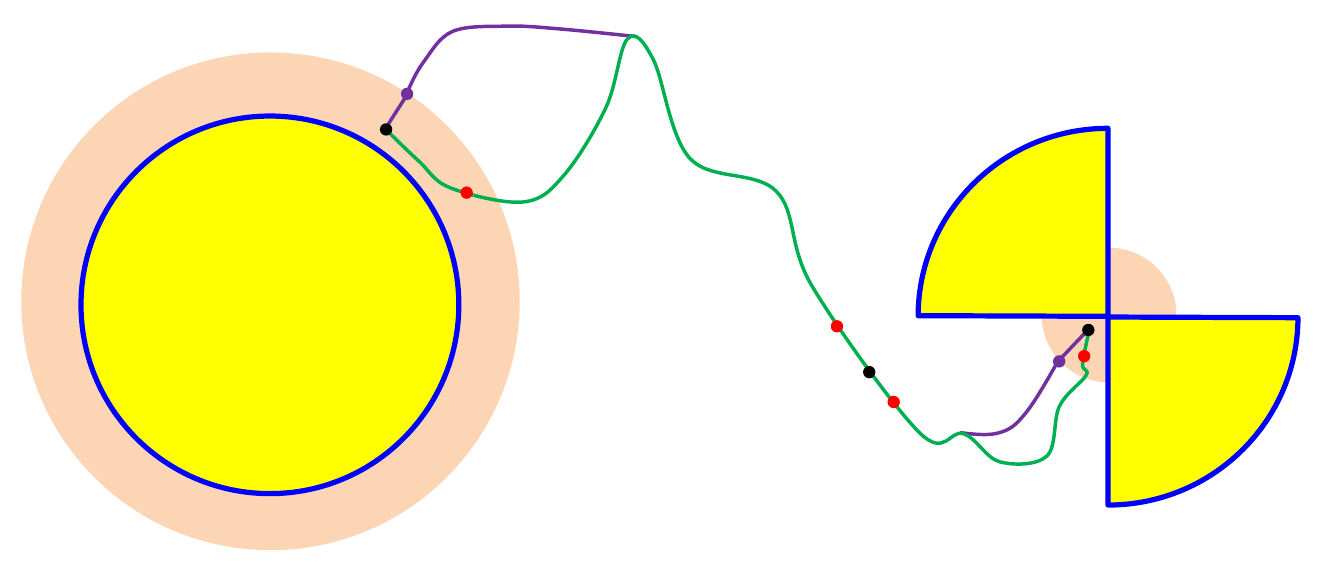}
  \put(-104,23){$D_r$}
  \put(-118,40){$Y_r$}
  \put(-95,43.5){$z_n$}
  \put(-87.5,45){$z^*$}
  \put(-48.7,17.8){$z_1$}
  \put(-29,21.5){\small{$Z_t$}}
  \put(-26.5,17){$z_*$}
  \put(-25,26){$z_0$}
  \put(-20,26){$b_{t'}$}
  \put(-8,7.5){$S_{t'}$}
  \caption{A homotopic deformation of the curve $\gamma_{n+1}$ (colored green). }
  \label{Fig-homo-seg}
\end{figure}

By a similar argument as Case (1), one can complete the induction at the step $n+1$ for Cases (2)-(6) similarly. The only difference is that one may need to define  $\widetilde{R}_{n+1}^{ini}$ or $\widetilde{R}_{n+1}^{end}$ as an empty set.

\medskip
\textbf{Step 4} (The conclusion)
To sum up, we have proved that for any $\theta\in \T$ and any $n\geqslant 1$, there exists a continuous curve $\widetilde{R}_n$ such that $\widetilde{R}_n\simeq R_{\theta,n}$ in $W$, where $\widetilde{R}_n$ consists of $3$ continuous curves (may be degenerate) $\widetilde{R}_n^{ini}$, $\widetilde{R}_n^{ess}$ and $\widetilde{R}_n^{end}$ in $W$ satisfying $\widetilde{R}_n^{ess}\cap \Theta= \emptyset$ and
\begin{equation}
\diam_{\C}(\widetilde{R}_n^{ini})<\varepsilon,\quad
l_{\sigma}(\widetilde{R}_n^{ess})< C' \text{\quad and\quad} \diam_{\C}(\widetilde{R}_n^{end})<\varepsilon.
\end{equation}
Note that there exists a number $\kappa>0$ such that the orbifold metric $\sigma_{\MO_F}(z)|dz|$ satisfies $\sigma_{\MO_F}(z)>\kappa>0$ for all $z\in W$. Hence we have $l_{\C}(\widetilde{R}_n^{ess})<C:=C'/\kappa$. The proof is complete.
\end{proof}

\begin{proof}[Proof of Proposition \ref{prop:attracting}]
We only consider the immediate parabolic basin since the proof for the attracting basin is completely similar. In fact, the argument is inspired by \cite[Proposition 5.2]{WYZZ22}.
Let $\Lambda$ be defined in \eqref{equ:Lambda}. Note that in the closure of an immediate parabolic basin $P_t$, each $\partial Q^{(n)}$ has a parameterization:
\begin{equation}
\Gamma_n(\theta):=\Lambda(\theta,n), \quad \text{where }\theta\in\T.
\end{equation}
Every ray segment $R_{\theta',n}$ connects $\Gamma_1(\theta')$ with $\Gamma_{n+1}(\theta')$, where $\theta'\in \T$ and $n\geqslant 1$.

For any given $\varepsilon>0$, let $C=C(\varepsilon)>0$ and $\widetilde{R}_n=\widetilde{R}_n^{ini}\cup\widetilde{R}_n^{ess}\cup\widetilde{R}_n^{end}\simeq R_{\theta',n}$ in $W$ be obtained in Lemma \ref{lem:essential-bd}. For any integer $m\geqslant 1$, the map
\begin{equation}
f^{\circ (m-1)}: \Lambda\big(\theta,[m,n+m]\big)\to R_{\theta',n}
\end{equation}
is a homeomorphism, where $\Lambda\big(\theta,[m,n+m]\big)$ is a continuous curve connecting $\Gamma_m(\theta)$ with $\Gamma_{n+m}(\theta)$ in $W$ (see \eqref{defi-ray-seg-para}). We use $\widetilde{R}_{n,m} $ to denote the pull back of $\widetilde{R}_n$ under $f^{\circ (m-1)}$ such that $\widetilde{R}_{n,m}\simeq \Lambda(\theta,[m,n+m])$ in $W$.

Let $\delta_0>0$ be the constant introduced in Lemma \ref{lem:contrac}. Since the Euclidean length of $\widetilde{R}_n^{ess}$ is less than $C=C(\varepsilon)$, it follows that $\widetilde{R}_n^{ess}$ can be divided into $M=M(\varepsilon)>2$ subcurves, such that each of them has spherical length less than $\delta_0/2$. Then each subcurve is contained in a Jordan domain $U_i\subset W_1$ with $\diam_{\EC}(U_i)<\delta_0$, where $1\leqslant i\leqslant M$. By Lemma \ref{lem:essential-bd}, without loss of generality, we assume that $l_{\EC}(\widetilde{R}_n^{ini})<\varepsilon<\delta_0/2$ and $l_{\EC}(\widetilde{R}_n^{end})<\varepsilon<\delta_0/2$.

By Lemma \ref{lem:contrac}, there exits $N\geqslant 1$ such that for all $m\geqslant N$, the spherical diameter of every component of $f^{-(m-1)}(U_i)$ and $f^{-(m-1)}(\widetilde{R}_n^{end})$ is less than $\varepsilon/(2M)$. Hence
\begin{equation}
\dist_{\EC}(\Gamma_m(\theta), \Gamma_{n+m}(\theta))\leqslant \diam_{\EC}(\widetilde{R}_{n,m} ) <(M+2)\cdot \frac{\varepsilon}{2M}<\varepsilon.
\end{equation}
Thus $\{\Gamma_n(\theta):\theta\in\T\}_{n\geqslant 1}$ forms a Cauchy sequence and it converges uniformly to a continuous curve $\Gamma(\theta): \T\to \overline{P}_t$. In particular, $\partial P_t=\Gamma(\T)$ is locally connected.
\end{proof}

\begin{rmk}
If $f$ has no Siegel disk, Proposition \ref{prop:attracting} gives another proof of the local connectivity of the Julia sets for some special parabolic rational maps, in particular, for geometrically finite rational maps.
\end{rmk}

\subsection{Proof of Theorem \ref{thm:local-connectivity}}

Based on Lemma \ref{lem:contrac}, Proposition \ref{prop:attracting} and the local connectivity criterion (Lemma \ref{lem:LC-criterion}), the proof of Theorem \ref{thm:local-connectivity} is now almost immediate (See also \cite[\S 5.2]{WYZZ22}). For completeness, we include a proof here.

\begin{proof}[Proof of Theorem \ref{thm:local-connectivity}]
Let $f$ be a rational map in Theorem \ref{thm:local-connectivity}. Iterating $f$ if necessary, we assume that every periodic Fatou component of $f$ has period one. According to \cite{Zha11} and Proposition \ref{prop:attracting}, we conclude that the boundaries of all Fatou components of $f$ are locally connected. By Lemma \ref{lem:LC-criterion}, it suffices to prove that the spherical diameter of the Fatou components of $f$ tends to zero.

Let $U_0$ be a fixed Fatou component of $f$. If $U_0$ is completely invariant ($U_0$ is either an immediate attracting or parabolic basin), by Proposition \ref{prop:attracting}, $J(f)=\partial U_0$ is locally connected. Suppose $U_0$ is not completely invariant ($U_0$ can be attracting, parabolic or Siegel). Then $f^{-1}(U_0)\setminus U_0$ consists of at least one but at most finitely Fatou components. We denote them by $\{U_1^{(1)},\cdots, U_1^{(i_1)}\}$ for some $i_1\geqslant 1$. There exists $k_0\geqslant 1$ such that for any $1\leqslant i\leqslant i_1$, any connected component $U_{k_0+1}^{(j)}$ of $f^{-k_0}(U_1^{(i)})$ is disjoint with $\MP(f)$. Thus for any integer $k\geqslant 1$, the map $f^{\circ k}:U_{k+k_0+1}^{(l)}\to U_{k_0+1}^{(j)}$ is conformal, where $U_{k+k_0+1}^{(l)}$ is any connected component of $f^{-k}(U_{k_0+1}^{(j)})$.

Let $\delta_0>0$ be the constant in Lemma \ref{lem:contrac}. There exists a uniform constant $M\geqslant 1$ such that any $U_{k_0+1}^{(j)}$ is covered by $M$ Jordan domains $\{D_m: 1\leqslant m\leqslant M\}$ in $W_1$ whose spherical diameters are less than $\delta_0$. For any given $\varepsilon>0$, by Lemma \ref{lem:contrac}, there exists $N\geqslant 1$ such that for all $k\geqslant N$, the spherical diameter of any connected component of $f^{-k}(D_m)$ is less than $\varepsilon/M$. Thus for all $k\geqslant N$, we have
\begin{equation}
\diam_{\EC}\big(U_{k+k_0+1}^{(l)}\big)<\varepsilon \text{\quad for all }l.
\end{equation}
Since $f$ has only finitely many periodic Fatou components, there are only finitely many Fatou components of $f$ whose spherical diameters are greater than $\varepsilon$.
\end{proof}

\section{Combinatorics of Siegels and parabolics}\label{sec:defi}

In this section, for the quadratic polynomials having a fixed bounded type Siegel disk or a parabolic point, we assign an address for each bounded Fatou component. Such combinatoric information will be used to study the ray equivalence later.

\subsection{Addresses of preimages of the Siegel disk}\label{f_theta}

Let $0<\theta<1$ be a bounded type irrational number. According to Douady-Herman (see \cite{Dou87}, \cite{Her87}), the quadratic polynomial
\begin{equation}
f_{\theta}:z\mapsto e^{2\pi\ii \theta}z+z^2
\end{equation}
has a Siegel disk $\Delta_\theta$ centered at $0$ whose boundary is quasi-circle containing the critical point $- e^{2\pi\ii \theta}/2$.
Petersen proved that the Julia set $J(f_\theta)$ of $f_\theta$ is locally connected and has zero Lebesgue measure \cite{Pet96} (see also \cite{Yam99}).
Note that every bounded Fatou component of $f_\theta$ is a Jordan disk, which is a preimage of $\Delta_\theta$. In this subsection, we assign addresses for these bounded Fatou components and some points on $J(f_\theta)$ by following \cite{Pet96} and \cite{YZ01}. A slight difference is that we don't use the Blaschke model where the unit disk corresponds to the Siegel disk.

\begin{defi}[Drops, centers and roots]
We call the Siegel disk $U_0:=\Delta_\theta$ the $0$-\emph{drop}\footnote{The name ``drop'' is very vivid when one looks the preimages of the unit disk under the Blaschke model of $f_\theta$ (see \cite{Pet96}). For convenience we still use this name here.} of $f_{\theta}$. Let $U_1:=f^{-1}_{\theta}(\Delta_\theta)\setminus\Delta_\theta$ be the unique $1$-drop, which is a Jordan domain attaching at the critical point $x_1:=- e^{2\pi\ii \theta}/2$. More generally, for any $n\geqslant 1$, each connected component $U$ of $f_{\theta}^{-(n-1)}(U_1)$ is a Jordan domain called an $n$-\emph{drop}, with $n$ the \textit{depth} of $U$. The map $f_{\theta}^{\circ n}:U\to \Delta_\theta$ is conformal.
The unique point $z=z(U)\in U$ satisfying $f_{\theta}^{\circ n}(z)=0$ is called the \emph{center} of $U$.
The unique point $x(U):=f_{\theta}^{-(n-1)}(x_1)\cap \partial U$ is called the \emph{root} of $U$.
\end{defi}

Let $U$ and $V$ be two different drops of depths $m$ and $n$, respectively. Then $\overline{U}\cap \overline{V}$ is either empty or a singleton. In the later case we necessarily have $m\neq n$. Suppose $m<n$. Then $\overline{U}\cap \overline{V}=x(V)$, which is a preimage of the critical point $x_1$. We call that $U$ is the \textit{parent} of $V$ and $V$ is a \textit{child} of $U$. In particular, every $n$-drop with $n\geqslant 1$ has a unique parent which is an $m$-drop with $0\leqslant m<n$. The root of this $n$-drop belongs to the boundary of its parent and the roots determine children uniquely.

Now we give a symbolic description to every drop by assigning an address. For $n\geqslant 1$, let $x_n= f_{\theta}^{-(n-1)}(x_1)\cap \partial\Delta_\theta$ and $U_n$ be the $n$-drop with root $x_n$. Consider any multi-index $\iota= \iota_1 \iota_2 \cdots \iota_k$ of length $k$, where each $\iota_j$ is a positive integer. We define the $(\iota_1 +\iota_2 +\cdots +\iota_k)$-drop $U_{\iota_1 \iota_2 \cdots \iota_k}$ of generation $k$ with root
\begin{equation}
x(U_{\iota_1 \iota_2 \cdots \iota_k})=x_{\iota_1 \iota_2 \cdots \iota_k}
\end{equation}
as follows. For $k=1$, the definition is done. For the induction step, suppose we have defined $x_{\iota_1 \iota_2 \cdots \iota_{k-1}}$ for all multi-indices $\iota_1 \iota_2 \cdots \iota_{k-1}$ of length $k-1$. Based on this we define
\begin{equation}\label{roots}
x_{\iota_1 \iota_2 \cdots \iota_k}:=
\left\{
\begin{array}{ll}
f_{\theta}^{-1}(x_{ \iota_2 \cdots \iota_k})\cap \partial U_{\iota_1 \iota_2 \cdots \iota_{k-1}}  & \text{if } \iota_1=1, \\
f_{\theta}^{-1}(x_{ (\iota_1 -1)\iota_2 \cdots \iota_k})\cap \partial U_{\iota_1 \iota_2 \cdots \iota_{k-1}} & \text{if } \iota_1>1.
\end{array}
\right.
\end{equation}
The first line of \eqref{roots} defines all the roots of the form $x_{1\iota_2 \cdots \iota_{k}}$ and the second line defines all the roots $x_{\iota_1 \iota_2 \cdots \iota_k}$ by induction on $\iota_1$.  The corresponding drops $U_{\iota_1 \iota_2 \cdots \iota_k}$ are determined by $x_{\iota_1\iota_2 \cdots \iota_{k}}$ and uniqueness.
Then all drops are given addresses and
\begin{equation}\label{roots-1}
f_{\theta}(U_{\iota_1 \iota_2 \cdots \iota_k})=
\left\{
\begin{array}{ll}
U_{ \iota_2 \cdots \iota_k}  & \text{if } \iota_1=1, \\
U_{ (\iota_1 -1)\iota_2 \cdots \iota_k} & \text{if } \iota_1>1.
\end{array}
\right.
\end{equation}

\begin{defi}[Limbs]
For any given multi-index $\iota_1 \iota_2 \cdots \iota_k$, define the \emph{limb} $L_{\iota_1 \iota_2 \cdots \iota_k}$ as the closure of the union of the drop $U_{\iota_1 \iota_2 \cdots \iota_k}$ and all its descendants (i.e., children and grandchildren, etc.):
\begin{equation}
L_{\iota_1 \iota_2 \cdots \iota_k}:=\overline{\bigcup U_{\iota_1 \iota_2 \cdots \iota_k\cdots}}\,.
\end{equation}
Note that the biggest limb $L_0=K(f_\theta)$ is the filled Julia set of $f_\theta$.
If $\iota_1 \iota_2 \cdots \iota_k\not= 0$, we call $x_{\iota_1 \iota_2 \cdots \iota_k}$ the root of $L_{\iota_1 \iota_2 \cdots \iota_k}$.
\end{defi}

According to \cite{Pet96}, the critical point $x_1$ of $f_{\theta}$ is the landing point of exactly two external rays in the basin of the infinity. Moreover, as the depth of a limb of $f_{\theta}$ goes to infinity, the corresponding Euclidean diameter goes to $0$. This implies that the Julia set of $f_{\theta}$ is locally connected and the filled Julia set $ K(f_{\theta})$ is the union of $\overline{\Delta}_\theta$ and all the limbs of generation $1$, i.e.,
$K(f_{\theta})=\overline{\Delta}_\theta\cup \bigcup_{n\geqslant 1} L_{n}$.

Consider a sequence of drops $\{U_0=\Delta_\theta, U_{\iota_1}$, $U_{\iota_1 \iota_2}$, $\cdots\}$, where each $U_{\iota_1 \iota_2 \cdots \iota_k} $  is the parent of $U_{\iota_1 \iota_2 \cdots \iota_{k+1}}$. The closure of the union of this sequence
\begin{equation}
\MC:=\overline{\bigcup\limits_k U_{\iota_1 \iota_2 \cdots \iota_k}}
\end{equation}
is a \emph{drop-chain}.
The intersection of the corresponding limbs $L_0\supset L_{\iota_1}\supset L_{\iota_1\iota_2}\supset\cdots$ must be a singleton which we denote by
\begin{equation}
\xi(\MC):=\bigcap\limits_k L_{\iota_1 \iota_2 \cdots \iota_k}.
\end{equation}
Note that we have
$\MC= \bigcup_k \overline{U}_{\iota_1 \iota_2 \cdots \iota_k}\cup \{\xi(\MC)\}$.

\begin{defi}[Drop-rays]
By a \emph{ray} in a drop $U$ we mean a hyperbolic geodesic which connects a boundary point $z_0\in \partial U$ to the center $z(U)$ of $U$.
We denote this ray by $\llbracket z_0, z(U) \rrbracket=\llbracket z(U),z_0 \rrbracket$.
For two different points $z_1,z_2\in\partial U$, we denote
\begin{equation}\label{equ:ray-2-pt}
\llbracket z_1,z_2 \rrbracket:=\llbracket z_1,z(U) \rrbracket \cup \llbracket z(U),z_2 \rrbracket.
\end{equation}
For any drop-chain $\MC=\overline{\bigcup_k U_{\iota_1 \iota_2 \cdots \iota_k}}$, the following path $R=R(\MC)$ is the ``most efficient" path in $\MC$ which connects $0$ to $\xi(\MC)$:
\begin{equation}
R(\MC):=\llbracket 0, x_{\iota_1} \rrbracket\cup \bigcup\limits_{k\geqslant 1}\llbracket x_{\iota_1\cdots\iota_k}, x_{\iota_1\cdots\iota_k\iota_{k+1}} \rrbracket\cup \{\xi(\MC)\}.
\end{equation}
We call $R(\MC)$ the \emph{drop-ray} associated with $\MC$, and say that $R(\MC)$ and $\MC$ \emph{land} at $\xi(\MC)$.
\end{defi}

According to \cite{Pet96} and the above definitions,
every point in $K(f_{\theta})$ either belongs to the closure of a drop or is the landing point of a unique drop-chain, and the assignment $\MC\mapsto \xi(\MC)$ is one-to-one, i.e., different drop-rays land at distinct points.

\begin{defi}[Spines]
Consider the two drop-chains
\begin{equation}
\MC_\theta:=\overline{U_0\cup U_1\cup U_{11}\cup \cdots}\,, \quad \MC_\theta':=\overline{U_0\cup U_2\cup U_{21}\cup \cdots}
\end{equation}
with $f_{\theta}(\MC_\theta')=\MC_\theta$. By definition, we see that $\MC_\theta$ and $\MC_\theta'$  land at the repelling fixed point $\xi(\MC_\theta)= \beta_\theta$ and $\xi(\MC_\theta')=\beta_\theta'$ respectively, where $f_\theta(\beta_\theta')=\beta_\theta$. The \emph{spine} of $f_{\theta}$ is defined as the union of the drop-rays
\begin{equation}\label{equ:spine-theta}
S_\theta:=R(\MC_\theta)\cup R(\MC_\theta').
\end{equation}
Notice that every point on $S_\theta$ which is not in the interior of $K(f_{\theta})$ is either one of the ends points $\beta_\theta$, $\beta_\theta'$, or a preimage of the critical point $x_1$.
\end{defi}

Since the Julia set $J(f_{\theta})$ is locally connected, there exists a \textit{Carath\'{e}odory loop}
\begin{equation}\label{equ:eta-theta}
\eta_{\theta}:\T\to J(f_{\theta})
\end{equation}
which conjugates the angle-doubling map to $f_{\theta}$ (see the definitions in \S\ref{subsec:mating-defi}). A point $z\in J(f_{\theta})$ is the landing point of an external ray with angle $t$ if and only if $\eta_{\theta}(t)=z$. There are two special points $\eta_{\theta}(0)=\beta_\theta$ and  $\eta_{\theta}(1/2)=\beta_\theta'$.
We define two connected subsets of the Julia set:
\begin{equation}\label{equ:J-0-1}
\begin{split}
J^0_{\theta}:=&\{z\in J(f_{\theta}):z=\eta_{\theta}(t)\ \textrm{for some}\ 0\leqslant t \leqslant 1/2 \}, \\
J^1_{\theta}:=&\{z\in J(f_{\theta}):z=\eta_{\theta}(t)\ \textrm{for some}\ 1/2\leqslant t \leqslant 1 \}.
\end{split}
\end{equation}

\begin{defi}[Itinerary]
For any $z\in J(f_{\theta})$, we define the \emph{itinerary} of $z$ with respect to the spine $S_{\theta}$ as: $\varepsilon=(\varepsilon_0,\varepsilon_1, \varepsilon_2,\cdots)$, where $\varepsilon_i\in\{0,1\}$ is determined by
\begin{equation}
z_i\in J^{\varepsilon_i}_{\theta}, \text{ where } z_0=z \text{ and } z_{i+1}=f_{\theta}(z_i), i\geqslant 0.
\end{equation}
If $z$ is a preimage of the repelling fixed point $\beta_\theta$ or is biaccessible, we assign two different itineraries to $z$.
\end{defi}

There are following three cases:

\medskip
\textbf{Case 1}. The orbit of $z$ never hits the spine $S_\theta$. Then there exists a unique angle $t$ with $z=\eta_\theta(t)$ and $t$ has the binary expansion $0.\varepsilon_0\varepsilon_1\varepsilon_2\cdots$.

\medskip
\textbf{Case 2}. The orbit of $z$ eventually hits the fixed point $\beta_\theta$.  The two itineraries of $\beta_\theta$ are given by:
    \begin{equation}
        \varepsilon(\beta_\theta)=(0,0,0,0,\cdots) \text{\quad and\quad}
        \varepsilon'(\beta_\theta)=(1,1,1,1,\cdots).
    \end{equation}
     Similarly, the two itineraries of $\beta_\theta'$ (recall that $f_\theta^{-1}(\beta_\theta)=\{\beta_\theta,\beta_\theta'\}$) are given by:
    \begin{equation}
        \varepsilon(\beta_\theta')=(1,0,0,0,\cdots) \text{\quad and\quad}
        \varepsilon'(\beta_\theta')=(0,1,1,1,\cdots).
    \end{equation}
    The two itineraries of $z$ with $z_n=\beta_\theta$ (where $n\geqslant 2$ is the smallest integer satisfying this property) will have the form
        \begin{equation}
        \varepsilon=(\varepsilon_0,\varepsilon_1,\cdots, \varepsilon_{n-2}, 1,0,0,0,\cdots)  \text{\quad and\quad}
        \varepsilon'=(\varepsilon_0,\varepsilon_1, \cdots, \varepsilon_{n-2}, 0,1,1,1,\cdots).
    \end{equation}
In this case, there exists a unique angle $t$ with $z=\eta_\theta(t)$ and $t$ has two binary expansions $0.\varepsilon_0\varepsilon_1\cdots\varepsilon_{n-2}1000\cdots$ and $0.\varepsilon_0\varepsilon_1\cdots\varepsilon_{n-2}0111\cdots$.

\medskip
\textbf{Case 3}. The orbit of $z$ eventually hits the critical point $x_1$. For such points, there are exactly two angles $0<t<s<1$ with $\eta_{\theta}(t)=\eta_{\theta}(s)=z$. Note that $f_\theta(x_1)$ is the critical value having a unique external ray landing on it with the angle of binary expansion $\omega=0.\omega_1\omega_2\cdots$. Hence if $z=x_{11\cdots1}\in f^{-(k-1)}_{\theta}(x_1)$ with $k\geqslant 1$, the two itineraries of $z$ are
    \begin{equation}
        \varepsilon=(\underbrace{0,0,\cdots,0}_{k \ \textrm{times}}, \omega_1,\omega_2,\cdots)  \text{\quad and\quad}
        \varepsilon'=(\underbrace{1,1,\cdots,1}_{k \ \textrm{times}}, \omega_1,\omega_2,\cdots).
    \end{equation}
    If $z=x_{21\cdots1}\in f^{-k}_{\theta}(x_1)$ with $k\geqslant 1$, the two itineraries of $z$ are
    \begin{equation}
        \varepsilon=(0,\underbrace{1,\cdots,1}_{k \ \textrm{times}}, \omega_1,\omega_2,\cdots)  \text{\quad and\quad}
        \varepsilon'=(1,\underbrace{0,\cdots,0}_{k \ \textrm{times}}, \omega_1,\omega_2,\cdots).
    \end{equation}
Moreover, if $z$ is off the spine whose orbit eventually hits the critical point $x_1$, the itineraries $\varepsilon$ and $\varepsilon'$ can be defined similarly based on those of $x_{11\cdots1}$ and $x_{21\cdots1}$, and these itineraries give the binary expansions of the two angles $t$ and $s$ respectively.

\medskip
By the above definitions, we have the following characterizations of the points on the Julia set $J(f_\theta)$ by itineraries and angles of external rays (see \cite[\S 7.1]{YZ01}):

\begin{prop}\label{prop:YZ-Siegel}
The angle(s) of the external ray(s) landing at any given $z\in J(f_\theta)$ is (are) determined by the itinerary(ies) of $z$, that is, two points in $J(f_\theta)$ having the same itinerary must coincide. Moreover, every infinite sequence of $0$'s and $1$'s can be realized as the itinerary of a unique point in $J(f_{\theta})$.
\end{prop}

\subsection{Addresses of parabolic Fatou components}\label{f_nu}

In this subsection, we give some similar definitions (drops, limbs, spines, itineraries etc.) as before for the quadratic parabolic polynomial:
\begin{equation}
f_\nu: z\mapsto e^{2\pi\ii \nu}z+z^2, \text{\quad where } \nu=q/p\in\Q.
\end{equation}
The combinatorial information is complicated than the Siegel case since some point may have two preimages on the boundary of the same Fatou component.
For the special case $\nu=0$, it is well known that the Julia set of $f_0(z)=z+z^2$ is a Jordan curve.
In the rest of this subsection, we assume that $\nu=q/p\in (0,1)$, where $p\geqslant 2$ and $q\geqslant 1$ are coprime integers.
According to Douady-Hubbard \cite[Chap.\,10]{DH85a} (see also \cite{TY96}), the Julia set of $f_\nu$ is connected and locally connected, and each bounded Fatou component is a Jordan domain.
As the Siegel case, we assign addresses for the bounded Fatou components of $f_\nu$ and some points on $J(f_\nu)$.

\medskip
Note that there are exactly $p\geqslant 2$ bounded Fatou components attaching at the parabolic fixed point $y_0:=0$, which are exactly $p$ immediate parabolic basins of $f_\nu$.
Let $U_0^{(0)}$ be the one containing the critical point $c_\nu=- e^{2\pi\ii \nu}/2$. We use $U_0^{(1)}$, $U_0^{(2)}$, $\cdots$, $U_0^{(p-1)}$ to denote the remaining $p-1$ ones which satisfy the following relationship:
\begin{equation}
f_{\nu}(U_0^{(i)})=U_0^{(i-1)}, \text{\quad where } 1\leqslant i\leqslant p \text{ and } U_0^{(p)}:=U_0^{(0)}.
\end{equation}
In fact, $\nu=q/p$ is the ``combinatoric" rotation number of these $p$ immediate parabolic basins. They surround $0$ with an order which is uniquely determined by $\nu$. For example, if $\nu=3/5$ (i.e., $q=3$ and $p=5$), then the immediate parabolic basins surround $0$ as $U_0^{(0)}$, $U_0^{(3)}$, $U_0^{(1)}$, $U_0^{(4)}$, $U_0^{(2)}$ in counterclockwise. See Figure \ref{Fig_parabolic-zoom}.

\begin{defi}[Drops, centers and roots]
We call $U_0^{(0)}$, $U_0^{(1)}$, $\cdots$, $U_0^{(p-1)}$ the $0$-\emph{drops} of $f_{\nu}$.
Denote $\MU_0:=\bigcup_{i=0}^{p-1} U_0^{(i)}$. The connected components of $f_{\nu}^{-1}(\MU_0)\setminus\MU_0$ are called $1$-drops, which consist of $p-1$ Jordan domains attaching at the other preimage $y_1\neq y_0$ of $y_0$. We denote them by $U_1^{(1)}$, $U_1^{(2)}$, $\cdots$, $U_1^{(p-1)}$, where
\begin{equation}
f_{\nu}(U_1^{(i)})=U_0^{(i-1)} \text{\quad for } 1\leqslant i\leqslant p-1.
\end{equation}
More generally, for any $n\geqslant 1$ and $1\leqslant i\leqslant p-1$, each connected component $U$ of $f_{\nu}^{-(n-1)}(U_1^{(i)})$ is a Jordan domain called an $n$-\emph{drop}, with $n$ the \textit{depth} of $U$. The maps
\begin{equation}
f_{\nu}^{\circ (n-1)}:U\to U_1^{(i)} \text{\quad and\quad} f_{\nu}^{\circ i}:U_1^{(i)}\to U_0^{(0)}
\end{equation}
are conformal.
The unique point $z=z(U)\in U$ satisfying $f_{\nu}^{\circ (n-1+i)}(z)=c_\nu$ is the \emph{center} of $U$.
The unique point $y(U):=f_{\nu}^{-(n-1+i)}(y_0)\cap \partial U$ is the \emph{root} of $U$.
\end{defi}

Let $U$ and $V$ be two different drops of depths $m$ and $n$, respectively. Then $\overline{U}\cap \overline{V}$ is either empty or a singleton. In the later case we have the following cases:
\begin{itemize}
    \item $m=n$ and $U$, $V$ have the same root $y(U)=y(V)$;
    \item $m\neq n$ and $\overline{U}\cap \overline{V}=y(V)$ (we assume that $m<n$ without loss of generality).
\end{itemize}
In the second case, we call that $U$ is the \textit{parent} of $V$ and $V$ is a \textit{child} of $U$. In particular, every $n$-drop with $n\geqslant 1$ has a unique parent which is an $m$-drop with $0\leqslant m<n$. The root of this $n$-drop belongs to the boundary of its parent and it determines exactly $p-1$ children.

\medskip
Now we give a symbolic description to every drop by assigning an address. To this end we first give addresses to all roots on the boundary of $\MU_0$.
Note that both $y_0=0$ and $y_1=f_\nu^{-1}(y_0)\setminus\{y_0\}$ are contained in $\partial U_0^{(0)}$. We denote
\begin{equation}
    y_i:=f_{\nu}^{-(i-1)}(y_1)\cap \partial U_0^{(i-1)}, \text{\quad where } 2\leqslant i\leqslant p.
\end{equation}
Since $f_{\nu}: U_0^{(0)}=U_0^{(p)}\to U_0^{(p-1)}$ is of degree $2$ and $U_0^{(0)}$, $U_0^{(p-1)}$ are Jordan domains, each point on $\partial U_0^{(p-1)}$ has exactly two preimages on $\partial U_0^{(0)}$ with the same depth. Indeed, $\partial U_0^{(0)}$ is the union of two orientated simple arcs $\partial_0 U_0^{(0)}$ and $\partial_1 U_0^{(0)}$ such that
\begin{itemize}
\item $\partial_0 U_0^{(0)}\cap\partial_1 U_0^{(0)}=\{y_0,y_1\}$, $\partial_0 U_0^{(0)}$ starts at $y_0$ and ends at $y_1$, $\partial_1 U_0^{(0)}$ starts at $y_1$ and ends at $y_0$;
\item $\partial_0 U_0^{(0)}\cup\partial_1 U_0^{(0)}$ is a Jordan curve having a counterclockwise direction; and
\item $f_\nu:\partial_m U_0^{(0)}\setminus\{y_0,y_1\}\to \partial U_0^{(p-1)}\setminus\{y_0\}$ is a homeomorphism for $m\in\{0,1\}$.
\end{itemize}
For the two preimages of $y_p\in\partial U_0^{(p-1)}$ on $\partial U_0^{(0)}$ under $f_\nu$, we denote them by $y_{p+1}^0\in \partial_0 U_0^{(0)}$ and $y_{p+1}^1\in \partial_1 U_0^{(0)}$. Hence the $4$ points $y_0$, $y_{p+1}^0$, $y_1$ and $y_{p+1}^1$ are located in counterclockwise order on $\partial U_0^{(0)}$.
For $m\in\{0,1\}$, we denote
\begin{equation}
y_{p+i}^m:=f_{\nu}^{-(i-1)}(y_{p+1}^m)\cap \partial U_0^{(i-1)}, \text{\quad where } 2\leqslant i\leqslant p.
\end{equation}
For $m\in\{0,1\}$, $y_{2p}^m$ has two preimages on $\partial U_0^{(0)}=\partial_0 U_0^{(0)}\cup\partial_1 U_0^{(0)}$ which we denote by
\begin{equation}
f_\nu^{-1}(y_{2p}^m)=\big\{y^{0m}_{2p+1}, y^{1m}_{2p+1}\big\},
\end{equation}
where $y^{0m}_{2p+1}\in \partial_{0} U_0^{(0)}$ and $y^{1m}_{2p+1}\in \partial_{1} U_0^{(0)}$. See Figure \ref{Fig_parabolic-zoom}.

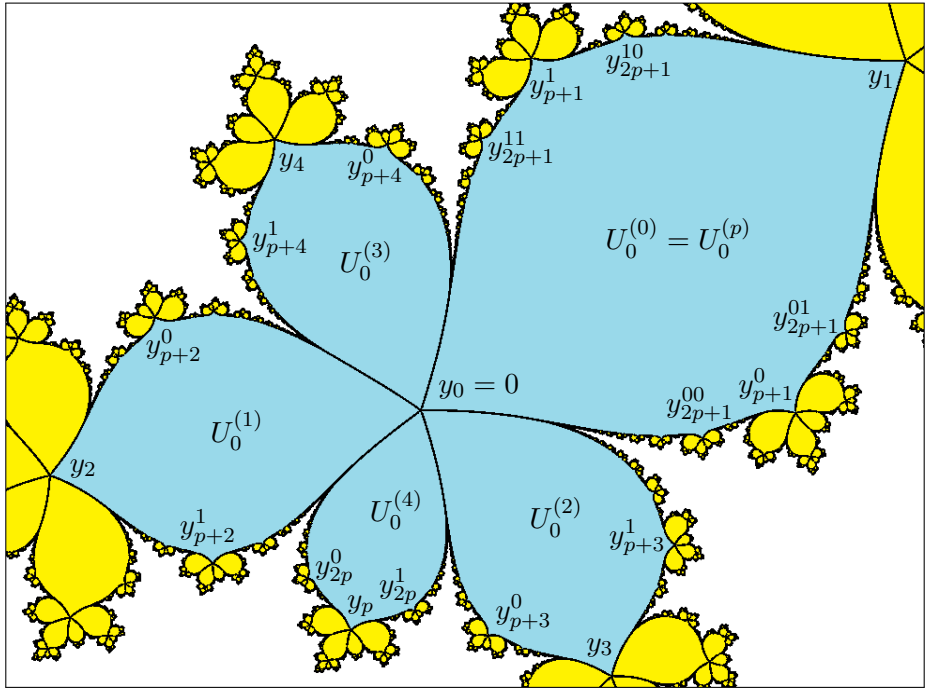
\begin{figure}[!htpb]
  \setlength{\unitlength}{1mm}
  \setlength{\fboxsep}{0pt}
  \centering
  \fbox{\includegraphics[width=0.85\textwidth]{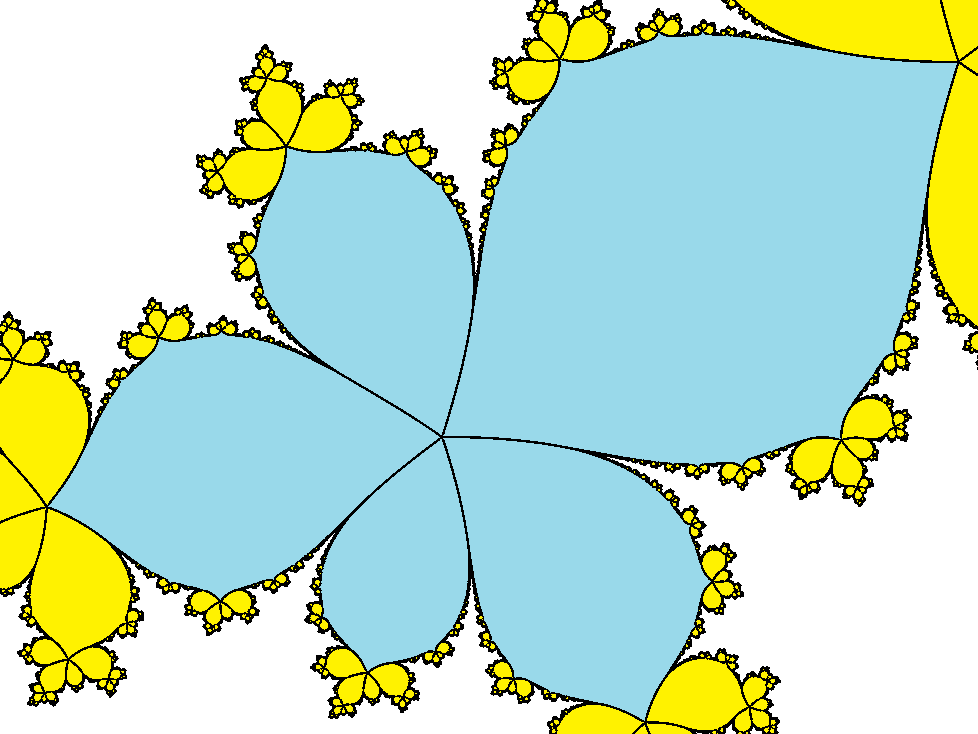}}
  \put(-64,39){$y_0=0$}
  \put(-7.5,80){$y_1$}
  \put(-112.5,28.5){$y_2$}
  \put(-44.5,4.8){$y_3$}
  \put(-85,69){$y_4$}
  \put(-76,10.5){$y_p$}
  \put(-24.5,39){$y_{p+1}^0$}
  \put(-52,78.8){$y_{p+1}^1$}
  \put(-34,36.5){$y_{2p+1}^{00}$}
  \put(-20,48){$y_{2p+1}^{01}$}
  \put(-42,82){$y_{2p+1}^{10}$}
  \put(-57.5,70.5){$y_{2p+1}^{11}$}
  \put(-102.5,44.2){$y_{p+2}^0$}
  \put(-98,20.3){$y_{p+2}^1$}
  \put(-56.5,9){$y_{p+3}^0$}
  \put(-41.5,19){$y_{p+3}^1$}
  \put(-76,68){$y_{p+4}^0$}
  \put(-88.5,58.5){$y_{p+4}^1$}
  \put(-80.2,15){$y_{2p}^0$}
  \put(-71.8,12.5){$y_{2p}^1$}
  \put(-42,58){$U_0^{(0)}=U_0^{(p)}$}
  \put(-73,22){$U_0^{(4)}$}
  \put(-77,55){$U_0^{(3)}$}
  \put(-52,21){$U_0^{(2)}$}
  \put(-94,32.5){$U_0^{(1)}$}
  \caption{A local view of the dynamical plane of $f_\nu$ near the immediate parabolic basins (colored cyan), where $\nu=3/5$ (i.e., $q=3$ and $p=5$). The parabolic fixed point $y_0=0$ and some of its preimages are marked.}
  \label{Fig_parabolic-zoom}
\end{figure}

\medskip
Let $\mm= m_1 m_2 \cdots m_\ell$ be any multi-index of length $\ell\geqslant 1$, where $m_j\in\{0,1\}$. If $\ell=0$, let $\mm:=\emptyset$ for convenience. We define the point
\begin{equation}
y_n^{\mm}=y_n^{m_1 m_2 \cdots m_\ell} \text{\quad with\quad} \ell p+1\leqslant n\leqslant (\ell+1)p
\end{equation}
inductively as follows. For $\ell=0,1$, the definitions are done. In particular, if $\ell=0$, we denote $y_n:=y_n^{\mm}$ for simplicity, where $1\leqslant n\leqslant p$. For the induction step, suppose we have defined $y_{n'}^{m_1 m_2 \cdots m_{\ell-1}}$ with $(\ell-1)p+1\leqslant n'\leqslant \ell p$ for all multi-indices $m_1 m_2 \cdots m_{\ell-1}$ of length $\ell-1$ with some $\ell\geqslant 2$. Based on this we define
\begin{equation}\label{equ:y-n}
y_{n}^{m_1 m_2 \cdots m_\ell}:=
\left\{
\begin{array}{ll}
f_{\nu}^{-1}(y_{n-1}^{m_2 \cdots m_\ell})\cap \partial_{m_1} U_0^{(0)}  & \text{if } n=\ell p+1, \\
f_{\nu}^{-1}(y_{n-1}^{m_1 m_2 \cdots m_\ell})\cap \partial U_0^{(i-1)} & \text{if } n=\ell p+i \text{ for } 2\leqslant i\leqslant p.
\end{array}
\right.
\end{equation}
The first line of \eqref{equ:y-n} defines all the roots of the form $y_{\ell p+1}^{m_1 m_2 \cdots m_\ell}$ and the second line defines all the roots $y_{\ell p+i}^{m_1 m_2 \cdots m_\ell}$ by induction on $i$.

\medskip
For $\ell\geqslant 0$ and $\ell p+1\leqslant n\leqslant (\ell+1)p$, let $\mm= m_1 m_2 \cdots m_\ell$ with $m_j\in\{0,1\}$ for all $1\leqslant j\leqslant \ell$ and $\{U_n^{\mm,(i)}:1\leqslant i\leqslant p-1\}$ be the $n$-drops with root $y_n^{\mm}$, where
\begin{equation}\label{equ:f-nu-1}
f_\nu^{\circ (n-1)}: U_n^{\mm,(i)}\to U_1^{(i)} \text{\quad and\quad}
f_\nu: U_1^{(i)}\to U_0^{(i-1)}
\end{equation}
are conformal. Note that if $\ell=0$, then $\mm=\emptyset$ and $y_n^{\mm}=y_n$. In this case we denote $U_n^{(i)}:=U_n^{\mm,(i)}$ for simplicity, where $1\leqslant n\leqslant p$ and $1\leqslant i\leqslant p-1$.

\begin{defi}[Admissible multi-indices]
For each $k\geqslant 1$, let $\Lambda_k$ be the collection of all pairs of multi-indices $\iota= \iota_1 \iota_2 \cdots \iota_k$ and $\mm=(\mm^1,\mm^2,\cdots,\mm^k)$ of length $k$ such that
\begin{itemize}
\item For each $1\leqslant s\leqslant k$,
\begin{equation}
\mm^s=m_1^s m_2^s \cdots m_{\ell_s}^s,
\end{equation}
where $m_j^s\in\{0,1\}$ for all $1\leqslant j\leqslant \ell_s$ and $\ell_s\geqslant 0$; and

\item $\ell_1 p+1\leqslant \iota_1\leqslant (\ell_1+1)p$ and $\ell_s p+1\leqslant \iota_s<(\ell_s+1)p$ for each $2\leqslant s\leqslant k$.
\end{itemize}
In particular, if $\ell_s=0$ for some $1\leqslant s\leqslant k$, then $\mm^s=\emptyset$, $1\leqslant \iota_1\leqslant p$ if $s=1$ and $1\leqslant \iota_s< p$ if $2\leqslant s\leqslant k$.
All pairs of $\iota$ and $\mm$ in $\Lambda_k$ are called \textit{admissible} multi-indices.
\end{defi}

For any pair of admissible multi-indices $\iota= \iota_1 \iota_2 \cdots \iota_k$ and $\mm=(\mm^1,\mm^2,\cdots,\mm^k)$ in $\Lambda_k$,
we define the $(\iota_1 +\iota_2 +\cdots +\iota_k)$-drops $\{U_{\iota_1 \iota_2 \cdots \iota_k}^{\mm^1,\mm^2,\cdots,\mm^k,(i_k)}:1\leqslant i_k\leqslant p-1\}$ of generation $k$ with root
\begin{equation}
y(U_{\iota_1 \iota_2 \cdots \iota_k}^{\mm^1,\mm^2,\cdots,\mm^k,(i_k)})=y_{\iota_1 \iota_2 \cdots \iota_k}^{\mm^1,\mm^2,\cdots,\mm^k}
\end{equation}
as follows. For $k=1$, the definition is done. For the induction step, suppose we have defined $y_{\iota_1 \iota_2 \cdots \iota_{k-1}}^{\mm^1,\mm^2,\cdots,\mm^{k-1}}$ for all pairs of admissible multi-indices $\iota_1 \iota_2 \cdots \iota_{k-1}$ and $(\mm^1,\mm^2,\cdots,\mm^{k-1})$ of length $k-1$ in $\Lambda_{k-1}$ with $k\geqslant 2$.
The corresponding drops $\{U_{\iota_1 \iota_2 \cdots \iota_{k-1}}^{\mm^1,\mm^2,\cdots,\mm^{k-1},(i_{k-1})}:1\leqslant i_{k-1}\leqslant p-1\}$ have the same root $y_{\iota_1 \iota_2 \cdots \iota_{k-1}}^{\mm^1,\mm^2,\cdots,\mm^{k-1}}$ such that the following map is conformal for each $1\leqslant i_{k-1}\leqslant p-1$:
\begin{equation}
f_\nu^{\circ (\iota_1 +\iota_2 +\cdots +\iota_{k-1})}: U_{\iota_1 \iota_2 \cdots \iota_{k-1}}^{\mm^1,\mm^2,\cdots,\mm^{k-1},(i_{k-1})}\to U_0^{(i_{k-1}-1)}.
\end{equation}

Based on this, for  any pair of admissible multi-indices $\iota= \iota_1 \iota_2 \cdots \iota_k$ and $\mm=(\mm^1,\mm^2,\cdots,\mm^k)$ in $\Lambda_k$, we have (note that in particular, $\ell_1 p+1\leqslant \iota_1\leqslant (\ell_1+1)p$ with $\ell_1\geqslant 0$ and $\mm^1=m_1^1 m_2^1 \cdots m_{\ell_1}^1$):
\begin{enumerate}
\item If $\iota_1=1$, then $\ell_1=0$, $\mm^1=\emptyset$ and we define
\begin{equation}
y_{\iota_1 \iota_2 \cdots \iota_k}^{\mm^1,\mm^2,\cdots,\mm^k}:=
f_{\nu}^{-1}\big(y_{\iota_2 \cdots \iota_k}^{\mm^2,\cdots,\mm^k}\big)\cap \partial U_{\iota_1 \iota_2 \cdots \iota_{k-1}}^{\mm^1,\mm^2,\cdots,\mm^{k-1},(i_{k-1})},
\end{equation}
where $\{1,2,\cdots, p-1\}\ni i_{k-1}\equiv\iota_k$ ($\Mod p$).

\item  If $\iota_1=\ell_1 p+i$ with $2\leqslant i\leqslant p$ and $\ell_1\geqslant 0$, we define
\begin{equation}
y_{\iota_1 \iota_2 \cdots \iota_k}^{\mm^1,\mm^2,\cdots,\mm^k}:=
f_{\nu}^{-1}\big(y_{(\iota_1-1) \iota_2 \cdots \iota_k}^{\mm^1,\mm^2,\cdots,\mm^k}\big)\cap \partial U_{\iota_1 \iota_2 \cdots \iota_{k-1}}^{\mm^1,\mm^2,\cdots,\mm^{k-1},(i_{k-1})},
\end{equation}
where $\{1,2,\cdots, p-1\}\ni i_{k-1}\equiv\iota_k$ ($\Mod p$).

\item  If $\iota_1=\ell_1 p+1$ with $\ell_1\geqslant 1$, we define
\begin{equation}
y_{\iota_1 \iota_2 \cdots \iota_k}^{\mm^1,\mm^2,\cdots,\mm^k}:=
f_{\nu}^{-1}\big(y_{(\iota_1-1) \iota_2 \cdots \iota_k}^{\widetilde{\mm}^1,\mm^2,\cdots,\mm^k}\big)\cap \partial U_{\iota_1 \iota_2 \cdots \iota_{k-1}}^{\mm^1,\mm^2,\cdots,\mm^{k-1},(i_{k-1})},
\end{equation}
where $\{1,2,\cdots, p-1\}\ni i_{k-1}\equiv\iota_k$ ($\Mod p$) and
\begin{equation}\label{equ:m-tilde}
\widetilde{\mm}^1=
\left\{
\begin{array}{ll}
\emptyset  & \text{if } \ell_1=1, \\
m_2^1\cdots m_{\ell_1}^1 & \text{if } \ell_1>1.
\end{array}
\right.
\end{equation}
\end{enumerate}
The above first case defines all the roots of the form $y_{1 \iota_2 \cdots \iota_k}^{\emptyset,\mm^2,\cdots,\mm^k}$ and the above latter two cases define all the roots $y_{\iota_1 \iota_2 \cdots \iota_k}^{\mm^1,\mm^2,\cdots,\mm^k}$ by induction on $\iota_1$.
The corresponding drops $\{U_{\iota_1 \iota_2 \cdots \iota_k}^{\mm^1,\mm^2,\cdots,\mm^k,(i_k)}:1\leqslant i_k\leqslant p-1\}$ have the same root $y_{\iota_1 \iota_2 \cdots \iota_k}^{\mm^1,\mm^2,\cdots,\mm^k}$ such that the following map is conformal:
\begin{equation}
f_\nu^{\circ (\iota_1 +\iota_2 +\cdots +\iota_k)}: U_{\iota_1 \iota_2 \cdots \iota_k}^{\mm^1,\mm^2,\cdots,\mm^k,(i_k)}\to U_{0}^{(i_k-1)}.
\end{equation}
Then all drops are given addresses and for each $k\geqslant 2$, or $k=1$ and $\iota_1>1$, we have
\begin{equation}\label{equ:roots-1-para}
f_{\nu}(U_{\iota_1 \iota_2 \cdots \iota_k}^{\mm^1,\mm^2,\cdots,\mm^k,(i_k)})=
\left\{
\begin{array}{ll}
U_{ \iota_2 \cdots \iota_k}^{\mm^2,\cdots,\mm^k,(i_k)}  & \text{if } \iota_1=1, \\
U_{ (\iota_1 -1)\iota_2 \cdots \iota_k}^{\mm^1,\mm^2,\cdots,\mm^k,(i_k)} & \text{if } \iota_1\neq\ell_1 p+1, \\
U_{ (\iota_1 -1)\iota_2 \cdots \iota_k}^{\widetilde{\mm}^1,\mm^2,\cdots,\mm^k,(i_k)} & \text{if } \iota_1=\ell_1 p+1 \text{ with }\ell_1\geqslant 1,
\end{array}
\right.
\end{equation}
where $\widetilde{\mm}^1$ is defined in \eqref{equ:m-tilde}.
If $k=1$ and $\iota_1=1$, we have
\begin{equation}
U_{1}^{\mm^1,(i_1)}=U_1^{(i_1)} \text{\quad and\quad} f_{\nu}(U_{1}^{\mm^1,(i_1)})=U_0^{(i_1-1)}.
\end{equation}
Moreover, if $\ell_s=0$ for all $1\leqslant s\leqslant k$, then $\mm^s=\emptyset$ for all $1\leqslant s\leqslant k$ and we denote
\begin{equation}
y_{\iota_1 \iota_2 \cdots \iota_k}:=y_{\iota_1 \iota_2 \cdots \iota_k}^{\mm^1,\mm^2,\cdots,\mm^k} \text{\quad and\quad}
U_{\iota_1 \iota_2 \cdots \iota_k}^{(i_k)}:=U_{\iota_1 \iota_2 \cdots \iota_k}^{\mm^1,\mm^2,\cdots,\mm^k,(i_k)}
\end{equation}
for simplicity.
See Figure \ref{Fig_parabolic}.

\begin{figure}[!htpb]
  \setlength{\unitlength}{1mm}
  \centering
  \includegraphics[width=0.85\textwidth]{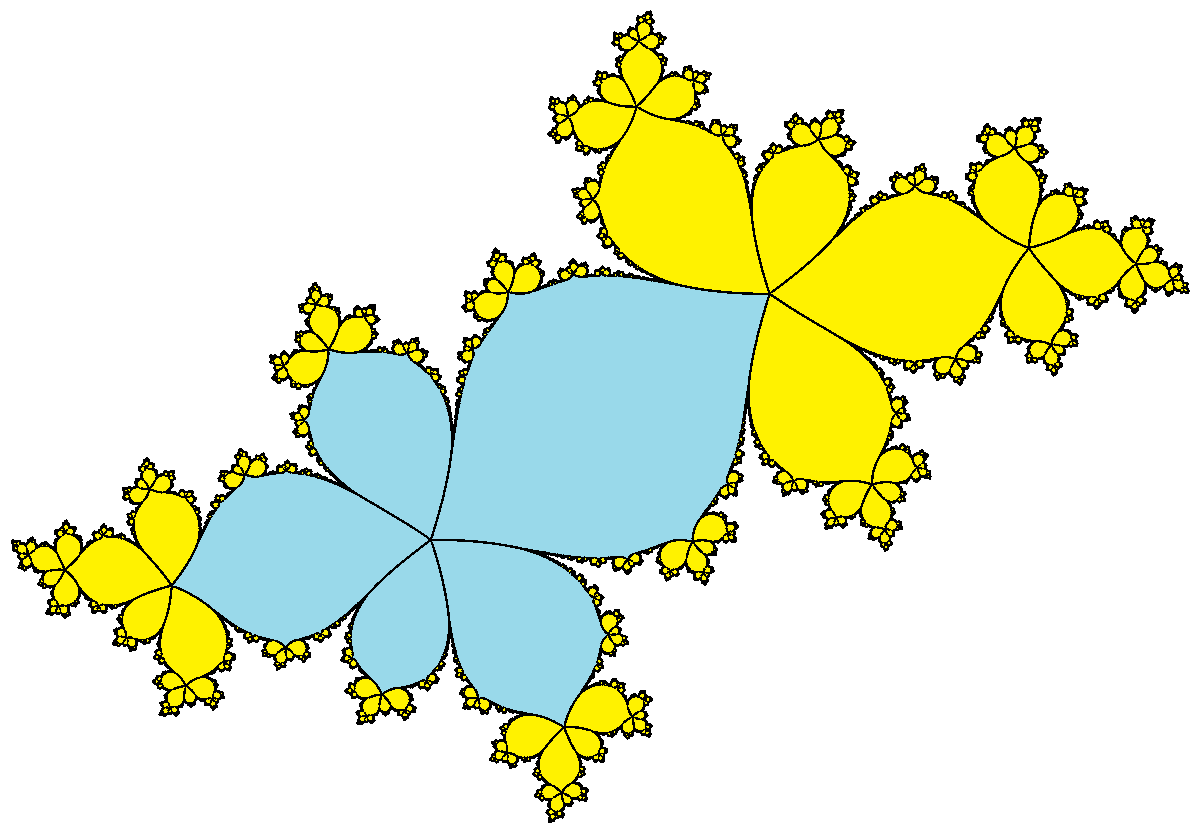}
  \put(-68,40){$U_0^{(0)}=U_0^{(p)}$}
  \put(-83.2,20){$U_0^{(4)}$}
  \put(-85,39){$U_0^{(3)}$}
  \put(-70,19){$U_0^{(2)}$}
  \put(-94,26.5){$U_0^{(1)}$}
  \put(-43.5,60){$U_1^{(4)}$}
  \put(-41,42){$U_1^{(3)}$}
  \put(-55.5,62){$U_1^{(2)}$}
  \put(-33,54.5){$U_1^{(1)}$}
  \put(-112,25){\footnotesize{$U_2^{(1)}$}}
  \put(-104.2,18){\footnotesize{$U_2^{(2)}$}}
  \put(-106.7,29){\footnotesize{$U_2^{(3)}$}}
  \put(-108.3,21){\tiny{$U_2^{(4)}$}}
  \put(-69.9,8.8){\tiny{$U_3^{(3)}$}}
  \put(-67.3,6){\tiny{$U_3^{(1)}$}}
  \put(-59.5,6.5){\tiny{$U_3^{(4)}$}}
  \put(-63,11.3){\tiny{$U_3^{(2)}$}}
  \put(-14.5,56.3){\footnotesize{$U_{11}^{(1)}$}}
  \put(-22.5,63){\footnotesize{$U_{11}^{(2)}$}}
  \put(-19.9,52){\footnotesize{$U_{11}^{(3)}$}}
  \put(-17.7,59.2){\rotatebox{55}{\tiny{$U_{11}^{(4)}$}}}
  \put(-56,73){\tiny{$U_{12}^{(3)}$}}
  \put(-59.3,76){\tiny{$U_{12}^{(1)}$}}
  \put(-65.5,77.5){\tiny{$U_{12}^{(4)}$}}
  \put(-63,70.5){\tiny{$U_{12}^{(2)}$}}
  \put(-57.5,22.5){\tiny{$U_{p+1}^{0,(3)}$}}
  \put(-50,22){\tiny{$U_{p+1}^{0,(1)}$}}
  \put(-47,26.5){\tiny{$U_{p+1}^{0,(4)}$}}
  \put(-46.5,31){\tiny{$U_{p+1}^{0,(2)}$}}
  \put(-70,60){\tiny{$U_{p+1}^{1,(3)}$}}
  \put(-77.8,60.5){\tiny{$U_{p+1}^{1,(1)}$}}
  \put(-80,56.2){\tiny{$U_{p+1}^{1,(4)}$}}
  \put(-82,51.5){\tiny{$U_{p+1}^{1,(2)}$}}
  \put(-75.5,31){$y_0$}
  \put(-48.5,51){$y_1$}
  \put(-101,25){$y_2$}
  \put(-67,13.5){$y_3$}
  \put(-87.5,45.5){$y_4$}
  \put(-83,16.5){$y_p$}
  \put(-24.1,57.5){$y_{11}$}
  \put(-57,69){$y_{12}$}
  \put(-38.5,37){$y_{13}$}
  \put(-42.2,66){$y_{14}$}
  \put(-58.5,31.5){$y_{p+1}^0$}
  \put(-70,50){$y_{p+1}^1$}
  \put(-33,50){\tiny{$y_{1(p+1)}^0$}}
  \put(-32,61.2){\tiny{$y_{1(p+1)}^1$}}
 \caption{A global view of the dynamical plane of $f_\nu$, where $\nu=3/5$ (i.e., $q=3$ and $p=5$). The $0$-drops, $1$-drops and some of their preimages are marked.}
  \label{Fig_parabolic}
\end{figure}
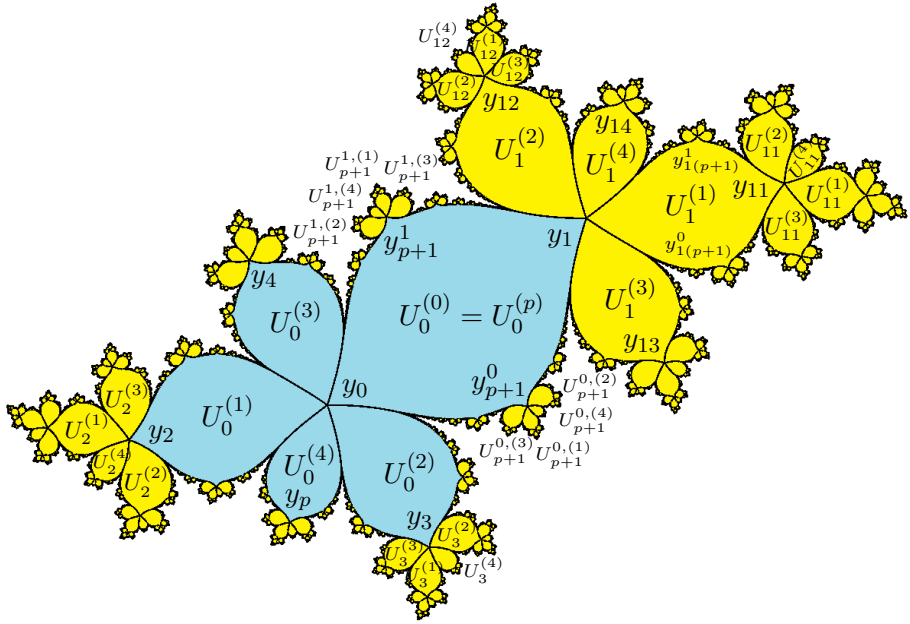

\medskip
In the following we give some definitions similar to those in \S\ref{f_theta}.

\begin{defi}[Limbs]
For any pair of admissible multi-indices $\iota= \iota_1 \iota_2 \cdots \iota_k$ and $\mm=(\mm^1,\mm^2,\cdots,\mm^k)$ with $k\geqslant 1$ in $\Lambda_k$ and any $1\leqslant i_k\leqslant p-1$, define the \emph{limb} $L_{\iota_1 \iota_2 \cdots \iota_k}^{\mm^1,\mm^2,\cdots,\mm^k,(i_k)}$ as the closure of the union of the drop $U_{\iota_1 \iota_2 \cdots \iota_k}^{\mm^1,\mm^2,\cdots,\mm^k,(i_k)}$ and all its descendants (i.e., children and grandchildren, etc.).
Note that $U_{\iota_1 \iota_2 \cdots \iota_{k+1}}^{\mm^1,\mm^2,\cdots,\mm^{k+1},(i_{k+1})}$ is a child of $U_{\iota_1 \iota_2 \cdots \iota_k}^{\mm^1,\mm^2,\cdots,\mm^k,(i_k)}$ if and only if $y_{\iota_1 \iota_2 \cdots \iota_{k+1}}^{\mm^1,\mm^2,\cdots,\mm^{k+1}}\in \partial U_{\iota_1 \iota_2 \cdots \iota_k}^{\mm^1,\mm^2,\cdots,\mm^k,(i_k)}$, which is equivalent to $i_k \equiv \iota_{k+1}$ $\Mod p$. Moreover, the limb $L_0^{(i_0)}$ with $0\leqslant i_0\leqslant p-1$ can be defined similarly.
We call $y_{\iota_1 \iota_2 \cdots \iota_k}^{\mm^1,\mm^2,\cdots,\mm^k}$ the root of $L_{\iota_1 \iota_2 \cdots \iota_k}^{\mm^1,\mm^2,\cdots,\mm^k,(i_k)}$.
\end{defi}

Since the parabolic fixed point $y_0$ of $f_\nu$ is the landing point of exactly $p$ external rays in the basin of the infinity, as the depth of a limb of $f_\nu$ goes to infinity, the corresponding Euclidean diameter goes to $0$. This implies that the Julia set of $f_\nu$ is locally connected and the filled Julia set $K(f_\nu)$ can be written as
\begin{equation}
K(f_\nu)\,=\,\bigcup_{i=0}^{p-1}\overline{L_0^{(i)}}\,=\,\bigcup_{i=0}^{p-1}\overline{U_0^{(i)}}~\cup \bigcup_{(n,\mm)\in\Lambda_1}\bigcup_{i=1}^{p-1} L_{n}^{\mm,(i)}.
\end{equation}

Consider a sequence of drops $\big\{U_0^{(i_0)}, U_{\iota_1}^{\mm^1,(i_1)}$, $U_{\iota_1 \iota_2}^{\mm^1,\mm^2,(i_2)}$, $\cdots\big\}$, where each drop $U_{\iota_1 \iota_2 \cdots \iota_k}^{\mm^1,\mm^2,\cdots,\mm^k,(i_k)}$  is the parent of $U_{\iota_1 \iota_2 \cdots \iota_{k+1}}^{\mm^1,\mm^2,\cdots,\mm^{k+1},(i_{k+1})}$. The closure of the union of this sequence
\begin{equation}\label{equ:drop-chain}
\MC:=\overline{\bigcup\limits_k U_{\iota_1 \iota_2 \cdots \iota_k}^{\mm^1,\mm^2,\cdots,\mm^k,(i_k)}}
\end{equation}
is a \emph{drop-chain}.
The intersection of the corresponding limbs $L_0^{(i_0)}\supset U_{\iota_1}^{\mm^1,(i_1)} \supset U_{\iota_1 \iota_2}^{\mm^1,\mm^2,(i_2)} \supset \cdots$ must be a singleton which we denote by
\begin{equation}
\xi(\MC):=\bigcap\limits_k L_{\iota_1 \iota_2 \cdots \iota_k}^{\mm^1,\mm^2,\cdots,\mm^k,(i_k)}.
\end{equation}

\begin{defi}[Drop-rays]
For two different points on the boundary of a drop $U$, we denote $\llbracket z_1,z_2 \rrbracket$ as \eqref{equ:ray-2-pt}, i.e., the union of two rays which connect $z_1$, $z_2$ with the center $z(U)$.
For any drop-chain as \eqref{equ:drop-chain}, the following path $R=R(\MC)$ is the ``most efficient" path in $\MC$ which connects $y_0=0$ to $\xi(\MC)$:
\begin{equation}
R(\MC):=\llbracket 0, y_{\iota_1}^{\mm^1} \rrbracket\cup
\bigcup\limits_{k\geqslant 1}\llbracket y_{\iota_1\cdots\iota_k}^{\mm^1,\cdots,\mm^k}, y_{\iota_1\cdots\iota_k\iota_{k+1}}^{\mm^1,\cdots,\mm^k,\mm^{k+1}} \rrbracket\cup \{\xi(\MC)\}.
\end{equation}
We call $R(\MC)$ the \emph{drop-ray} associated with $\MC$, and say that $R(\MC)$ and $\MC$ \emph{land} at $\xi(\MC)$.
\end{defi}

According to the above definitions,
every point in $K(f_\nu)$ either belongs to the closure of a drop or is the landing point of a unique drop-chain, and the assignment $\MC\mapsto \xi(\MC)$ is one-to-one, i.e., different drop-rays land at distinct points.

\begin{defi}[Spines]
Consider the two drop-chains
\begin{equation}
\MC_\nu:=\overline{U_0^{(0)}\cup U_1^{(1)}\cup U_{11}^{(1)}\cup \cdots}\,, \quad \MC'_\nu:=\overline{U_0^{(1)}\cup U_2^{(1)}\cup U_{21}^{(1)}\cup \cdots}
\end{equation}
with $f_\nu(\MC_\nu')=\MC_\nu$. By definition, we see that $\MC_\nu$ and $\MC_\nu'$  land at the repelling fixed point $\xi(\MC_\nu)= \beta_\nu$ and $\xi(\MC_\nu')=\beta_\nu'$ respectively, where $f_\nu(\beta_\nu')=\beta_\nu$. The \emph{spine} of $f_\nu$ is defined as the union of the drop-rays
\begin{equation}\label{equ:spine-nu}
S_\nu:=R(\MC_\nu)\cup R(\MC_\nu').
\end{equation}
Notice that every point on $S_\nu$ which is not in the interior of $K(f_\nu)$ is either one of the ends points $\beta_\nu$, $\beta_\nu'$, or a preimage of the fixed point $y_0=0$.
\end{defi}

Since $J(f_\nu)$ is locally connected, we have the Carath\'{e}odory loop $\eta_\nu:\T\to J(f_\nu)$ as \eqref{equ:eta-theta} which conjugates the angle-doubling map to $f_\nu$. Similar to \eqref{equ:J-0-1}, we also define two connected subsets $J^0_\nu$ and $J^1_\nu$ of the Julia set $J(f_\nu)$.

\begin{defi}[Itinerary]
For any $z\in J(f_{\nu})$, we define the \emph{itinerary} of $z$ with respect to the spine $S_{\nu}$ as: $\delta=(\delta_0,\delta_1, \delta_2,\cdots)$, where $\delta_i\in\{0,1\}$ is determined by
\begin{equation}
z_i\in J^{\delta_i}_{\nu}, \text{ where } z_0=z \text{ and } z_{i+1}=f_{\nu}(z_i), i\geqslant 0.
\end{equation}
If $z$ is a preimage of the repelling fixed point $\beta_\nu$, we assign two different itineraries to $z$. If $z$ is $p$-accessible, we assign $p$ different itineraries to $z$.
\end{defi}

There are following three cases:

\medskip
\textbf{Case 1}. The orbit of $z$ never hits the spine $S_\nu$. Then there exists a unique angle $t$ with $z=\eta_\nu(t)$ and $t$ has the binary expansion $0.\delta_0\delta_1\delta_2\cdots$.

\medskip
\textbf{Case 2}. The orbit of $z$ eventually hits the fixed point $\beta_\nu$. The two itineraries of $z$ are given as in \S\ref{f_theta}.
In particular, there exists a unique angle $t$ with $z=\eta_\nu(t)$ and $t$ has two binary expansions $0.\delta_0\delta_1\cdots\delta_{n-2}1000\cdots$ and $0.\delta_0\delta_1\cdots\delta_{n-2} 0111\cdots$ for the $z$ satisfying $z_n=\beta_\nu$ (where $n\geqslant 2$ is the smallest integer satisfying this property).

\medskip
\textbf{Case 3}. The orbit of $z$ eventually hits the fixed point $y_0=0$.
For such points, there are exactly $p$ angles $0<s_1<s_2<\cdots<s_p<1$ with $\eta_{\nu}(s_i)=z$ for all $1\leqslant i\leqslant p$.
In particular, there are exactly $p$ angles $0<t_1<t_2<\cdots<t_p<1$ with $\eta_{\nu}(t_i)=y_0$ for all $1\leqslant i\leqslant p$. They are uniquely determined by the rotation number $\nu=q/p$:
\begin{equation}
2t_i\equiv t_{i+q~(\Mod p)}~(\Mod 1), \text{\quad where } 1\leqslant i\leqslant p.
\end{equation}
Hence they have the form $t_i=a_i/(2^p-1)$ for integers $1\leqslant a_1<a_2<\cdots<a_p<2^p-1$. Note that each $t_i$ is $p$-periodic under the doubling map. Thus the binary expansion of $t_1$ has the form
\begin{equation}
0.\overline{\sigma_1\sigma_2\cdots\sigma_p}=0.\sigma_1\sigma_2\cdots\sigma_p\sigma_1\sigma_2\cdots\sigma_p\cdots,
\end{equation}
where $\sigma_i\in\{0,1\}$ for $1\leqslant i\leqslant p$. Then under the binary expansion, we have\footnote{For example, if $\nu=3/5$ (i.e., $q=3$ and $p=5$), then $t_1=11/31=0.\overline{01011}$, $t_4=22/31=0.\overline{10110}$, $t_2=13/31=0.\overline{01101}$, $t_5=26/31=0.\overline{11010}$ and $t_3=21/31=0.\overline{10101}$.}
\begin{equation}
t_{1+(i-1)q~(\Mod p)}=0.\overline{\sigma_i\sigma_{i+1}\cdots\sigma_p\sigma_1\cdots\sigma_{i-1}}, \text{\quad where } 1\leqslant i\leqslant p.
\end{equation}
Denote $\tau(i):=1+(i-1)q~(\Mod p)$, where $1\leqslant i\leqslant p$.
The $p$ itineraries of $y_0=0$ are
\begin{equation}
\delta^{(\tau(i))}(y_0)=\big(\overline{\sigma_i,\sigma_{i+1},\cdots,\sigma_p,\sigma_1,\cdots,\sigma_{i-1}}\big), \text{\quad where } 1\leqslant i\leqslant p.
\end{equation}
We write $\{\delta^{(\tau(i))}(y_0): 1\leqslant i\leqslant p\}$ in another form:
\begin{equation}\label{equ:delta-i-y0}
\delta^{(i)}(y_0)=\big(\overline{\sigma_1^i,\sigma_2^i,\cdots,\sigma_p^i}\big), \text{\quad where } 1\leqslant i\leqslant p.
\end{equation}

If $z=y_{11\cdots1}\in f_{\nu}^{-(k-1)}(y_1)$ with $k\geqslant 1$, then the $p$ itineraries of $z$ are given as $\{\delta^{(i)}:1\leqslant i\leqslant p\}$, where
\begin{equation}\label{equ:delta-i}
\delta^{(i)}=
\left\{
\begin{array}{ll}
\Big(\underbrace{0,0,\cdots,0}_{k \ \textrm{times}}, \overline{\sigma_1^i,\sigma_2^i,\cdots,\sigma_p^i}\Big)  & \text{if } 1\leqslant i\leqslant q, \\
\Big(\underbrace{1,1,\cdots,1}_{k \ \textrm{times}}, \overline{\sigma_1^i,\sigma_2^i,\cdots,\sigma_p^i}\Big)  & \text{if } q+1\leqslant i\leqslant p.
\end{array}
\right.
\end{equation}
The binary expansions of the $p$ angles $\{s_i:1\leqslant i\leqslant p\}$ landing at $z$ are
\begin{equation}\label{equ:s-i-1}
s_i=
\left\{
\begin{array}{ll}
0.\underbrace{00\cdots 0}_{k \ \textrm{times}}\overline{\sigma_1^i\sigma_2^i\cdots\sigma_p^i}  & \text{if } 1\leqslant i\leqslant q, \\
0.\underbrace{11\cdots 1}_{k \ \textrm{times}}\overline{\sigma_1^i\sigma_2^i\cdots\sigma_p^i}  & \text{if } q+1\leqslant i\leqslant p.
\end{array}
\right.
\end{equation}
If $z=y_{21\cdots1}\in f^{-k}_{\nu}(y_1)$ with $k\geqslant 1$, then the $p$ itineraries of $z$ are given as $\{\delta^{(i)}:1\leqslant i\leqslant p\}$, where
\begin{equation}\label{equ:delta-i-2}
\delta^{(i)}=
\left\{
\begin{array}{ll}
\Big(0,\underbrace{1,\cdots,1}_{k \ \textrm{times}}, \overline{\sigma_1^i,\sigma_2^i,\cdots,\sigma_p^i}\Big)  & \text{if } 1\leqslant i\leqslant p-q, \\
\Big(1,\underbrace{0,\cdots,0}_{k \ \textrm{times}}, \overline{\sigma_1^i,\sigma_2^i,\cdots,\sigma_p^i}\Big)  & \text{if } p-q+1\leqslant i\leqslant p.
\end{array}
\right.
\end{equation}
Based on \eqref{equ:delta-i-2}, the binary expansions of the $p$ angles $\{s_i:1\leqslant i\leqslant p\}$ landing at $z$ are obtained similarly as \eqref{equ:s-i-1}.

Moreover, if $z$ is off the spine whose orbit eventually hits the  $y_0$, the itineraries $\{\delta^{(i)}:1\leqslant i\leqslant p\}$ can be defined similarly based on those of $y_{11\cdots1}$ and $y_{21\cdots1}$, and these itineraries give the binary expansions of the $p$ angles $\{s_i:1\leqslant i\leqslant p\}$ respectively.

\medskip
By the above definitions, Proposition \ref{prop:YZ-Siegel} also holds for the parabolic polynomial $f_\nu$, i.e., the following characterizations of the points on the Julia set $J(f_\nu)$ by  itineraries and angles of external rays are true:

\begin{prop}\label{prop:parabolic}
The angle(s) of the external ray(s) landing at any given $z\in J(f_\nu)$ is (are) determined by the itinerary(ies) of $z$, that is, two points in $J(f_\nu)$ having the same itinerary must coincide. Moreover, every infinite sequence of $0$'s and $1$'s can be realized as the itinerary of a unique point in $J(f_{\nu})$.
\end{prop}

\section{Proof of Theorem \ref{thm:mateable}}\label{proof}

In this section, we first study the dynamical properties of a quadratic rational map $F_{\theta,\nu}$ which serves as the model map of the mating  of $f_\theta$ and $f_\nu$.
Then we prove that $f_\theta$ and $f_\nu$ are conformally mateable based on the contraction results in \S\ref{sec:lc} and the combinatorial information in \S\ref{sec:defi}.

\subsection{Dynamics of the model map}

In the rest of this paper, we fix a bounded type irrational number
\begin{equation}
\theta\in(0,1)\setminus\Q \text{\quad and\quad} \nu=q/p\in[0,1)\cap \Q,
\end{equation}
where $q=0$, $p=1$ if $\nu=0$ and $q\geqslant 1$, $p\geqslant 2$ are coprime integers if $\nu\neq 0$. Then
\begin{equation}\label{equ:F-theta-nu}
F_{\theta,\nu}(z)=\frac{e^{2\pi\ii \nu}z+z^2}{1+e^{2\pi\ii \theta}z}
\end{equation}
has a Siegel disk $\Delta^\infty$ centered at infinity with rotation number $\theta$ and a fixed parabolic point at zero with combinatorial rotation number $\nu$.
According to \cite{GS03} or \cite{Zha11}, $\partial\Delta^\infty$ contains a critical point $c^\infty$ of $F_{\theta,\nu}$. The other critical point $c^0$ is contained in an immediate parabolic basin $\widetilde{U}_0^{(0)}$ of $F_{\theta,\nu}$. See Figures \ref{Fig_Mating-plane} and \ref{Fig_mating-cauliflower}.

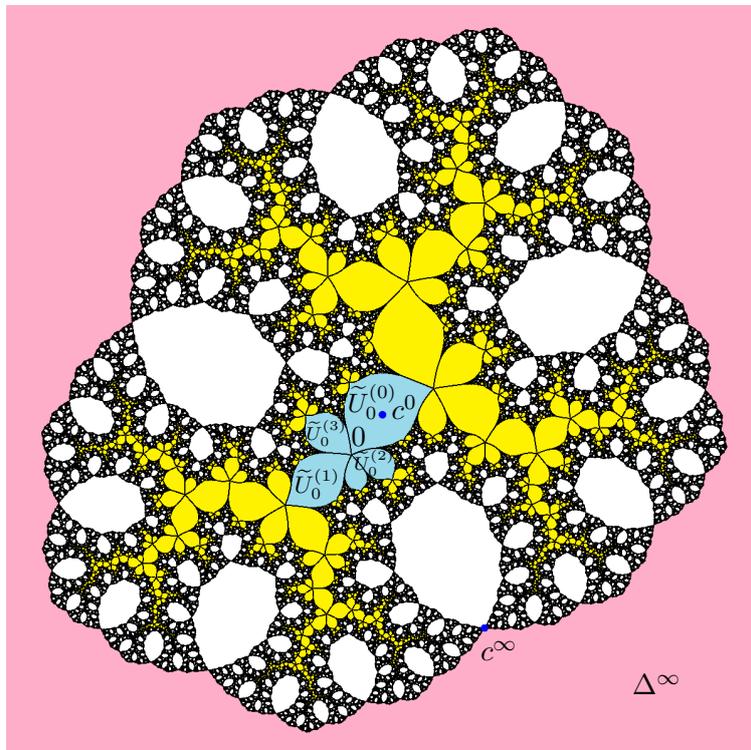
\begin{figure}[!htpb]
  \setlength{\unitlength}{1mm}
  \centering
  \includegraphics[width=0.7\textwidth]{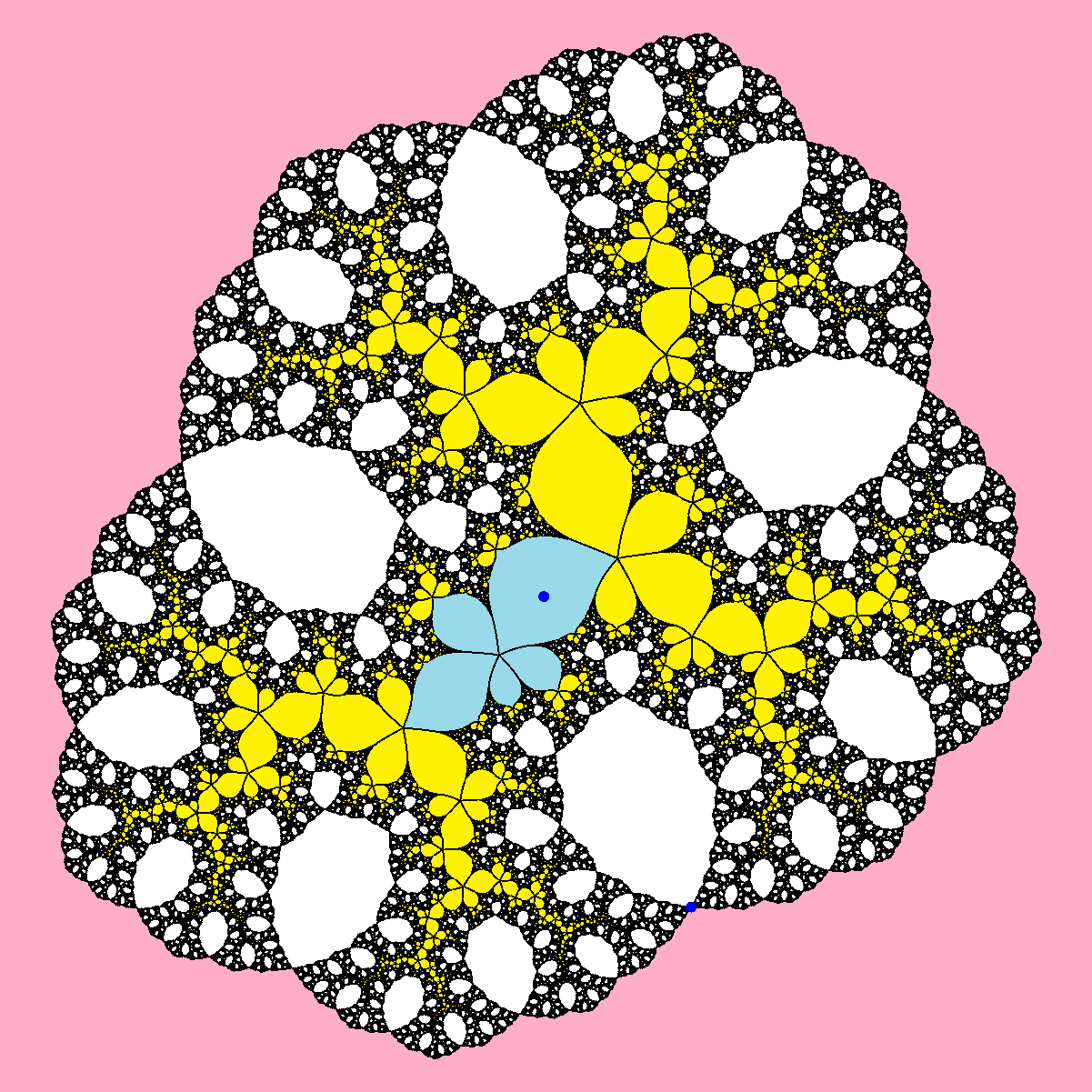}
  \put(-37,12.5){$c^\infty$}
  \put(-48.7,44.5){$c^0$}
  \put(-54,41){$0$}
  \put(-54.5,45.5){\small{$\widetilde{U}_0^{(0)}$}}
  \put(-61.5,34.8){\footnotesize{$\widetilde{U}_0^{(1)}$}}
  \put(-53.7,37.8){\tiny{$\widetilde{U}_0^{(2)}$}}
  \put(-60,41.8){\tiny{$\widetilde{U}_0^{(3)}$}}
  \put(-17,8){$\Delta^\infty$}
  \caption{The Julia set of $F_{\theta,\nu}$, where $\theta=(\sqrt{5}-1)/2$ and $\nu=3/5$. The Siegel disk centered at $\infty$ and the immediate parabolic basins attached at $0$ are colored pink and cyan respectively. The two critical points $c^\infty$ and $c^0$ are marked.}
  \label{Fig_Mating-plane}
\end{figure}

\begin{lem}\label{lem:lc-F}
The Julia set of $F_{\theta,\nu}$ is connected and locally connected. Moreover, if $\nu\neq 0$ (i.e., $p\geqslant 2$), then each Fatou component of $F_{\theta,\nu}$ is a Jordan domain.
\end{lem}

\begin{proof}
Note that $F_{\theta,\nu}$ is a quadratic rational map. According to \cite{Yin92} or \cite{Mil93}, $J(F_{\theta,\nu})$ is connected or totally disconnected (i.e., a Cantor set). However, the latter case cannot happen since $F_{\theta,\nu}$ has a Siegel disk. Hence $J(F_{\theta,\nu})$ is connected. By Theorem \ref{thm:local-connectivity}, $J(F_{\theta,\nu})$ is locally connected.

According to \cite{Zha11}, the Siegel disk $\Delta^\infty$ centered at $\infty$ is a quasi-disk whose boundary passing through the critical point $c^\infty$. Since the other critical point $c^0$ lies in the immediate parabolic basin $\widetilde{U}_0^{(0)}$, any Fatou component which is eventually iterated to $\Delta^\infty$ is a Jordan domain (actually a quasi-disk).

Since $F_{\theta,\nu}$ has exactly one critical point in $\EC\setminus\overline{\Delta^\infty}$, this implies that $F_{\theta,\nu}$ has exactly $p$ immediate parabolic basins attaching at $0$ and they form exactly one cycle of parabolic periodic Fatou components $\big\{\widetilde{U}_0^{(0)}, \widetilde{U}_0^{(1)}, \cdots, \widetilde{U}_0^{(p-1)}\big\}$. We name these Fatou components such that
\begin{equation}\label{equ:U-i-tilde}
F_{\theta,\nu}(\widetilde{U}_0^{(i)})=\widetilde{U}_0^{(i-1)}, \text{\quad where } 1\leqslant i\leqslant p \text{ and } \widetilde{U}_0^{(p)}:=\widetilde{U}_0^{(0)}.
\end{equation}
Thus it suffices to prove that $\widetilde{U}_0^{(0)}$ is a Jordan domain when $p\geqslant 2$.

\medskip
Besides $\Delta^\infty$ itself, $\widetilde{U}_0:=\Delta^\infty$ has another preimage $\widetilde{U}_1$ under $F_{\theta,\nu}$ which attaches at the critical point $\widetilde{x}_1:=c^\infty$. Similar to the quadratic polynomial $f_\theta$, $F_{\theta,\nu}^{-1}(\widetilde{U}_1)$ consists of two Fatou components $\widetilde{U}_{1,1}$ and $\widetilde{U}_2$, where $\widetilde{U}_{1,1}$ attaches at $\widetilde{x}_{1,1}\in \partial \widetilde{U}_1$, $\widetilde{U}_2$ attaches at $\widetilde{x}_2\in \partial \widetilde{U}_0$ and $F_{\theta,\nu}^{-1}(\widetilde{x}_1)=\{\widetilde{x}_{1,1},\widetilde{x}_2\}$. Inductively, one can first define $\widetilde{x}_n\in\partial \widetilde{U}_0$ and the \textit{drop} $\widetilde{U}_n$ attaching at $\widetilde{x}_n$ (where $\widetilde{x}_n$ is called the \textit{root} of $\widetilde{U}_n$) with $n\geqslant 2$ such that
\begin{equation}
F_{\theta,\nu}^{\circ (n-1)}(\widetilde{x}_n)=\widetilde{x}_1 \text{\quad and \quad} F_{\theta,\nu}^{\circ (n-1)}(\widetilde{U}_n)=\widetilde{U}_1.
\end{equation}
Then for any multi-index $\iota= \iota_1 \iota_2 \cdots \iota_k$ of length $k\geqslant 2$, where each $\iota_j$ is a positive integer, the $(\iota_1 +\iota_2 +\cdots +\iota_k)$-drop $\widetilde{U}_{\iota_1 \iota_2 \cdots \iota_k}$ of generation $k$ with root
\begin{equation}\label{equ:drop-tilde}
\widetilde{x}(\widetilde{U}_{\iota_1 \iota_2 \cdots \iota_k})=\widetilde{x}_{\iota_1 \iota_2 \cdots \iota_k}
\end{equation}
can be defined similarly to \eqref{roots}.
Note that the intersection of the closure of two different drops is either empty or a singleton. Hence the parent, children can be defined completely similarly to $f_\theta$.
By the above definition, for any multi-index $\iota= \iota_1 \iota_2 \cdots \iota_k$ of length $k\geqslant 1$, the union of the closure of the sequence
\begin{equation}\label{equ:sequence-1}
\widetilde{U}_0=\Delta^\infty, \widetilde{U}_{\iota_1}, \widetilde{U}_{\iota_1\iota_2}, \cdots, \widetilde{U}_{\iota_1 \iota_2 \cdots \iota_k}
\end{equation}
is connected, where $\widetilde{U}_{\iota_1 \iota_2 \cdots \iota_j}$ (resp. $\widetilde{U}_{\iota_1}$) is a child of $\widetilde{U}_{\iota_1 \iota_2 \cdots \iota_{j-1}}$ (resp. $\widetilde{U}_0$) and
\begin{equation}\label{equ:root-tilde}
\partial\widetilde{U}_{\iota_1 \iota_2 \cdots \iota_{j-1}}\cap\partial\widetilde{U}_{\iota_1 \iota_2 \cdots \iota_j}=\{\widetilde{x}_{\iota_1 \iota_2 \cdots \iota_j}\}
\end{equation}
is a preimage of the critical point $\widetilde{x}_1$.

\medskip
We have proved that the boundary of the immediate parabolic basin $\widetilde{U}_0^{(0)}$ is locally connected. Assume that it is not a Jordan curve. Then there exists a Jordan curve $\gamma$ containing a point $z_0$ in the closure of $\widetilde{U}_0^{(0)}$ such that
\begin{equation}\label{equ:cond-1}
z_0\in \gamma\cap \partial \widetilde{U}_0^{(0)}, \quad \gamma\setminus\{z_0\}\subset \widetilde{U}_0^{(0)} \text{\quad and\quad}
J(F_{\theta,\nu})\cap\gamma^{\Int}\neq\emptyset,
\end{equation}
where $\gamma^{\Int}$ is the connected component of $\EC\setminus\gamma$ which does not contain $\infty$. We claim that there exists a Fatou component (i.e., a drop defined above) $\widetilde{U}_{\iota_1 \iota_2 \cdots \iota_k}$, which is an iterated preimage of $\Delta^\infty$ such that
\begin{equation}\label{equ:claim-1}
\widetilde{U}_{\iota_1 \iota_2 \cdots \iota_k}\subset\gamma^{\Int}.
\end{equation}
Indeed, otherwise $\{F_{\theta,\nu}^{\circ n}\}_{n\geqslant 0}$ does not take the values in $\Delta^\infty$ and hence is a normal family in $\gamma^{\Int}$, which contradicts the assumption that $J(F_{\theta,\nu})\cap\gamma^{\Int}\neq\emptyset$.

Note that $\widetilde{U}_{\iota_1 \iota_2 \cdots \iota_k}$ can be connected to $\Delta^\infty$ by the finite sequence \eqref{equ:sequence-1} of closed Jordan disks.
By \eqref{equ:root-tilde}, \eqref{equ:cond-1} and \eqref{equ:claim-1}, we conclude that $z_0=\widetilde{x}_{\iota_1 \iota_2 \cdots \iota_j}$ is a preimage of the critical point $\widetilde{x}_1$.
Since $\{F_{\theta,\nu}^{\circ (np+j)}(\widetilde{x}_1):n\geqslant 0\}$ is dense in $\partial\Delta^\infty$ for any $0\leqslant j\leqslant p-1$, it follows that $\partial\Delta^\infty\subset \partial \widetilde{U}_0^{(0)}$. Hence for any two different points $z_1,z_2\in\partial\Delta^\infty$, there exists a simple arc $\gamma^{(0)}$ in the closure of $\widetilde{U}_0^{(0)}$ connecting $z_1$ with $z_2$ such that $\gamma^{(0)}\setminus\{z_1,z_2\}\subset \widetilde{U}_0^{(0)}$.

If $p\geqslant 2$, by a similar argument as above, we have $\partial\Delta^\infty\subset \partial \widetilde{U}_0^{(1)}$, where $\widetilde{U}_0^{(1)}\neq \widetilde{U}_0^{(0)}$. Moreover, for any two different points $z_3,z_4\in\partial\Delta^\infty\setminus\{z_1,z_2\}$ satisfying $z_1,z_3,z_2,z_4$ lie counterclockwise on $\partial\Delta^\infty$, there exists a simple arc $\gamma^{(1)}$ in the closure of $\widetilde{U}_0^{(1)}$ connecting $z_3$ with $z_4$ such that $\gamma^{(1)}\setminus\{z_3,z_4\}\subset \widetilde{U}_0^{(1)}$. Hence $\gamma^{(0)}\cap\gamma^{(1)}\neq\emptyset$. This is impossible since $\widetilde{U}_0^{(0)}\cap \widetilde{U}_0^{(1)}=\emptyset$.
Thus $\widetilde{U}_0^{(0)}$ is a Jordan domain when $p\geqslant 2$ and the proof is complete.
\end{proof}

\begin{rmk}
If $\nu=0$ (i.e. $p=1$), then the unique parabolic Fatou component of $F_{\theta,\nu}$ is not a Jordan domain. In fact, according to Petersen and Roesch \cite{PR21}, in this case the Julia set of $F_{\theta,\nu}$ is homeomorphic to $J(f_\theta)$. See Figure \ref{Fig_mating-cauliflower}.
\end{rmk}

\begin{figure}[!htpb]
  \setlength{\unitlength}{1mm}
  \centering
  \includegraphics[width=0.9\textwidth]{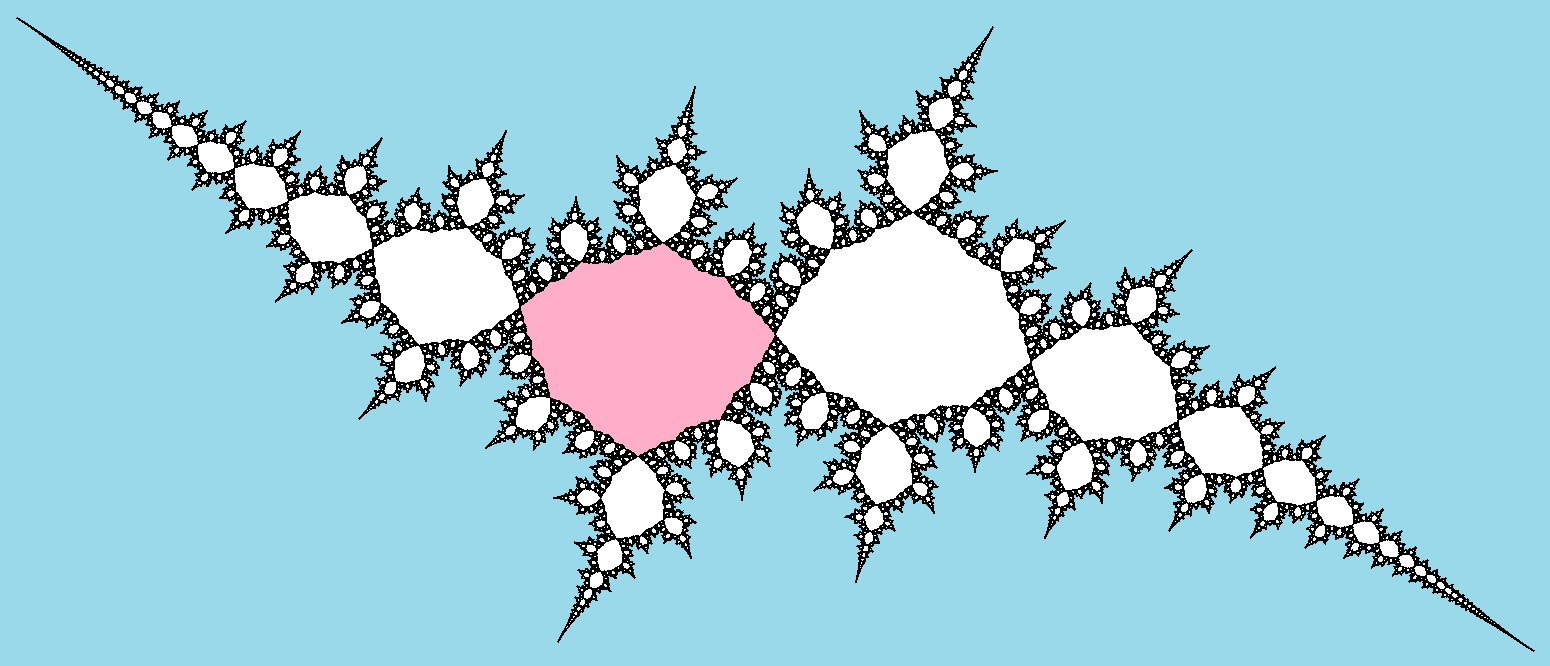}
  \put(-62.5,26.7){$c^\infty$}
  \put(-20,40){$\widetilde{U}_0^{(0)}$}
  \put(-77,25){$\Delta^\infty$}
  \caption{The Julia set of $F_{\theta,\nu}$, where $\theta=(\sqrt{5}-1)/2$ and $\nu=0$ (we put the Siegel disk in $\C$ under a fractional linear map). According to Petersen and Roesch, $J(F_{\theta,\nu})$ is homeomorphic to $J(f_\theta)$. See also Lemma \ref{semiconjugacy}(a) and Figure \ref{Fig_Julias}.}
  \label{Fig_mating-cauliflower}
\end{figure}

As an analogue to the filled Julia sets $K(f_\theta)$ and $K(f_\nu)$, we define
\begin{equation}
K^{0}(F_{\theta,\nu}):=\{z\in\C:\ \textrm{the orbit } \{F_{\theta,\nu}^{\circ n}(z)\}_{n\geqslant 0}\ \textrm{never intersects}\ \Delta^{\infty}\}
\end{equation}
and
\begin{equation}
K^{\infty}(F_{\theta,\nu}):=\overline{\EC\setminus K^{0}(F_{\theta,\nu})}.
\end{equation}
Then $J(F_{\theta,\nu})=\partial K^{0}(F_{\theta,\nu})=\partial K^{\infty}(F_{\theta,\nu})$.
Based on Lemma \ref{lem:lc-F} and \eqref{equ:U-i-tilde}, if $p\geqslant 2$, we can assign an address for each preimage of the immediate parabolic basins similar to the quadratic parabolic polynomial $f_{\nu}$ in \S\ref{f_nu}. Specifically, for any admissible multi-indices $\iota= \iota_1 \iota_2 \cdots \iota_k$ and $\mm=(\mm^1,\mm^2,\cdots,\mm^k)$, the $(\iota_1 +\iota_2 +\cdots +\iota_k)$-drops $\{\widetilde{U}_{\iota_1 \iota_2 \cdots \iota_k}^{\mm^1,\mm^2,\cdots,\mm^k,(i_k)}:1\leqslant i_k\leqslant p-1\}$ of generation $k$ with root
\begin{equation}
\widetilde{y}(\widetilde{U}_{\iota_1 \iota_2 \cdots \iota_k}^{\mm^1,\mm^2,\cdots,\mm^k,(i_k)})=\widetilde{y}_{\iota_1 \iota_2 \cdots \iota_k}^{\mm^1,\mm^2,\cdots,\mm^k}
\end{equation}
growing from the immediate parabolic basins $\big\{\widetilde{U}_0^{(0)}, \widetilde{U}_0^{(1)}, \cdots, \widetilde{U}_0^{(p-1)}\big\}$ in $K^{0}(F_{\theta,\nu})$ can be defined.
Note that the $(\iota_1 +\iota_2 +\cdots +\iota_k)$-drops $\widetilde{U}_{\iota_1 \iota_2 \cdots \iota_k}$ growing from $\Delta^\infty$ have been defined in the proof of Lemma \ref{lem:lc-F}. Hence the limbs $\widetilde{L}_{\iota_1 \iota_2 \cdots \iota_k}$ in $K^{\infty}(F_{\theta,\nu})$, limbs $\widetilde{L}_{\iota_1 \iota_2 \cdots \iota_k}^{\mm^1,\mm^2,\cdots,\mm^k,(i_k)}$ in $K^{0}(F_{\theta,\nu})$ and the drop-chains etc. can be defined completely similarly, i.e., each limb is the closure of the union of the drop with the same address and all its descendants (see \S\ref{sec:defi}).

\begin{lem}\label{limb}
As the depth $\iota_1+\iota_2 +\cdots+\iota_k$ of a limb $\widetilde{L}_{\iota_1 \iota_2 \cdots \iota_k}$ \textup{(}resp. $\widetilde{L}_{\iota_1 \iota_2 \cdots \iota_k}^{\mm^1,\mm^2,\cdots,\mm^k,(i_k)}$ when $p\geqslant 2$\textup{)} goes to infinity, the Euclidean diameter $\diam (\widetilde{L}_{\iota_1 \iota_2 \cdots \iota_k})$
\textup{(}resp. $\diam$ $(\widetilde{L}_{\iota_1 \iota_2 \cdots \iota_k}^{\mm^1,\mm^2,\cdots,\mm^k,(i_k)})$\textup{)} goes to zero.
\end{lem}

\begin{proof}
We only consider $p\geqslant 2$ since the case for $p=1$ is similar and simpler. Then $\partial \Delta^\infty\cap\partial \widetilde{U}_0^{(i)}=\emptyset$ for each immediate parabolic basin $\widetilde{U}_0^{(i)}$, where $0\leqslant i\leqslant p-1$. Hence
\begin{equation}
\Omega:=\EC\setminus\Big(\overline{\Delta^\infty}\cup\bigcup_{i=0}^{p-1}\overline{\widetilde{U}_0^{(i)}} \Big)
\end{equation}
is a domain whose boundary is piecewise simple (see Figure \ref{Fig_Mating-plane}).
For any limb $\widetilde{L}_{\iota_1 \iota_2 \cdots \iota_k}\neq \widetilde{L}_0$ and $\widetilde{L}_{\iota_1 \iota_2 \cdots \iota_k}^{\mm^1,\mm^2,\cdots,\mm^k,(i_k)}\not\in\big\{\widetilde{L}_0^{(0)}, \widetilde{L}_0^{(1)}, \cdots, \widetilde{L}_0^{(p-1)}\big\}$ with $1\leqslant i_k\leqslant p-1$, we have
\begin{equation}\label{equ:cover-1}
\widetilde{L}_{\iota_1 \iota_2 \cdots \iota_k}\subset\overline{\Omega} \text{\quad and\quad}
\widetilde{L}_{\iota_1 \iota_2 \cdots \iota_k}^{\mm^1,\mm^2,\cdots,\mm^k,(i_k)} \subset\overline{\Omega}.
\end{equation}
Moreover,
\begin{equation}\label{equ:homeo-1}
F_{\theta,\nu}^{\circ (l-1)}: \widetilde{L}_{\iota_1 \iota_2 \cdots \iota_k}\to \widetilde{L}_1  \text{\quad and\quad}
F_{\theta,\nu}^{\circ (l-1)}: \widetilde{L}_{\iota_1 \iota_2 \cdots \iota_k}^{\mm^1,\mm^2,\cdots,\mm^k,(i_k)} \to \widetilde{L}_1^{(i_k)}
\end{equation}
are homeomorphisms, where $l=\iota_1+\iota_2 +\cdots+\iota_k$.

\medskip
By Lemma \ref{lem:contrac}, there exist $\delta_0>0$ and an integer $M>0$ such that
\begin{itemize}
\item $\Omega$ can be covered by $M$ Jordan disks $\{\Omega_j:1\leqslant j\leqslant M\}$ with $\diam(\Omega_j)<\delta_0$; and
\item For any $\varepsilon>0$, there exists $N=N(\varepsilon)\geqslant 1$, such that $\diam(\Omega_j^{(n)})<\varepsilon$ for all $n\geqslant N$, where $\Omega_j^{(n)}$ is any connected component of $F_{\theta,\nu}^{-n}(\Omega_j)$.
\end{itemize}
By \eqref{equ:cover-1} and \eqref{equ:homeo-1}, for any $\varepsilon>0$, there exists $N=N(\varepsilon/M)\geqslant 1$ such that if $l=\iota_1+\iota_2 +\cdots+\iota_k\geqslant N+1$, then
\begin{equation}
\diam(\widetilde{L}_{\iota_1 \iota_2 \cdots \iota_k})<\varepsilon \text{\quad and\quad}
\diam(\widetilde{L}_{\iota_1 \iota_2 \cdots \iota_k}^{\mm^1,\mm^2,\cdots,\mm^k,(i_k)})<\varepsilon.
\end{equation}
The proof is complete.
\end{proof}

\subsection{Mateable of Siegel and parabolic quadratics}

In this subsection, we give a dynamical decomposition of the model map $F_{\theta,\nu}$ and prove Theorem \ref{thm:mateable}.
Recall that $\eta_\theta:\T\to J(f_\theta)$ and $\eta_\nu:\T\to J(f_\nu)$ are the Carath\'{e}odory loops which conjugate the angle-doubling map to $f_\theta$ and $f_\nu$ respectively.
The proof of the following result is inspired by \cite[\S 7.2]{YZ01}.

\begin{lem}\label{semiconjugacy}
There exist two continuous maps $\phi_{\theta}: K(f_{\theta})\to \EC$ and $\phi_{\nu}:K(f_{\nu})\to \EC$ such that
\begin{equation}\label{equ:semi-conj}
\phi_{\theta}\circ f_{\theta}= F_{\theta,\nu} \circ \phi_{\theta} \text{\quad and\quad} \phi_{\nu}\circ f_{\nu}= F_{\theta,\nu}\circ \phi_{\nu}
\end{equation}
hold on $K(f_{\theta})$ and $K(f_{\nu})$ respectively, where $\phi_{\theta}$ and $\phi_{\nu}$ can be chosen such that
\begin{enumerate}
\item  They are conformal in the interiors of $K(f_{\theta})$ and $K(f_{\nu})$ respectively. In particular, if $\nu=0$, then $\phi_\theta$ can be chosen further such that its restriction on $K(f_\theta)$ is a homeomorphism and $J(F_{\theta,\nu})$ is homeomorphic to $J(f_\theta)$;

\item $\phi_\theta\big(K(f_\theta)\big)\cup \phi_\nu\big(K(f_\nu)\big)=\EC$ and $\phi_{\theta}(z)=\phi_{\nu}(w)$ if and only if there exists an angle $t\in \T$ such that $z=\eta_{\theta}(t)$ and $w=\eta_{\nu}(-t)$.
\end{enumerate}
\end{lem}

\begin{proof}
(a) We only state the construction of $\phi_\nu$ since the proof of the properties of $\phi_{\theta}$ is completely the same.
Consider the quadratic parabolic polynomial $f_{\nu}$ and the rational map $F_{\theta,\nu}$. Since both of $f^{\circ p}_{\nu}$ and $F_{\theta,\nu}^{\circ p}$ are conformally conjugate to $z\mapsto z+z^{2}$ in the immediate parabolic basins, one can define a conformal conjugacy
\begin{equation}\label{equ:conj-immed}
\phi_{\nu}: \bigcup_{i=0}^{p-1}U_0^{(i)} \to  \bigcup_{i=0}^{p-1}\widetilde{U}_0^{(i)},
\end{equation}
where $\phi_\nu(U_0^{(i)})=\widetilde{U}_0^{(i)}$ for each $i$.

Suppose $\nu\neq 0$ (i.e., $p\geqslant 2$). Then $U_0^{(i)}$ and $\widetilde{U}_0^{(i)}$ are Jordan domains by Lemma \ref{lem:lc-F}. The conformal conjugacy $\phi_\nu$ can be extended homeomorphically to the closure of $\bigcup_{i=0}^{p-1}U_0^{(i)}$ such that it is a conjugacy which we still denote by $\phi_\nu$.
Moreover, by pulling back immediate parabolic basins, $\phi_{\nu}$ can be extended homeomorphically to the union of the closures of all drops of $f_{\nu}$ with the form $U_{\iota_1 \iota_2 \cdots \iota_k}^{\mm^1,\mm^2,\cdots,\mm^k,(i_k)}$ such that the following restriction is conformal:
\begin{equation}
\phi_\nu: U_{\iota_1 \iota_2 \cdots \iota_k}^{\mm^1,\mm^2,\cdots,\mm^k,(i_k)} \to \widetilde{U}_{\iota_1 \iota_2 \cdots \iota_k}^{\mm^1,\mm^2,\cdots,\mm^k,(i_k)}.
\end{equation}

Next we extend $\phi_{\nu}$ to a continuous semiconjugacy $\phi_\nu:K(f_{\nu})\to K^0(F_{\theta,\nu})$. By definitions in \S\ref{f_nu}, every point in $K(f_{\nu})$ is either in the closure of a drop or is the landing point of a unique drop-chain. Thus it remains to define $\phi_\nu$ at the landing points of drop-chains of $f_{\nu}$.
Let $\MC_{\nu}$ be a drop-chain of $f_\nu$ which lands at a point $z$ and consider the corresponding drop-chain $\widetilde{\MC}_{\nu}$ of $F_{\theta,\nu}$ whose drops have the same addresses. By Lemma \ref{limb}, the diameters of the corresponding limbs go to zero as the depth goes to infinity. Hence $\widetilde{\MC}_{\nu}$ lands at a well-defined point $z'\in K^0(F_{\theta,\nu})$. It is naturally to define $\phi_{\nu}(z)=z'$.

Now we prove that $\phi_{\nu}$ is continuous as a map from $K(f_{\nu})$ into $\EC$. Take a point $z\in K(f_{\nu})$ and a sequence $z_n\in K(f_{\nu})$ converging to $z$. If $z$ belongs to the interior of $K(f_{\nu})$, the continuity is trivial. If $z\in \partial K(f_{\nu})=J(f_\nu)$, by Lemma \ref{limb}, as the depth of a limb $\widetilde{L}_{\iota_1 \iota_2 \cdots \iota_k}^{\mm^1,\mm^2,\cdots,\mm^k,(i_k)}$ of $F_{\theta,\nu}$ goes to infinity, $\diam(\widetilde{L}_{\iota_1 \iota_2 \cdots \iota_k}^{\mm^1,\mm^2,\cdots,\mm^k,(i_k)})$ goes to zero. It is clear that $\phi_{\nu}(z_n)\to \phi_{\nu}(z)$ as $n\to\infty$.
Hence the semiconjugacy relations in \eqref{equ:semi-conj} hold and we have $\phi_\theta(K(f_\theta))=K^\infty(F_{\theta,\nu})$ and $\phi_\nu(K(f_\nu))=K^0(F_{\theta,\nu})$.
Note that the semiconjugacy relation for $\phi_\theta$ holds not only for $p\geqslant 2$, but also for $p=1$.

\medskip
If $\nu=0$ (i.e., $p=1$), by Lemma \ref{lem:lc-F}, $F_{\theta,\nu}$ has a completely invariant parabolic Fatou component $\widetilde{U}_0^{(0)}$ whose boundary is locally connected. By \eqref{equ:conj-immed}, the semiconjugacy relation $\phi_{\nu}\circ f_{\nu}= F_{\theta,\nu}\circ \phi_{\nu}$ still holds. It remains to prove that in this case the semiconjugacy $\phi_\theta$ obtained above is actually a conjugacy.
This is equivalent to prove that any two different drop-chains of $F_{\theta,\nu}$ in $K^\infty(F_{\theta,\nu})$ land at distinct points. In fact, otherwise there exists a simple curve $\gamma$ such that the sequence $\{F_{\theta,\nu}^{\circ n}\}_{n\geqslant 0}$ is normal in the bounded domain surrounded by $\gamma$ (similar to the argument in Lemma \ref{lem:lc-F}), which is a contradiction. Therefore, if $\nu=0$, the restriction of $\phi_\theta$ on $K(f_\theta)$ is a homeomorphism and $J(F_{\theta,\nu})$ is homeomorphic to $J(f_\theta)$.

\medskip
(b) In the rest proof we assume that $\nu\neq 0$ (i.e., $p\geqslant 2$) since during the proof we will see that the case $\nu=0$ (i.e., $p=1$) is completely similar and much simpler (because in this case $F_{\theta,\nu}$ has a completely invariant parabolic Fatou component).

\medskip
\textbf{Step 1.} We denote two drop-chains of $F_{\theta,\nu}$ growing from $\widetilde{U}_0=\Delta^\infty$ by:
\begin{equation}
\widetilde{\MC}^\infty:=\overline{\widetilde{U}_0\cup\widetilde{U}_1\cup \widetilde{U}_{11}\cup \cdots } \text{\quad and\quad}
\widetilde{\MC}'^\infty:=\overline{\widetilde{U}_0\cup\widetilde{U}_2\cup \widetilde{U}_{21}\cup \cdots}\,,
\end{equation}
where $F_{\theta,\nu}(\widetilde{\MC}'^\infty)=\widetilde{\MC}^\infty$.
By Lemma \ref{limb}, the drop-chain $\widetilde{\MC}^\infty$ lands at a fixed point $\widetilde{\beta}$ of $F_{\theta,\nu}$. Obviously $\widetilde{\beta}\neq \infty$ since $\infty$ is contained in the Siegel disk. We claim that $\widetilde{\beta}$ is not the parabolic fixed point $0$. For otherwise, the drop-chain $\widetilde{\MC}^\infty$ will be mapped to another because $F_{\theta,\nu}$ has combinatorial rotation number $q/p\in (0,1)$ at $0$. Thus $\widetilde{\beta}$ must be a repelling fixed point of $F_{\theta,\nu}$ and the drop-chain $\widetilde{\MC}'^\infty$ lands at a prefixed point $\widetilde{\beta}'\neq \widetilde{\beta}$ with $F_{\theta,\nu}(\widetilde{\beta}')=\widetilde{\beta}$.

We also denote two drop-chains of $F_{\theta,\nu}$ growing from two immediate parabolic basins by:
\begin{equation}
\widetilde{\MC}^0:=\overline{\widetilde{U}^{(0)}_0\cup\widetilde{U}^{(1)}_1\cup \widetilde{U}^{(1)}_{11}\cup \cdots } \text{\quad and\quad} \widetilde{\MC}'^0:=\overline{\widetilde{U}^{(1)}_0\cup\widetilde{U}^{(1)}_2\cup \widetilde{U}^{(1)}_{21}\cup \cdots}\,,
\end{equation}
where $F_{\theta,\nu}(\widetilde{\MC}'^0)=\widetilde{\MC}^0$. By a similar argument as above, we conclude that the drop-chains $\widetilde{\MC}^0$ and $\widetilde{\MC}'^0$ also land at $\widetilde{\beta}$ and $\widetilde{\beta}'$ respectively, since $F_{\theta,\nu}$ has exactly three fixed points and only one of them is repelling.

Let $S_{\theta}$ and $S_{\nu}$ be the \textit{spines} of $f_{\theta}$ and $f_{\nu}$ defined in \eqref{equ:spine-theta} and \eqref{equ:spine-nu} respectively.
We define
\begin{equation}
\Sigma := \Sigma_{\theta}\cup \Sigma_{\nu}, \text{\quad where }\Sigma_{\theta}:=\phi_{\theta}(S_{\theta}) \text{ and }\Sigma_{\nu}:=\phi_{\nu}(S_{\nu}).
\end{equation}
By definition, the two simple arcs $\Sigma_{\theta}$ and $\Sigma_{\nu}$ only intersect at the two endpoints $\widetilde{\beta}$ and $\widetilde{\beta}'$. Therefore $\Sigma$ is a Jordan curve on $\EC$. See Figure \ref{Fig_spine}.

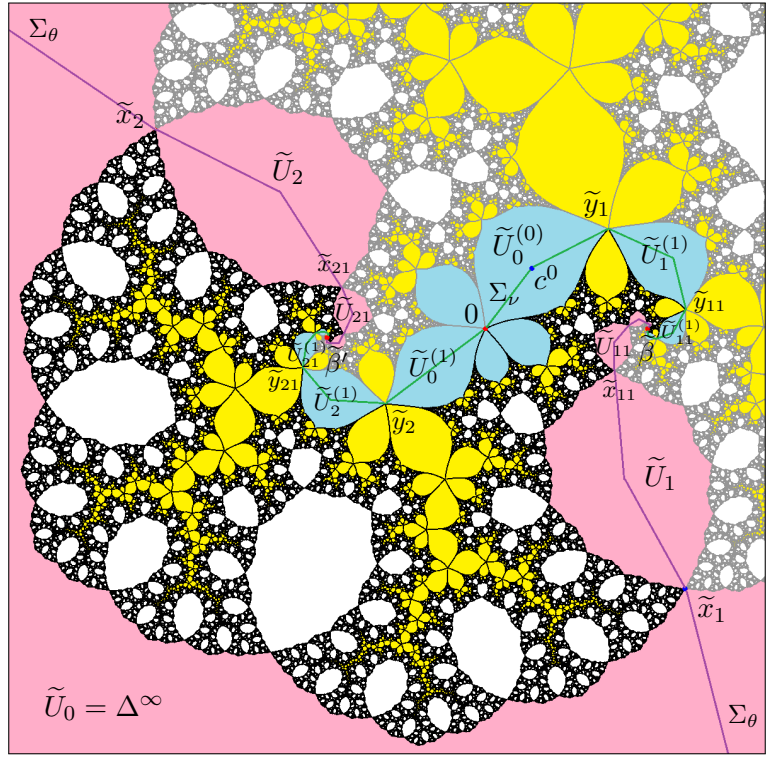
\begin{figure}[!htpb]
  \setlength{\unitlength}{1mm}
  \setlength{\fboxsep}{0pt}
  \centering
  \fbox{\includegraphics[width=0.7\textwidth]{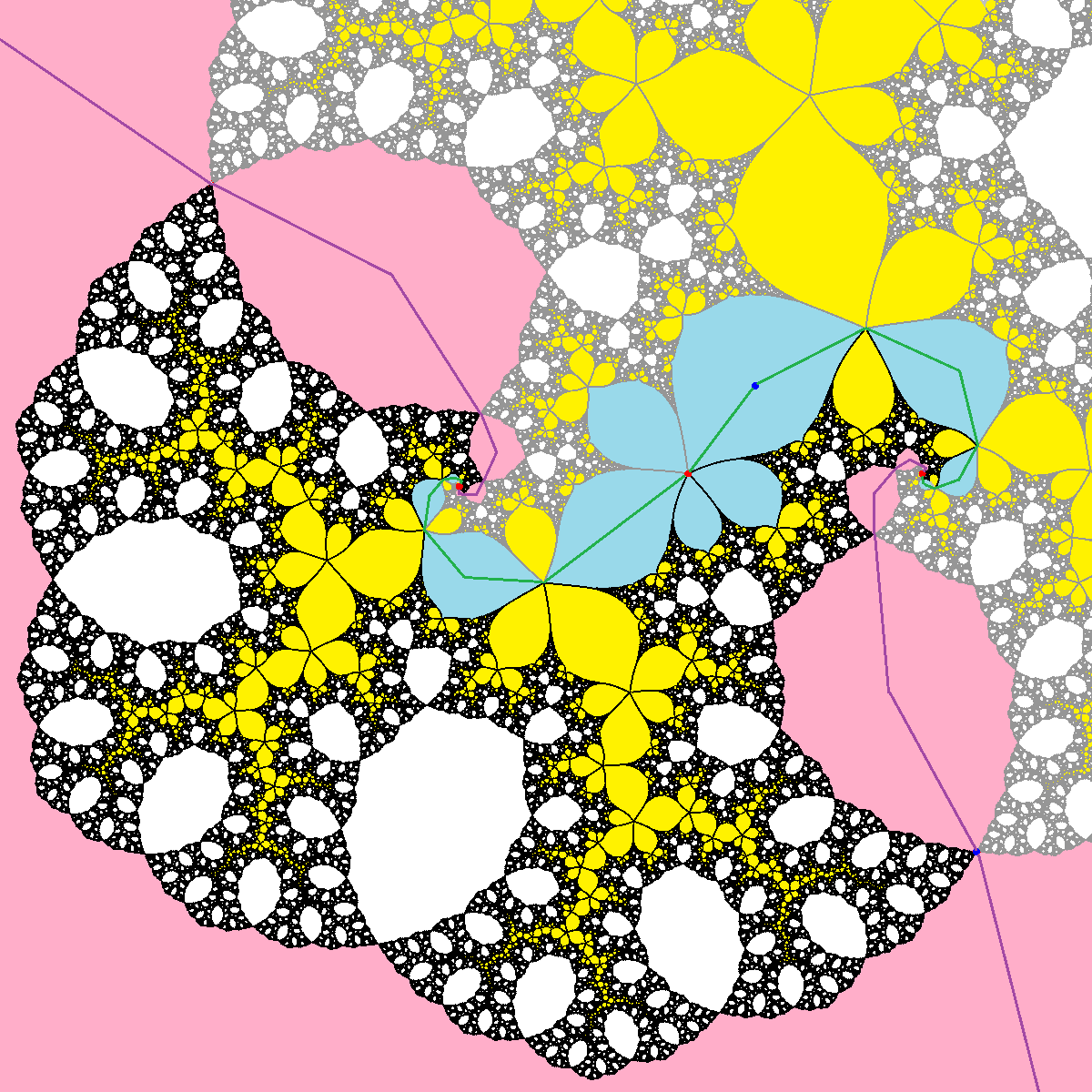}}
  \put(-95,5){$\widetilde{U}_0=\Delta^\infty$}
  \put(-9,18){$\widetilde{x}_1$}
  \put(-21.5,47.2){\footnotesize{$\widetilde{x}_{11}$}}
  \put(-16,35.5){$\widetilde{U}_1$}
  \put(-65,76){$\widetilde{U}_2$}
  \put(-85.5,83.5){$\widetilde{x}_2$}
  \put(-59,64){\footnotesize$\widetilde{x}_{21}$}
  \put(-56.7,57.5){\footnotesize{$\widetilde{U}_{21}$}}
  \put(-22.4,53){\footnotesize{$\widetilde{U}_{11}$}}
  \put(-16.8,52){\small{$\widetilde{\beta}$}}
  \put(-36,66){$\widetilde{U}_0^{(0)}$}
  \put(-16.5,65){\small{$\widetilde{U}_1^{(1)}$}}
  \put(-13.8,55){\tiny{$\widetilde{U}_{11}^{(1)}$}}
  \put(-47,50){\small{$\widetilde{U}_0^{(1)}$}}
  \put(-59.5,45.5){\footnotesize{$\widetilde{U}_2^{(1)}$}}
  \put(-63,52){\tiny{$\widetilde{U}_{21}^{(1)}$}}
  \put(-58,50.5){\small{$\widetilde{\beta}'$}}
  \put(-30.5,61.5){$c^0$}
  \put(-39.7,57){$0$}
  \put(-24,71.8){$\widetilde{y}_1$}
  \put(-49.2,43){\small{$\widetilde{y}_2$}}
  \put(-9.5,59){\footnotesize{$\widetilde{y}_{11}$}}
  \put(-65.5,48.5){\footnotesize{$\widetilde{y}_{21}$}}
  \put(-5,4){\small{$\Sigma_\theta$}}
  \put(-97,95){\small{$\Sigma_\theta$}}
  \put(-36.5,60){\small{$\Sigma_\nu$}}
  \caption{A decomposition of the dynamical plane of $F_{\theta,\nu}$. Some crucial points and Fatou components are marked. The union of spines $\Sigma_\theta\cup\Sigma_\nu$ is a Jordan curve cutting the Riemann sphere into two Jordan domains so that the itineraries of the points in $J(F_{\theta,\nu})$ can be defined.}
  \label{Fig_spine}
\end{figure}

Let $J_\theta^i$ and $J_\nu^i$ be defined in \S\ref{sec:defi} (see \eqref{equ:J-0-1}), where $i=0,1$.
Consider the four connected sets
\begin{equation}
\Gamma _{\theta}^i:=\phi_{\theta}(J_{\theta}^i) \text{\quad and\quad} \Gamma _{\nu}^i:=\phi_{\nu}(J_{\nu}^i),
\end{equation}
where $i=0,1$. Denote
\begin{equation}
\begin{split}
X:=&~\{\widetilde{\beta},\widetilde{\beta}',c^\infty=\widetilde{x}_1,\widetilde{x}_{11}, \widetilde{x}_{111},\cdots,\widetilde{x}_{2},\widetilde{x}_{21},\widetilde{x}_{211},\cdots\} \text{\quad and} \\
Y:=&~\{\widetilde{\beta},\widetilde{\beta}',0=\widetilde{y}_0,\widetilde{y}_1,\widetilde{y}_{11}, \widetilde{y}_{111},\cdots,\widetilde{y}_{2},\widetilde{y}_{21},\widetilde{y}_{211},\cdots\}.
\end{split}
\end{equation}
We claim that
\begin{equation}\label{equ:Gamma}
\Gamma _{\theta}^0\cap \Gamma _{\theta}^1=\Gamma _{\nu}^0 \cap \Gamma _{\nu}^1=X\cup Y
\end{equation}
and
\begin{equation}\label{itinerary}
\Gamma _{\theta}^0=\Gamma _{\nu}^1 \text{\quad and\quad}\ \Gamma _{\theta}^1=\Gamma _{\nu}^0.
\end{equation}

\medskip
By definition we have $X\subset \Gamma _{\theta}^0\cap \Gamma _{\theta}^1\subset X\cup Y$ and $Y\subset \Gamma _{\nu}^0\cap \Gamma _{\nu}^1\subset X\cup Y$. To obtain \eqref{equ:Gamma}, it suffices to prove that $X\subset \Gamma _{\nu}^0\cap \Gamma _{\nu}^1$ and $Y\subset \Gamma _{\theta}^0\cap \Gamma _{\theta}^1$.
Note that by Lemmas \ref{lem:lc-F} and \ref{limb}, these two inclusion relations follow from the fact that $\widetilde{x}_1=c^\infty$ is the landing point of exactly $2$ drop-chains growing from the immediate parabolic basins $\big\{\widetilde{U}_0^{(0)}, \widetilde{U}_0^{(1)}, \cdots, \widetilde{U}_0^{(p-1)}\big\}$ in $K^{0}(F_{\theta,\nu})$ and that $\widetilde{y}_0=0$ is the landing point of exactly $p$ drop-chains growing from $\widetilde{U}_0=\Delta^\infty$ in $K^{\infty}(F_{\theta,\nu})$. Therefore \eqref{equ:Gamma} is true.

Note that $\Sigma$ is a Jordan curve dividing the Riemann sphere into two disjoint Jordan domains $D_0$ and $D_1$, where $\Gamma_\theta^0\subset \overline{D}_0$ and $\Gamma_\theta^1\subset \overline{D}_1$. On the other hand, we have $\Gamma_\nu^0\subset \overline{D}_0$ or $\Gamma_\nu^1\subset \overline{D}_0$. Considering the direction on the Riemann sphere, we must have $\Gamma_\nu^0\subset \overline{D}_1$ and $\Gamma_\nu^1\subset \overline{D}_0$.
By \eqref{equ:Gamma} and
\begin{equation}\label{equ:J-decom}
\Gamma_\theta^0\cup\Gamma_\theta^1=\partial K^\infty(F_{\theta,\nu})=J(F_{\theta,\nu})=\partial K^0(F_{\theta,\nu})=\Gamma_\nu^0\cup\Gamma_\nu^1,
\end{equation}
it follows that \eqref{itinerary} holds.
See Figure \ref{Fig_spine}, where $\Gamma _{\theta}^0=\Gamma _{\nu}^1$ is colored black while $\Gamma _{\theta}^1=\Gamma _{\nu}^0$ is colored gray.

\medskip
\textbf{Step 2.}
Based on \eqref{itinerary} and \eqref{equ:J-decom}, we now ready to define the itinerary(ies) of every point $\zeta\in J(F_{\theta,\nu})$ with respect to $\Sigma_\theta$ and $\Sigma_\nu$ by looking at the points in the forward orbit of $\zeta$ and deciding whether they belong to $\Gamma _{\theta}^0=\Gamma _{\nu}^1$ or $\Gamma _{\theta}^1=\Gamma _{\nu}^0$.
Specifically, the \emph{itinerary(ies)} of $\zeta$ with respect to $\Sigma_{\theta}$ (resp. $\Sigma_\nu$) is defined as $\varepsilon=(\varepsilon_0,\varepsilon_1, \varepsilon_2,\cdots)$ (resp. $\delta=(\delta_0,\delta_1,\delta_2,\cdots)$), where $\varepsilon_n\in\{0,1\}$ (resp. $\delta_n\in\{0,1\}$) is determined by
\begin{equation}
\zeta_n:=F_{\theta,\nu}^{\circ n}(\zeta)\in \Gamma^{\varepsilon_n}_{\theta}~(\text{resp. } \zeta_n\in \Gamma^{\delta_n}_{\nu}), \text{\quad where }n\geqslant 0.
\end{equation}
Suppose $\zeta=\phi_{\theta}(z)=\phi_{\nu}(w)$, where $z\in J(f_\theta)$ and $w\in J(f_\nu)$. We consider four cases:

\medskip
\textit{Case 1}: Suppose $\zeta$ is not a preimage of $\widetilde{\beta}$, $\widetilde{x}_1$ or $\widetilde{y}_0$.
Then $\zeta$ has a unique $\Sigma_{\theta}$-itinerary $(\varepsilon_0,\varepsilon_1,\varepsilon_2,\cdots)$ and a unique $\Sigma_{\nu}$-itinerary $(\delta_0,\delta_1,\delta_2,\cdots)$, where $\delta_j=1-\varepsilon_j$ for all $j$.
Then $z$ has the $S_{\theta}$-itinerary $(\varepsilon_0,\varepsilon_1,\varepsilon_2,\cdots)$ and $w$ has the $S_{\nu}$-itinerary $(\delta_0,\delta_1,\delta_2,\cdots)$.
Setting $t=0.\varepsilon_0\varepsilon_1\varepsilon_2\cdots$ in the binary expansion, we have $z=\eta_{\theta}(t)$ and $w=\eta_{\nu}(-t)$.

\medskip
\textit{Case 2}: Suppose $\zeta$ is a preimage of $\widetilde{\beta}$. Then both $z$ and $w$ are preimages of the corresponding $\beta$-fixed points for $f_{\theta}$ and $f_{\nu}$.
If $\zeta_n=\widetilde{\beta}$ for the minimal integer $n\geqslant 0$ having such property, then the two $\Sigma_\theta$-itineraries of $\zeta$ coincide with those of $z$ with respect to $S_{\theta}$:
\begin{equation}
(\varepsilon_0,\varepsilon_1,\cdots,\varepsilon_{n-2},1,0,0,0,\cdots) \text{\quad and \quad}
(\varepsilon_0,\varepsilon_1,\cdots,\varepsilon_{n-2},0,1,1,1,\cdots).
\end{equation}
Both of which determine the same angle $t$ with $z = \eta_{\theta}(t)$.
The two $\Sigma_\nu$-itineraries of $\zeta$ coincide with those of $w$ with respect to $S_{\nu}$:
\begin{equation}
(\delta_0,\delta_1,\cdots,\delta_{n-2},0,1,1,1,\cdots)  \text{\quad and \quad}
(\delta_0,\delta_1,\cdots,\delta_{n-2},1,0,0,0,\cdots),
\end{equation}
where $\delta_j=1-\varepsilon_j$ for $0\leqslant j\leqslant n-2$.
Both of which determine the same angle $s$ with $w = \eta_{\nu}(s)$.
This implies that the binary digits of $t$ and $s$ are opposite and we have $z=\eta_{\theta}(t)$ and $w=\eta_{\nu}(-t)$.

\medskip
\textit{Case 3}: Suppose $\zeta$ is a preimage of $\widetilde{x}_1$. Since $\widetilde{x}_1$ and $\widetilde{y}_0$ have disjoint orbits under $F_{\theta,\nu}$, $\zeta$ cannot be a preimage of $\widetilde{y}_0$. This implies that $z$ is a preimage of the critical point $x_1$ of $f_{\theta}$. Denote one critical value of $F_{\theta,\nu}$ by $\widetilde{x}_0:=F_{\theta,\nu}(\widetilde{x}_1)$. Note that the $\Sigma_\theta$-itinerary and $\Sigma_\nu$-itinerary of $\widetilde{x}_0$ are unique.
Let $n\geqslant 0$ be the smallest integer such that $\zeta_n\in X\setminus\{\widetilde{\beta},\widetilde{\beta}'\}$. We consider $\zeta_n=\widetilde{x}_{11\cdots 1}\in F_{\theta,\nu}^{-(k-1)}(\widetilde{x}_1)$ for some $k\geqslant 1$ (The argument for $\zeta_n=\widetilde{x}_{21\cdots 1}$ is similar).
Therefore $\zeta$ has two $\Sigma_\theta$-itineraries which coincide with those of $z$ with respect to $S_{\theta}$:
\begin{equation}
\begin{split}
\varepsilon_\theta=&~(\varepsilon_0,\varepsilon_1,\cdots,\varepsilon_{n-1},\underbrace{0,\cdots,0}_{k \ \textrm{terms}},
\underbrace{\omega_1,\omega_2,\cdots}_{\Sigma_{\theta}\textrm{-itinerary of }\widetilde{x}_0}) \text{\quad and} \\
\varepsilon_\theta'=&~(\varepsilon_0,\varepsilon_1,\cdots,\varepsilon_{n-1},\underbrace{1,\cdots,1}_{k \ \textrm{terms}},
\underbrace{\omega_1,\omega_2,\cdots}_{\Sigma_{\theta}\textrm{-itinerary of } \widetilde{x}_0}).
\end{split}
\end{equation}

Note that $\zeta\in J(F_{\theta,\nu})=\partial K^0(F_{\theta,\nu})$ is the landing point of exactly two drop-chains in $K^0(F_{\theta,\nu})$ whose preimages under $\phi_\nu$ land at two distinct points $w,w'\in J(f_\nu)$ respectively (i.e., $\phi_\nu(w)=\phi_\nu(w')=\zeta$).
These two points are neither preimages of the $\beta$-fixed point nor preimages of the parabolic fixed point $0$ of $f_{\nu}$. So each of them has a unique $S_{\nu}$-itinerary:
\begin{equation}
\begin{split}
\delta_{\nu}=&~(\delta_0,\delta_1,\cdots,\delta_{n-1},\underbrace{1,\cdots,1}_{k \ \textrm{terms}} ,
\underbrace{\sigma_1,\sigma_2,\cdots}_{\Sigma_{\nu}\textrm{-itinerary of } \widetilde{x}_0}) \text{\quad and} \\
\delta_{\nu}'=&~(\delta_0,\delta_1,\cdots,\delta_{n-1},\underbrace{0,\cdots,0}_{k \ \textrm{terms}} ,
\underbrace{\sigma_1,\sigma_2,\cdots}_{\Sigma_{\nu}\textrm{-itinerary of } \widetilde{x}_0}),
\end{split}
\end{equation}
where $\delta_j=1-\varepsilon_j$ for $0\leqslant j\leqslant n-1$ and $\sigma_j=1-\omega_j$ for all $j\geqslant 1$.
Let $t$ and $t'$ be the angles with binary expansions corresponding to the itineraries $\varepsilon_\theta$ and $\varepsilon_\theta'$ respectively. Then $z=\eta_\theta(t)=\eta_\theta(t')$, $w=\eta_\nu(-t)$ and $w'=\eta_\nu(-t')$.

\medskip
\textit{Case 4}: Suppose $\zeta$ is a preimage of $\widetilde{y}_0$. Then $\zeta$ cannot be a preimage of $\widetilde{x}_1$. This implies that $w$ is a preimage of the parabolic fixed point $y_0=0$ of $f_{\nu}$. Let $n\geqslant 0$ be the smallest integer such that $\zeta_n\in Y\setminus\{\widetilde{\beta},\widetilde{\beta}'\}$. We consider $\zeta_n=\widetilde{y}_{11\cdots 1}\in F_{\theta,\nu}^{-(k-1)}(\widetilde{y}_1)$ for some $k\geqslant 1$ (The argument for $\zeta_n=\widetilde{y}_0$ and $\zeta_n=\widetilde{y}_{21\cdots 1}$ are similar).
Therefore $\zeta$ has $p$ $\Sigma_\nu$-itineraries $\{\delta_\nu^{(i)}:1\leqslant i\leqslant p\}$ which coincide with those of $w$ with respect to $S_{\nu}$ (see \S\ref{f_nu}):
\begin{equation}
\delta_\nu^{(i)}=
\left\{
\begin{array}{ll}
\Big(\delta_0,\delta_1,\cdots,\delta_{n-1},\underbrace{0,0,\cdots,0}_{k \ \textrm{times}}, \overline{\sigma_1^i,\sigma_2^i,\cdots,\sigma_p^i}\Big)  & \text{if } 1\leqslant i\leqslant q, \\
\Big(\delta_0,\delta_1,\cdots,\delta_{n-1},\underbrace{1,1,\cdots,1}_{k \ \textrm{times}}, \overline{\sigma_1^i,\sigma_2^i,\cdots,\sigma_p^i}\Big)  & \text{if } q+1\leqslant i\leqslant p,
\end{array}
\right.
\end{equation}
where $(\sigma_1^i,\sigma_2^i,\cdots,\sigma_p^i)$ is defined in \eqref{equ:delta-i-y0} for $1\leqslant i\leqslant p$.

Note that $\zeta\in J(F_{\theta,\nu})=\partial K^\infty(F_{\theta,\nu})$ is the landing point of exactly $p$ drop-chains in $K^\infty(F_{\theta,\nu})$ whose preimages under $\phi_\theta$ land at $p$ distinct points $z_1,\cdots,z_p\in J(f_\nu)$ respectively (i.e., $\phi_\theta(z_i)=\zeta$ for all $1\leqslant i\leqslant p$).
These $p$ points are neither preimages of the $\beta$-fixed point nor preimages of the critical point of $f_{\theta}$. So each of them has a unique $S_{\theta}$-itinerary:
\begin{equation}
\varepsilon_\theta^{(i)}=
\left\{
\begin{array}{ll}
\Big(\varepsilon_0,\varepsilon_1,\cdots,\varepsilon_{n-1},\underbrace{1,1,\cdots,1}_{k \ \textrm{times}}, \overline{\omega_1^i,\omega_2^i,\cdots,\omega_p^i}\Big)  & \text{if } 1\leqslant i\leqslant q, \\
\Big(\varepsilon_0,\varepsilon_1,\cdots,\varepsilon_{n-1},\underbrace{0,0,\cdots,0}_{k \ \textrm{times}}, \overline{\omega_1^i,\omega_2^i,\cdots,\omega_p^i}\Big)  & \text{if } q+1\leqslant i\leqslant p,
\end{array}
\right.
\end{equation}
where $\varepsilon_j=1-\delta_j$ for $0\leqslant j\leqslant n-1$ and $\omega_j^i=1-\sigma_j^i$ for all $1\leqslant i,j\leqslant p$.
Let $s_i\in[0,1]$ be the angle with binary expansion corresponding to the itinerary $\delta_\nu^{(i)}$ and denote $t_i=1-s_i$, where $1\leqslant i\leqslant p$.
Then $z_i=\eta_\theta(t_i)$ for each $1\leqslant i\leqslant p$ and $w=\eta_\nu(-t_i)$ for all $1\leqslant i\leqslant p$.

To sum up the above four cases, we have proved that if $\phi_{\theta}(z)=\phi_{\nu}(w)$, then there exists an angle $t\in \T$ such that $z=\eta_{\theta}(t)$ and $w=\eta_{\nu}(-t)$.

\medskip
\textbf{Step 3.}
Conversely, consider two points $z\in\partial K(f_{\theta})$ and $w\in\partial K(f_{\nu})$ such that $z=\eta_{\theta}(t)$ and $w=\eta_{\nu}(-t)=\eta_{\nu}(1-t)$ for some $t\in \T$.
Denote $\zeta=\phi_{\theta}(z)$ and $\xi=\phi_{\nu}(w)$. We need to prove that $\zeta=\xi$.
The binary expansion $(\varepsilon_0,\varepsilon_1,\varepsilon_2,\cdots)$ of the angle $t$ gives an itinerary of $z$ with respect to $S_{\theta}$. Since $1-t$ has the binary expansion $0.\delta_0\delta_1\delta_2\cdots$, the point $w$ has an itinerary $(\delta_0,\delta_1,\delta_2,\cdots)$ with respect to $S_{\nu}$, where $\delta_i=1-\varepsilon_i$ for all $i$.
By \eqref{itinerary}, it follows that $\zeta$ and $\xi$ have the same $\Sigma_\theta$-itinerary (and also the same $\Sigma_\nu$-itinerary).
By Propositions \ref{prop:YZ-Siegel}, \ref{prop:parabolic} and the above Step 2, any two points in $J(F_{\theta,\nu})$ with the same $\Sigma_\theta$ or $\Sigma_\nu$-itinerary must coincide. This implies that $\zeta=\xi$.

\medskip
\textbf{Step 4.}
It remains to prove that $\phi_\theta\big(K(f_\theta)\big)\cup \phi_\nu\big(K(f_\nu)\big)=\EC$. In fact, $\phi_\theta\big(K(f_\theta)\big)$ contains all the Fatou components which are eventually iterated to the Siegel disk $\Delta^\infty$ centered at infinity and $\phi_\nu\big(K(f_\nu)\big)$ contains all the Fatou components which are eventually iterated to the immediate parabolic basins attaching at the origin. Moreover, by \eqref{equ:J-decom}, $J(F_{\theta,\nu})$ is contained in $\phi_\theta\big(K(f_\theta)\big)$ and also in $\phi_\nu\big(K(f_\nu)\big)$. It follows that $\phi_\theta\big(K(f_\theta)\big)\cup \phi_\nu\big(K(f_\nu)\big)$ contains the Fatou set and the Julia set of $F_{\theta,\nu}$ and hence it must cover the whole Riemann sphere.
\end{proof}

\begin{proof}[Proof of Theorem \ref{thm:mateable}]
The uniqueness is immediate since any quadratic rational map having a fixed Siegel disk of rotation number $\theta$ and a parabolic fixed point with multiplier $e^{2\pi\ii \nu}$ is conformally conjugate to the normalized map $F_{\theta,\nu}$ defined in \eqref{equ:F-theta-nu}.

To prove the conformally mateable of $f_\theta$ and $f_\nu$,  by definition in the introduction, it suffices to prove the existence of continuous maps $\phi_{\theta}:K(f_{\theta}) \to \EC$ and $\phi_{\nu}:K(f_{\nu}) \to \EC$ with the following properties:
\begin{enumerate}
  \item $\phi_{\theta} \circ f_{\theta} =F_{\theta,\nu}\circ \phi_{\theta}$ and $\phi_{\nu} \circ f_{\nu} =F_{\theta,\nu}\circ \phi_{\nu}$;
  \item $\phi_\theta\big(K(f_\theta)\big)\cup \phi_\nu\big(K(f_\nu)\big)=\EC$;
  \item $\phi_{\theta}$ and  $\phi_{\nu}$ are conformal in the interiors of $K(f_{\theta})$ and $K(f_{\nu})$;
  \item $\phi_{\theta}(z)=\phi_{\nu}(w)$ if and only if $z$ and $w$ are ray equivalent.
\end{enumerate}

Let $\phi_{\theta}: K(f_{\theta})\to \EC$ and $\phi_{\nu}:K(f_{\nu})\to \EC$ be the continuous maps obtained in Lemma \ref{semiconjugacy}. Then the first three properties have been proved. It remains to verify Property (d). In fact, still by Lemma \ref{semiconjugacy}, $\phi_{\theta}(z)=\phi_{\nu}(w)$ if and only if $z=\eta_{\theta}(t)$ and $w=\eta_{\nu}(-t)$ for some $t\in\T$, where $\eta_{\theta}$ and $\eta_{\nu}$ are Carath\'{e}odory loops of $J(f_{\theta})$ and $J(f_{\nu})$ respectively. This is equivalent to that $z$ and $w$ are ray equivalent by definition.
\end{proof}

\begin{rmk}
Let $g_c(z)=z^2+c$ be a quadratic polynomial in the hyperbolic component of the Mandelbrot set attaching at the root $e^{2\pi \ii \nu}$, where $\nu=q/p\in(0,1)\cap\Q$. Then $g_c$ is a hyperbolic polynomial which is topologically conjugate to the parabolic polynomial $f_\nu(z)=e^{2\pi\ii \nu}z+z^2$. By a completely similar proof of Theorem \ref{thm:mateable} (In particular, Lemma \ref{main-lem-WYZZ} is needed), $g_c$ and $f_\theta$ are conformally mateable for any bounded type $\theta$ and the mating is unique up to conjugacy by a M\"{o}bius map.
For example, when $e^{2\pi \ii \nu}=-1$ (i.e., $q/p=1/2$), the basilica polynomial $g_c(z)=z^2-1$ and $f_\theta$ are conformally mateable (see \cite{Yan15}).
\end{rmk}

%----------------------------------------------------------------------------------------------------------------
\bibliographystyle{amsalpha}
\bibliography{E:/Latex-model/Ref1}

\end{document}